\documentclass
[
    a4paper,
    DIV=11,
    12pt,
    abstracton
]
{scrartcl}

\usepackage{fullpage,authblk,amsmath,amssymb,amsthm, mathtools}
\usepackage[table]{xcolor}
\usepackage{bm,bbm}
\usepackage{enumerate,paralist}
\usepackage{graphicx}
\usepackage{tikz}
\usetikzlibrary{arrows,decorations.pathmorphing,backgrounds,fit,positioning,shapes.symbols,chains}
\usetikzlibrary{arrows.meta} 
\usepackage{pgfplots}
\usepackage{multirow}
\usepackage[bf, format=plain]{caption} 
\usepackage{subcaption}
\usepackage{dsfont}
\usepackage{stackengine} 
\usepgfplotslibrary{fillbetween}
\usepackage{todonotes}

\usepackage[pdffitwindow=false,
            plainpages=false,
            pdfpagelabels=true,
            pdfpagemode=UseOutlines,
            pdfpagelayout=SinglePage,
            bookmarks=false,
            colorlinks=true,
            hyperfootnotes=false,
            linkcolor=blue,
            urlcolor=blue!30!black,
            citecolor=green!50!black]{hyperref}

\newtheorem{definition}{Definition}
\newtheorem{cor}[definition]{Corollary}
\newtheorem{lemma}[definition]{Lemma}
\newtheorem{proposition}[definition]{Proposition}

\newtheorem{remark}[definition]{Remark}
\newtheorem{theorem}[definition]{Theorem}

\newcommand{\R}{\mathbb{R}}
\newcommand{\ep}{\varepsilon}
\newcommand{\rv}[1]{\bm{#1}}

\newcommand{\set}[1]{\mathbb{#1}}

\newcommand{\cL}{\mathcal{L}}

\newcommand{\temp}{\mathrm{temp}}
\newcommand{\spat}{\mathrm{spat}}
\renewcommand{\epsilon}{\varepsilon}

\newcommand{\Goa}{G_{0,a}}
\newcommand{\incmap}{\mathbbm{i}}

\newcommand{\bbGamma}{\reflectbox{\rotatebox[origin=c]{180}{$\mathbb L$}}}

\newcommand{\new}[1]{{#1}}

\graphicspath{{./graphics/}}

\title{Detecting the birth and death of finite-time coherent sets}
\author[1]{Gary~Froyland}
\author[2]{P\'eter~Koltai}

\affil[1]{School of Mathematics and Statistics, University of New South Wales, Sydney NSW 2052, Australia.}
\affil[2]{Institute of Mathematics, Freie Universit\"at Berlin, 14195 Berlin, Germany.}

\date{}

\begin{document}

\maketitle

\begin{abstract}
Finite-time coherent sets (FTCSs) are distinguished regions of phase space that resist mixing with the surrounding space for some finite period of time;  physical manifestations include eddies and vortices in the ocean and atmosphere, respectively.
The boundaries of finite-time coherent sets are examples of Lagrangian coherent structures (LCSs).
The selection of the time duration over which FTCS and LCS computations are made in practice is crucial to their success.
If this time is longer than the lifetime of coherence of individual objects then existing methods will fail to detect the shorter-lived coherence.
It is of clear practical interest to determine the full lifetime of coherent objects, but in complicated practical situations, for example a field of ocean eddies with varying lifetimes, this is impossible with existing approaches.
Moreover, determining the timing of emergence and destruction of coherent sets is of significant scientific interest.
In this work we introduce new constructions to address these issues.
The key components are an inflated dynamic Laplace operator and the concept of semi-material FTCSs.
We make strong mathematical connections between the inflated dynamic Laplacian and the standard dynamic Laplacian \cite{Fro15}, showing that the latter arises as a limit of the former.
The spectrum and eigenfunctions of the inflated dynamic Laplacian directly provide information on the number, lifetimes, and evolution of coherent~sets.
\end{abstract}

\section{Introduction}
Lagrangian methods have proven to be powerful tools for elucidating the transport properties of non-autonomous and time-dependent dynamical systems.
Beginning with early approaches \cite{MMP84,romkedar90,pierrehumbert91,pierrehumbertyang93,pojehaller98} on identifying distinguished transport barriers, in the last fifteen years there has been focus on coherent behaviour.
This includes so-called Lagrangian coherent structures (LCSs), again targeting barriers to transport, and for which there are a large variety of approaches to their definition and identification:  a very small sample of this work is \cite{shaddenetal05,rypina2011investigating,allshouse2012detecting,budivsic2012geometry,haller2013coherent,ma2014differential,AlPe15}.
Finite-time coherent sets (FTCSs) \cite{FSM10,F13} are mobile regions in the phase space that resist mixing and provide a skeleton around which more complicated dynamics occurs.
Despite the moniker ``coherent'', these structures are often ephemeral:  they emerge, live for some time, and then decay and die.

Lagrangian methods are by their nature concerned with computations that follow trajectories over some
specified time interval of interest, rather than combining information across time at a fixed location in phase
space as in so-called Eulerian methods (e.g.\ using sea-surface height as a method of finding ocean eddies \cite{fu06,chelton07}).
By following trajectories, Lagrangian methods thus primarily detect structures that are coherent (according to various criteria) \emph{for the dominant part of the time interval under study}.
This reliance of Lagrangian coherent structure theory and numerics on objects being coherent
throughout (or throughout a large proportion of) the computed flow duration has remained essentially unchanged
since their introduction almost two decades ago.

The question of determining when coherent structures are born and when they die is largely unaddressed in the dynamical systems literature.
To quote MacMillan \emph{et al.} \cite{macmillan20}: \emph{``One major shortcoming of these (LCS) techniques, however, is the lack of an objective procedure for identifying time scales of interest, or an ability to characterise the lives, deaths, or age of coherent structures, especially when relevant flow time scales are larger than the time scales associated with coherence.''}
Several previous studies have investigated lifetimes in the context of ocean eddies, e.g.\ Froyland \emph{et al.} \cite{FrEtAl12} first identified a suitable timescale and then carried out a series of FTCS computations on time windows sliding forward in time, Andrade \emph{et al.} \cite{AnKaBV20} exhaustively search a discretised two-parameter space $(t,T)$ where $t$ is the initial time and $T$ is the flow duration, using these pairs as variable inputs to many separate LCS computations. El Aouni~\cite{ElA21} identifies the timespans of local rotational motion during each Lagrangian trajectory and then defines an eddy as those trajectories that are close at the beginning and the end of their respective timespans.

We build a theoretical framework to directly tackle this problem, using the successful spectral approach of the dynamic Laplace operator~\cite{Fro15,FrKw17} as a foundation.
We time-expand our spatial domain to create an inflated dynamic Laplace operator and allow ``time'' to become a diffusion process itself.
This enables us to relax the strict requirement that coherent sets or coherent structures be {exactly material} (i.e., follow flow trajectories), and leads to the notion of semi-material FTCS, which naturally allow coherent regions to appear and vanish over time.
Our constructions are interpreted from multiple viewpoints:  the spectrum of Laplace--Beltrami operators, the trajectories of stochastic differential equations, and properties of dynamic Riemannian metrics.

\subsection{Setting and Background}

We consider deterministic and stochastically perturbed time-dependent dynamical systems.
We work primarily in continuous time, however the ideas and constructions naturally cover the discrete-time case.
Let $v:[0,\tau]\times \mathbb{R}^d \to \R^d$ denote a \new{smooth} time-dependent velocity field, over a finite time duration $[0,\tau]$; for simplicity we assume that $v(t,\cdot)$ is divergence free for all $t\in [0,\tau]$.
Because we are interested in finite-time coherence we consider $\tau<\infty$; see e.g.\ \cite{FLS10,GiDa20} for techniques related to infinite-time coherence.
Denote by $M\subset\mathbb{R}^d$ a \new{$d$}-dimensional, connected, compact submanifold representing the phase space at time $t=0$.
\new{The operator-theoretic and geometric results in this paper are extendable to general compact Riemannian manifolds $M$, but to avoid obscuring the key contributions we work with flat $M$ embedded in Euclidean space.}
Denote by $\phi_t:M\to\phi_t(M)$ the flow map generated by $v$ from time 0 to time $t$, and set $M_t=\phi_t(M)$, $t\in[0,\tau]$.
\new{Note that $\phi_t$ diffeomorphically maps $M$ onto $M_t$.}
Denote by $\mathcal{P}_t:L^2(M)\to L^2(M_t)$ the transfer operator for $\phi_t$, using Lebesgue as the reference measure.
Finite-time coherent sets described by \cite{F13,DJM,FKS20} are constructed by adding small isotropic diffusion to the phase space dynamics.
One creates an operator $\mathcal{P}_{\epsilon,t}:L^2(M)\to L^2(M_t)$, which solves the Fokker--Planck equation
\begin{equation}
\label{eqn:fp}
\partial_t f = -\nabla\cdot(fv)+\frac{\epsilon^2}{2}\Delta f
\end{equation}
with homogeneous Neumann boundary conditions; that is, $f(t,\cdot):=\mathcal{P}_{\epsilon,t}f(0,\cdot)$ is a solution to~\eqref{eqn:fp}.
For $t>0$, the compact operator $\mathcal{P}_{\epsilon,t}$ has a singular value 1 with unit multiplicity, and there is a gap to the next singular value.

Coherent sets over the interval $[0,\tau]$ are described by level sets of leading singular vectors of $\mathcal{P}_{\epsilon,\tau}$;  in particular at time $0$, one considers level sets of the eigenvectors $f$ of $\mathcal{P}_{\epsilon,\tau}^*\mathcal{P}_{\epsilon,\tau}$ and at time $t$ level sets of~$\mathcal{P}_{\ep,t}f$, see \cite{F13,FPG14}.
For small $\epsilon$, these coherent sets are approximately \emph{material} under the purely advective dynamics of $\phi_t$, meaning that if $A_t$ is a coherent set at time~$t$, then $A_t \approx \phi_t( \phi_s^{-1}(A_{s}) )$.
The family becomes more material as $\epsilon$ is decreased~\cite{FPG14}.

In the limit as $\epsilon\to 0$, for fixed $t$, $\mathcal{P}_{\epsilon,t}^*\mathcal{P}_{\epsilon,t}$ approaches the identity operator, and one can take a singular limit to obtain a \emph{dynamic Laplace operator} \cite{Fro15,FrKw17,KaSch21}, denoted~$\Delta^D$.
In this purely deterministic setting, coherent sets at time $t=0$ are identified as level sets of dominant eigenfunctions of~$\Delta^D$ \cite{Fro15}.
These level sets are \emph{exactly material} under the flow $\phi_t$, and represent material sets that stay most coherent under vanishing diffusion.
In this work we relax the strict materiality requirement while maintaining purely advective dynamics on $\phi_t(M)$, $t\in [0,\tau]$.
This will enable the identification of coherent sets that appear and disappear within some larger time window~$[0,\tau]$.

Let $\phi_t^*e$ denote the pullback of the Euclidean metric $e$ from the manifold $\phi_t(M)$ to the manifold~$M$.
In the divergence-free setting considered in this paper, the dynamic Laplacian $\Delta^D$ on $L^2(M)$ has the form~\cite{Fro15}
\begin{equation}
\label{DL}
\Delta^D=\frac{1}{\tau}\int_0^\tau \Delta_{\phi_t^*e}\ dt = \frac{1}{\tau}\int_0^\tau \Delta_{g_t}\ dt,
\end{equation}
where from now on we use the notation $g_t:=\phi_t^*e$.
This is an average of Laplace--Beltrami operators for the Riemannian manifolds $(M,g_t)$,~$t\in[0,\tau]$.

\subsection{Relaxing materiality and a new key object}
In order to relax materiality, we time-expand the phase space, giving each manifold $(M,g_t)$ its own $t$-fibre.
Topologically, this time-expanded domain is simply
\begin{equation}
\label{eq:M0}
\mathbb{M}_0:=\bigcup_{t\in [0,\tau]}\{t\}\times M=[0,\tau]\times M.
\end{equation}
We define the ``co-evolved'' spacetime manifold by
\begin{equation}
\label{eq:M1}
\set{M}_1:=\bigcup_{t\in [0,\tau]}\{ t \}\times\phi_t(M).
\end{equation}
In $\mathbb{M}_0$, a curve corresponding to the trajectory $\{\phi_t(x):0\le t\le \tau\}\subset M$ is simply the line $\{(t,x):0\le t\le \tau\}$.
The canonical mapping from $\set{M}_0$ to the trajectory manifold $\set{M}_1$, associating initial conditions with trajectories, is
\begin{equation}
	\label{eq:Phi}
	\Phi: \set{M}_0 \to \set{M}_1,\quad (t,x) \mapsto (t, \phi_t(x)).
\end{equation}
Figure~\ref{fig:topsum} illustrates these constructions in two situations:  there is a coherent family of sets $\phi_t(A)$ throughout the whole time interval $[0,\tau]$ (upper row) and a coherent family of sets $\phi_t(A_t)$ that is born at time $\tau_1$ and extinguished at time $\tau_2$ (lower row).
\begin{figure}[htbp]
  \centering
  \includegraphics[width = 1\textwidth]{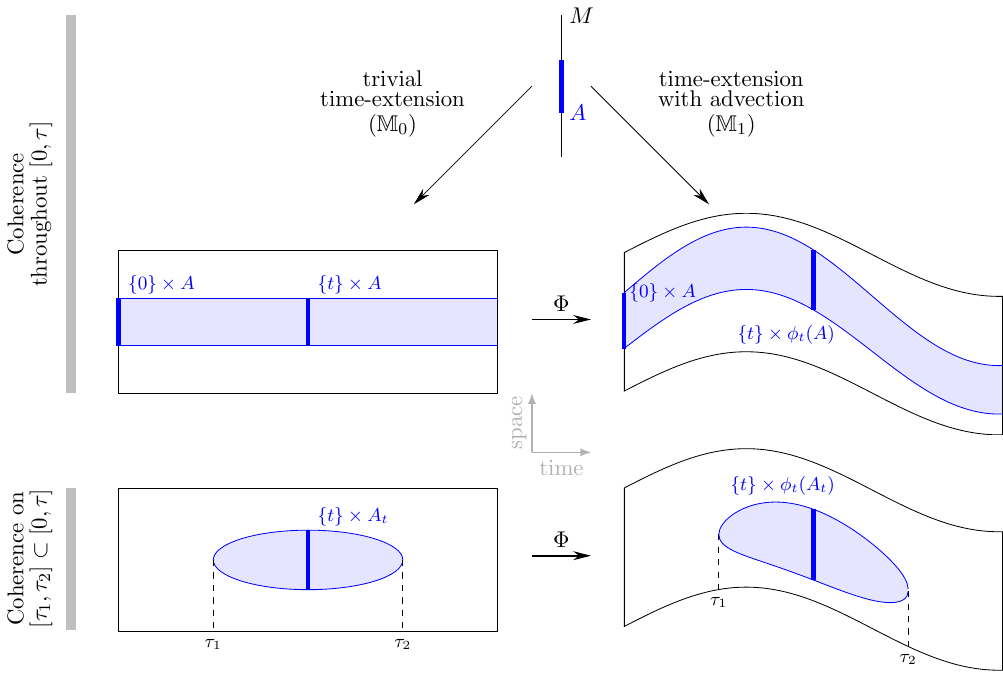}

  \caption{\emph{Time-expanded constructions and fully vs partially present coherence}.
The top of the diagram shows the situation where there is (for simplicity of presentation) a single coherent set $A\subset M$ present throughout the time interval~$[0,\tau]$,  shown as a dark vertical blue line at the very top of the figure. Upper left: by trivial copying in time we obtain the pale blue horizontal strip $[0,\tau] \times A\subset\mathbb{M}_0$. Upper right: by evolving $A$ forward in time with the dynamics from time 0 to time $\tau$ we trace out the pale blue set $\bigcup_{t\in[0,\tau]}\{t\}\times\phi_t(A)\subset\mathbb{M}_1$.
  The lower row of the diagram concerns the situation where there is a coherent set present \emph{only for part of the time interval}, say a subinterval $[\tau_1,\tau_2]\subset [0,\tau]$.  Lower right: following the dynamics, a coherent set appears at $\tau_1$ from a small expanding core, exists for a while, and then shrinks and dissipates completely at~$\tau_2$.  Lower left:  We pull back the lower right image to time $t=0$ using the inverse of~$\Phi$.}
  \label{fig:topsum}
\end{figure}

By considering the Euclidean metric on each $\phi_t(M)$, the above constructions naturally suggest a metric on~$\set{M}_0$.
At a point $(t,x)\in\mathbb{M}_0$, we define local distances by the metric with coordinate representation
\begin{equation}
\label{metriceqn}
\left(
  \begin{array}{cc}
    1 & 0 \\
    0 & D\phi_t(x)^\top D\phi_t(x) \\
  \end{array}
\right),
\end{equation}
where the lower right block is the local matrix representation of $g_t=\phi_t^*e$ at~$x\in M$.
We denote by $G_0$ the metric on $\mathbb{M}_0$ given pointwise by \eqref{metriceqn}.

Our key new object is the \emph{inflated dynamic Laplace operator}, which we briefly now describe, with further details to follow.
Consider the Laplace--Beltrami operator $\Delta_{G_0}:L^2(\mathbb{M}_0,G_0)\to L^2(\mathbb{M}_0,G_0)$;  for the moment we delay the discussion of boundary conditions.
Because of the time-fibered structure~\eqref{metriceqn} of the metric $G_0$, we may write
$
\Delta_{G_0} F(t,\cdot) = \partial_{tt} F(t,\cdot)+\Delta_{g_t} \, F(t,\cdot).
$
In fact, we will consider a family of Laplace--Beltrami operators with a parameter~$a>0$,
\begin{equation}
\label{eq:LBab}
\Delta_{\Goa} F(t,\cdot) = a^2  \partial_{tt} F(t,\cdot)+ \Delta_{g_t} \, F(t,\cdot).
\end{equation}
We will show that these operators interpolate between the dynamic Laplacian $\Delta^D$ (in the $a\to\infty$ limit), whose level sets of eigenfunctions form \emph{exactly material} families of coherent sets, and a purely non-dynamic Laplace--Beltrami operator (when $a=0$), whose level sets of eigenfunctions need not have any material properties.
Thus, the parameter $a$ interpolates the material requirement from strictly material to non-material.
By selecting an appropriate $a$, the eigenfunctions of $\Delta_{\Goa}$ will identify (i) time intervals of strong and weak mixing, and (ii) coherent sets within the time intervals of weak mixing.

We note that time-expansion has been used in the context of transfer operators to find coherent sets of periodic \cite{FrKo17}, finite-time aperiodic \cite{FKS20} and aperiodic \cite{GiDa20} flows.
There are several differences between these works and our current constructions, including, but not limited to:  (i) we do not require coherent behaviour throughout the flow duration being considered, (ii) we work with Laplace--Beltrami operators instead of transfer operators, (iii) our analysis is carried out on $\set{M}_0$, rather than the co-evolved time-expanded manifold, (iv) we consider time as a diffusion process, instead of it increasing with a constant speed.
Moreover, in addition to the stochastic trajectory and transfer operator interpretations in \cite{FrKo17,FKS20}, we also provide a differential-geometric perspective.
Other work arising from the dynamic Laplacian includes \cite{KK,KaSch21,schillingetal21}, where the emphasis is on the time-averaged processes generated by the dynamic Laplacian in the initial time slice on $M$.

Laplace-spectral approaches \cite{Gomez2013diffusion} to analysing multilayer networks \cite{de2013mathematical,boccaletti2014structure,kivela2014multilayer} share some structural similarities to~\eqref{eq:LBab}, where diffusion occurs both within and across network layers.
In the particular case of a two-layer network, \cite{Gomez2013diffusion} study the dependence of the spectrum of a ``supra-Laplacian'' on the coupling strength.
This is formally similar to the construction of \cite[equation (7)]{FackEtAl19}, which considers multiple layers.
Our results concerning the behaviour of the spectrum and eigenfunctions of $\Delta_{\Goa}$ with varying diffusion strength $a$ should carry over to multilayer networks to describe the analogous behaviour with varying interlayer coupling strength and connect to graph-based versions of the dynamic Laplace operator~\cite{FrKw15}.

Finally, the birth and death of a coherent set represents a structural change in the dynamics.
An unrelated type of structural change for almost-invariant and coherent sets is the crossing of eigenvalues of the transfer operator, which is sometimes, but not always, associated with the bifurcation of the sets.
Case studies that consider these types of perturbations include \cite{Jungeetal04} (autonomous), \cite{groveretal12} (periodic), and \cite{blachutGT20,ndouretal21} (non-autonomous).
However, we note that bifurcations are not particularly prevalent, for example if a spectral value is currently isolated from other spectral values, \cite{AFJ21} shows that the eigenvalues and eigenfunctions of the dynamic Laplacian varying differentiably for small perturbations of the flow duration, or of the underlying dynamics.

\paragraph{Outline}
In section~\ref{sect:diff-time} we provide independent motivation for the above geometric construction using stochastic trajectories, and then connect this to the geometry on $\set{M}_1$ and $\set{M}_0$ in sections \ref{ssec:timediffusedM1} and~\ref{ssec:pullback_SDE}, respectively.
In section~\ref{sec:dynLapConns} we derive results interpolating between material and non-material coherence.
Section~\ref{ssec:dynLap} briefly recaps the dynamic Laplace operator and section~\ref{ssec:avging} shows that the dynamic Laplacian~$\Delta^D$ arises from the inflated dynamic Laplacian $\Delta_{\Goa}$ in the $a\to\infty$ limit, by invoking the theory of averaging.
The dynamic theory of Cheeger and Sobolev constants for $\Delta^D$ is linked to classical notions of these constants on the Riemannian manifold $\set{M}_0$ in section~\ref{ssec:cheeger}.
We characterise the behavior of the spectrum of the inflated dynamic Laplacian $\Delta_{\Goa}$ in section~\ref{ssec:eigenvalues}, connecting it with the spectrum of the dynamic Laplacian.
In section~\ref{ssec:practical} this theoretical information is synthesised into a practical approach to find coherent sets with lifetimes shorter than the full flow duration.
A reduced PDE corresponding to (\ref{eq:LBab}) where all spatial information is collapsed is derived in section~\ref{ssec:surr_deriv}, enabling a comparison of instantaneous coherent set decay at time $t$ with average decay across~$[0,\tau]$.
In idealised coherent and mixing regimes, section~\ref{ssec:surrogate} provides fine detail on the behaviour of the time-fibre norms of the eigenfunctions of the inflated dynamic Laplacian.
We develop a trajectory-based numerical scheme based on a specialised finite element method in section~\ref{sec:FEMDL}, and illustrate our theory via an example in section~\ref{sec:examples}.

\section{From diffusion to geometry}

Recall that deterministic trajectories are represented in $\mathbb{M}_0$ as straight lines parallel to the time axis, and they can be uniquely parametrized by their initial conditions~$(0,x_0) \in \{0\}\times M$.
In this section we interpret the paths generated by the SDE associated with the inflated dynamic Laplace operator $\Delta_{\Goa}$ in~\eqref{eq:LBab} on $\set{M}_0$.
These paths are driven by a pure diffusion process and in section \ref{ssec:pullback_SDE} we will show that these paths \emph{independently} jump \emph{along} and \emph{between} trajectories $\{(t,x_0)\,\vert\, t\in [0,\tau]\}$ of the deterministic flow in~$\set{M}_0$.
In a dynamical sense it is natural to first consider the dynamics on its ``true'', co-evolved space, so we begin our analysis with a process on~$\set{M}_1$ in sections \ref{sect:diff-time} and \ref{ssec:timediffusedM1}, and then pull these constructions back with $\Phi$ to $\set{M}_0$, in section~\ref{ssec:pullback_SDE}.

\subsection{Trajectory-based view}
\label{sect:diff-time}

Recall from the introduction that coherent sets as defined in \cite{F13,DJM} rely on the addition of diffusion to the deterministic dynamics;  this is so that large boundaries are penalised through greater diffusive mixing. The process we will consider, which gives rise to the Fokker--Planck (Kolmogorov forward) equation~\eqref{eqn:fp}, is a time-inhomogeneous It{\^o} diffusion process governed by the SDE
\begin{equation} \label{eq:SDE}
d\rv{x}_t = v(t,\rv{x}_t)dt +  \ep \, d\rv{b}_t,
\end{equation}
where $v$ is a smooth $d$-dimensional divergence-free velocity field, $ \ep  \ge 0$, and $\rv{b}_t$ is a $d$-dimensional standard Wiener process.
We assume $\rv{x}_0$ to be uniformly distributed, thus the process is stationary.
Let the set of initial conditions, $M$, be a \new{compact}, smooth, flat, $d$-dimensional manifold (usually a subset of $\R^d$ with smooth boundary, a cylinder, or a torus) equipped with a (Euclidean) metric.
If $M$ has a boundary, \eqref{eq:SDE} is equipped with reflecting boundary conditions on $\smash{\bigcup_{t\in[0,\tau]} \{t\}\times \partial \left(\phi_t(M)\right)}$ in space-time; maintaining stationarity. Note that the domain for the SDE co-evolves with the deterministic flow driven by the velocity field~$v$, thus~\eqref{eq:SDE} lives on~$\set{M}_1$.

Next we will view the temporal component $t$ of \eqref{eq:SDE} as an independent variable $\theta$ undergoing diffusion. Since this diffusion can move in both directions along a line, the time component of~\eqref{eq:SDE} will also evolve in positive and in negative directions.
We assume that the temporal parameter $\theta$ performs a Brownian diffusion with reflecting boundary conditions on $[0,\tau]$ and constant diffusion coefficient~$ a >0$:
\begin{equation} \label{eq:theta_dyn}
d\rv{\theta}_t =  a \, d\rv{w}_t,
\end{equation}
where $\rv{w}_t$ is a standard one-dimensional Wiener process independent of~$\rv{b}_t$.
Equation \eqref{eq:SDE} now becomes
\begin{equation} \label{eq:fwdbwd_augAlt}
d\rv{x}_t =  a \, v(\rv{\theta}_t,\rv{x}_t)\circ d\rv{w}_t +  \ep \, d\rv{b}_t,
\end{equation}
where we now assumed the Stratonovich form, and will explain the reason for this below. We call this the \emph{time-diffused process}.
An equivalent form of the system of equations \eqref{eq:theta_dyn} and \eqref{eq:fwdbwd_augAlt} can be written in~$\R^{d+1}$ as
\begin{equation}
\label{eq:difftimeproc}
d\rv{X}_t = \underbrace{\begin{pmatrix}
 a \, & 0^\top\\
 a \,\, v(\rv{\theta}_t,\rv{x}_t) &  \ep \, \mathrm{Id}_{d\times d}
\end{pmatrix}}_{=: \sigma( \rv{X}_t )} \circ\, d\rv{B}_t,
\end{equation}
with
\[
\rv{X}_t = \begin{pmatrix}
\rv{\theta}_t\\
\rv{x}_t
\end{pmatrix},
\
\rv{B}_t = \begin{pmatrix}
\rv{w}_t\\ \rv{b}_t
\end{pmatrix}.
\]
Note that we use upper case letters to denote the time-augmented version of a variable.
The main reason for the the Stratonovich interpretation in~\eqref{eq:fwdbwd_augAlt} and~\eqref{eq:difftimeproc} is satisfaction of the chain rule, which is important in the first of the following two situations.
\begin{enumerate}
\item
\textbf{Spatial deterministic limit: $\epsilon\to 0$.}
When $\ep=0$, there is no spatial noise, and so we would like the paths of \eqref{eq:theta_dyn}--\eqref{eq:fwdbwd_augAlt} to \emph{stay on trajectories $\{(\theta,x_{\theta})\,\vert\, \theta\in [0,\tau]\}$ of the deterministic ODE $\dot{x}_t = v(t,x_t)$.}
This is only guaranteed in the Stratonovich case~\cite[Proposition~1.2.8]{Hsu02}.
For a deterministic trajectory $(x_r)_{r\in [0,\tau]}$ parametrized by the random time parameter $r = \rv{\theta}_t$, i.e., $\rv{z}_t := x_{\rv{\theta}_t}$, we have
\[
d\rv{z}_t = \frac{d}{dr} x_r \big\vert_{r=\rv{\theta}_t} \circ d\rv{\theta}_t=  v(\rv{\theta}_t, \rv{z}_t ) \circ (  a \, d\rv{w}_t),
\]
exactly~\eqref{eq:fwdbwd_augAlt} for~$\ep=0$.
Since the solution of \eqref{eq:fwdbwd_augAlt} is unique (for a fixed realization of the process~$\rv{w}_t$), $\rv{z}_t = x_{\rv{\theta}_t}$ is this solution, and it clearly evolves along trajectories of the deterministic ODE.
\item \textbf{High temporal diffusion limit: $a\to\infty$.}
A similar situation occurs if instead we fix $\epsilon>0$ and let the temporal diffusion coefficient $a$ increase to infinity.
     In the Lagrangian frame $\set{M}_0$, the drift $v$ in (\ref{eq:fwdbwd_augAlt}) is zero and $\Phi$-pullbacks of the paths of~$\rv{X}_t$ become increasingly aligned with the time axis because the stochastic variation in the temporal component (controlled by $a$) dominates the stochastic variation in space (controlled by $\epsilon$).
     In $\set{M}_1$, paths of~$\rv{X}_t$ therefore align with trajectories $\{(\theta,x_{\theta})\,\vert\, \theta\in [0,\tau]\}$ with overwhelming probability in the $a\to\infty$ limit.
\end{enumerate}
We will return to these situations in section~\ref{ssec:avging}, showing that they are effectively equivalent.
In the situation where $\epsilon$ is large relative to $a$, the stochastic trajectories of \eqref{eq:theta_dyn}--\eqref{eq:fwdbwd_augAlt} may significantly deviate from the deterministic trajectories, and this deviation will be crucial for relaxing the strictly material nature of FTCSs.

\subsection{The time-diffused process and Brownian motion on $\set{M}_1$}
\label{ssec:timediffusedM1}

So far the processes $\rv{b}_t$ and $\rv{w}_t$ have been standard Wiener processes with respect to the Euclidean metric on $\phi_t(M)$, $0\le t\le \tau$.
Next we wish to interpret certain processes as (standard) Brownian motions with respect to a suitable (Riemannian) metric. For clarity, in these cases the metric (or the entire Riemannian manifold) will be explicitly stated. For this interpretation, we will use the fact that standard Brownian motion on a Riemannian manifold is (up to equivalence in law) given by its generator, which is the \emph{one half} Laplace--Beltrami operator on the manifold~\cite[Chapter~3]{Hsu02}.

The process we would like to interpret as Brownian motion on a Riemannian manifold is $\rv{X}_t$, governed by the SDE~\eqref{eq:difftimeproc}.
This augmented process lives on the augmented spacetime manifold~$\set{M}_1$, defined by \eqref{eq:M1}, and so we need to find a metric on $\set{M}_1$ whose Laplace--Beltrami operator is the generator of this process.
One reason for seeking this connection is to link our construction to geometric characterisations of coherent sets~\cite{Fro15}.
A second reason is to further develop our formalism for coherent sets when the ``materialness requirement'' that all previous work relied on is relaxed.
In our derivation, we will rely on linking a SDE---in law---to a Riemannian metric via the associated Fokker--Planck (Kolmogorov forward) equation and the Laplace--Beltrami operator. Here, the Fokker--Planck equation identifies a SDE uniquely only in law, not pathwise.

We recall that we assumed $v$ to be divergence-free. We expect to be able to derive similar results to those below for the non-divergence-free case as well, however, at the cost of more technical exposition that would obscure the main points.
With this assumption, using the definition of $\sigma$ in \eqref{eq:difftimeproc}, it is immediate that $\nabla\cdot \sigma^\top  := \nabla_{(\theta,x)}\cdot \sigma^\top  \equiv 0$, where the divergence operator is applied to a matrix row-wise.
We recall from~\cite[p.~63]{Pav14} that the It{\^o} form of a Stratonovich SDE $d\rv{z}_t = b\,dt + \sigma \circ d\rv{w}_t$, where the drift $b$ and the diffusion coefficient $\sigma$ depend on time and space as well, is given by
	$d\rv{z}_t = (b+ \frac{1}{2}(\nabla \cdot (\sigma \sigma^\top ) - \sigma \nabla\cdot\sigma^\top ))\,dt + \sigma\,d\rv{w}_t.$
Thus, the It{\^o} form of \eqref{eq:difftimeproc} is
\begin{equation}
\label{ito}
d\rv{X}_t = \frac12 \nabla\cdot\Sigma_1 (\rv{X}_t)dt + \sigma(\rv{X}_t)\, d\rv{B}_t, \quad\text{where }\Sigma_1 =\sigma\sigma^\top ,
\end{equation}
and the drift is $\frac12 \nabla\cdot\Sigma_1$.
From \cite[pp.~69 \& 71]{Pav14} we recall that an It{\^o} SDE $d\rv{z}_t = b\,dt + \sigma\,d\rv{w}_t$ has the Fokker--Planck operator (forward Kolmogorov generator, acting on the usual Sobolev space $H^2$, or in a variational characterisation described by a bilinear form on~$H^1$)
\begin{equation}
	\label{foot:ItoFPE}
	\cL^* f = \nabla\cdot\left( -bf + \tfrac12 \nabla\cdot (\sigma\sigma^\top  f) \right) = \nabla\cdot\left( \left(-b + \tfrac12 \nabla\cdot (\sigma\sigma^\top ) \right) f + \tfrac12 \sigma\sigma^\top  \nabla f \right).
\end{equation}
Thus, the Fokker--Planck operator corresponding to~\eqref{ito}, $\mathcal{L}_1^*:H^2(\set{M}_1)\to L^2(\set{M}_1)$,
becomes
\begin{equation} \label{eq:FP_Ito1}
\begin{aligned}
\cL_1^* F &= \nabla\cdot \left( -\tfrac12 \nabla\cdot\Sigma_1\, F + \tfrac12 \nabla\cdot \left(\Sigma_1 F\right)\right) =
\tfrac12 \nabla\cdot \left( -  \nabla\cdot\Sigma_1\, F +  \nabla\cdot\Sigma_1\, F + \Sigma_1\nabla F\right)\\
&= \tfrac12 \nabla\cdot \left( \Sigma_1\, \nabla F\right),
\end{aligned}
\end{equation}
with homogeneous Neumann boundary conditions both in space and time, corresponding to the reflecting boundary conditions for the SDE.
The subscripts in $\Sigma_1$ and $\cL_1$ indicate that these objects are naturally connected to~$\set{M}_1$, and we use $\cL_1^*$ to denote the Fokker--Planck operator, despite the above operator being self-adjoint.

\begin{proposition}
\label{prop:M1metric}
We have $\cL_1^* = \frac12 \Delta_{G_1}$, where $\Delta_{G_1}$ is the Laplace--Beltrami operator corresponding to the metric $G_1$ given by the metric tensor
\begin{equation}
	\label{eq:Sigma1}
	G_1(\theta,x) = \Sigma_1(\theta,x)^{-1} = \begin{pmatrix}
\frac{1}{a^2} + \frac{1}{\ep^2}v(\theta,x)^\top v(\theta,x)  & -\frac{1}{\ep^2} v(\theta,x)^\top \\[6pt] -\frac{1}{\ep^2} v(\theta,x) & \frac{1}{\ep^2} \mathrm{Id}
\end{pmatrix}
\end{equation}
in the Euclidean coordinates $(\theta,x)$ on $\set{M}_1$.
\end{proposition}

\begin{proof}
For a general metric $g$ the Laplace--Beltrami operator is given by
\begin{equation}
\label{eq:LB}
\Delta_{g} f = \frac{1}{\sqrt{|g|}} \nabla\cdot \left(\sqrt{|g|}g^{-1} \nabla f \right),
\end{equation}
where $|g|$ denotes the absolute value of the determinant of the metric tensor, the latter also denoted by~$g$. In our particular situation, from~\eqref{eq:difftimeproc} we have that $\det \sigma =  a \ep^d$, and hence $\det \Sigma_1 = ( \det \sigma)^2 =  a^2 \ep^{2d}$. Since the determinant of $\Sigma_1$ is constant, the volume distortion $\sqrt{|G_1|}$ cancels out and we obtain $\smash{ \Delta_{G_1} F = \nabla\cdot (\Sigma_1\, \nabla F) }$ for a sufficiently smooth function~$F:\set{M}_1 \to \R$. A comparison with \eqref{eq:FP_Ito1} implies the claim. The explicit form can be seen from $\Sigma_1^{-1} = \sigma^{-\top}\sigma^{-1}$, (\ref{eq:difftimeproc}), and the identity
\begin{equation}
\label{eq:invMat}
\begin{pmatrix}
\mathrm{a} & 0^\top \\ \mathrm{v} & \mathrm{J}
\end{pmatrix}^{-1} = \begin{pmatrix}
\mathrm{a}^{-1} & 0^\top \\ -\mathrm{a}^{-1}\mathrm{J}^{-1}\mathrm{v} & \mathrm{J}^{-1}
\end{pmatrix}
\quad\text{for } \mathrm{a}\in\R\setminus \{0\},\mathrm{v}\in\R^d,\mathrm{J}\in\R^{d\times d} \text{ invertible}.
\end{equation}
\end{proof}

The following is an immediate consequence of Proposition \ref{prop:M1metric}. Note that we interchangeably use the metric and its metric tensor expressed in Euclidean coordinates.
\begin{cor} \label{cor:generator1}
The generator of the process $\rv{X}_t$ governed by  \eqref{eq:difftimeproc} on $\mathbb{M}_1$ is given by $\smash{ \cL_1^* = \frac12 \Delta_{\Sigma_1^{-1}} }$. Hence, $\rv{X}_t$ is equivalent in law to the standard Brownian motion on the Riemannian manifold~$(\set{M}_1,G_1)$, where $G_1 = \Sigma_1^{-1}$.
\end{cor}
The following diagram summarises the one-to-one relationships we have used between objects associated with $\set{M}_1$, where the first one is merely in law:

\medskip
{

\centering

\begin{tikzpicture}
[auto,font=\footnotesize,
object/.style={
rectangle, thick, draw, fill=black!5, inner sep=5pt, text width=2.8cm, minimum height=24mm
}]


\node [object] (sde) {
Stochastic differential equation
\[
d\rv{X}_t = \sigma( \rv{X}_t ) \circ d\rv{B}_t
\]
};

\node [object, right=10mm of sde] (fpe) {
Fokker--Planck (Kolmogorov forward) equation
\[
\!\partial_t F = \tfrac12 \nabla\cdot \left( \Sigma_1\, \nabla F\right)
\]
};

\node [object, right=10mm of fpe] (LB) {
Laplace--Beltrami operator on Riemannian manifold
\[
\Delta_{G_1}=\Delta_{\Sigma_1^{-1}}
\]
};

\node [object, right=10mm of LB] (metric) {
Metric on Riemannian manifold
\[
G_1 = \Sigma_1^{-1}
\]
};


\draw[implies-implies,double equal sign distance] (sde) -- (fpe) node[above,midway] {\footnotesize in law};
\draw[implies-implies,double equal sign distance] (fpe) -- (LB);
\draw[implies-implies,double equal sign distance] (LB) -- (metric);

\end{tikzpicture}

}
Next we will use the fact that these relationships remain valid under pullback of the respective objects from $\set{M}_1$ to $\set{M}_0$.

\subsection{The pullback of the time-diffused process on $\set{M}_0$}
\label{ssec:pullback_SDE}

\subsubsection{The pullback process}

A deterministic  trajectory $\{\phi_\theta (x)\}_{0\le \theta\le \tau}$ initialized at $x$ at time $\theta=0$ lies on the curve $\{( \theta,\phi_\theta (x)): 0\le \theta\le \tau\}\subset \set{M}_1$. As these curves are parametrized by their initial conditions in $M$, we will now pull back the time-diffused process \eqref{eq:difftimeproc} to the manifold $\set{M}_0$ consisting of time and initial-condition pairs.
This amounts to describing the process in what is often called the {Lagrangian frame}.

Recall from \eqref{eq:Phi} the canonical mapping $\Phi: (\theta,x) \mapsto (\theta, \phi_\theta(x))$ from $\set{M}_0$ to the trajectory manifold $\set{M}_1$. It has the Jacobian matrix
\begin{equation}
\label{jacmatrix}
D\Phi (\theta,x) = \begin{pmatrix}
1 & 0^\top  \\
v(\theta,\phi_{\theta} x) & J_{\theta}(x)
\end{pmatrix},
\end{equation}
where $J_\theta(x) := \partial_x \left(\phi_\theta(x)\right)$ is a shorthand for the Jacobian matrix of the time-$\theta$-flow map of the deterministic system (i.e., the ODE \eqref{eq:SDE} with~$\ep =0$).

If $\rv{Y}_t := \Phi^{-1}(\rv{X}_t)$ denotes the pulled-back time-diffused process~\eqref{eq:difftimeproc}, then this satisfies the SDE $d\rv{Y}_t = \left(D\Phi^{-1}\right)\!(\rv{X}_t) \sigma(\rv{X}_t)\circ d\rv{B}_t$, since SDEs in Stratonovich form obey the chain rule~\cite[Prop.~3.4]{Pav14}. Using that $\left(D\Phi^{-1}\right)\!(\Phi(Y)) = \left(D\Phi(Y)\right)^{-1}$ and the identity \eqref{eq:invMat} we obtain that
\begin{equation}
\label{eq:AugStratonPullback}
\begin{aligned}
d\rv{Y}_t &= \begin{pmatrix}
1 & 0^\top \\
-J_{\rv{\theta}_t} (\rv{y}_t)^{-1} v(\rv{\theta}_t,\rv{x}_t) &  \, J_{\rv{\theta}_t} (\rv{y}_t)^{-1}
\end{pmatrix}\begin{pmatrix}
a & 0^\top \\
a v(\rv{\theta}_t,\rv{x}_t) &  \ep \, \mathrm{Id}
\end{pmatrix} \circ d\rv{B}_t\\
&= \begin{pmatrix}
 a & 0^\top \\
0 &  \ep \, J_{\rv{\theta}_t} (\rv{y}_t)^{-1}
\end{pmatrix}\circ d\rv{B}_t\,,
\end{aligned}
\end{equation}
with
\[
\rv{Y}_t = \begin{pmatrix}
\rv{\theta}_t\\
\rv{y}_t
\end{pmatrix},
\
\rv{B}_t = \begin{pmatrix}
\rv{w}_t\\ \rv{b}_t
\end{pmatrix}.
\]
Hence, the pullback to $\set{M}_0$ (block-)diagonalises the diffusion coefficient (matrix) of the time-diffused process, so that the independent noise processes $\rv{w}_t$ and $\rv{b}_t$ only influence the time and space parameters $\theta$ and $y$, respectively. This is reflected in the Laplace--Beltrami operator~\eqref{eq:LB_G} below, where there are no mixed time and space-derivatives.

The dynamic interpretation of this process is straightforward. If we pull back the time-diffused process to $\set{M}_0$, where deterministic trajectories are straight lines, then
\begin{itemize}
\item the noise $\rv{w}_t$ drives the diffusion \emph{along} a single trajectory,
\item $\rv{b}_t$ drives the diffusion \emph{between} trajectories, and
\item the influence of the nonlinear dynamics on the diffusion between trajectories is encoded in the appearance of the Jacobian matrices~$J_{\theta}$.
\end{itemize}

\subsubsection{The pullback metric}

To obtain a characterisation of the pulled-back process $\rv{Y}_t$ in law, we will now pull back the metric $G_1$ to $\set{M}_0$ and consider the Brownian motion it generates thereon.

We briefly recall some general facts about pullback metrics. Let $\Psi: M\to N$ be a diffeomorphism between two smooth manifolds, and endow $N$ with a metric~$n$. The pullback metric $\Psi^*n$ on $M$ is defined by $(\Psi^*n)_x(v,w) = n_{\Psi(x)}(D\Psi\, v,D\Psi\, w)$, the subscripts referring to the point at which the metric is evaluated.
If in local coordinates at a point $\Psi(x)\in N$ the metric $n$ is expressed by the metric tensor $\mathcal{N}$, then the metric $\Psi^*n$ is expressed in local coordinates at $x\in M$ by $(D\Psi(x))^\top  \mathcal{N} D\Psi(x)$.

Note that for both $\set{M}_0$ and $\set{M}_1$, local coordinates coincide with the global Euclidean coordinates, and recall that $G_1=\Sigma_1^{-1}$.
With $D\Phi$ in (\ref{jacmatrix}), we obtain by Proposition~\ref{prop:M1metric} and the local coordinate expression for the pullback above that the pullback metric $G_0:= \Phi^* G_1$ with inverse metric tensor~$\Sigma_0$ satisfies
\begin{equation} \label{eq:G_as_pullback}
\begin{aligned}
\Sigma_0(\theta,y)^{-1} :\!\!&= G_0(\theta,y) = \left( \Phi^* G_1 \right) (\theta,y) = D\Phi(\theta,y)^{\top}\Sigma_1(\theta,\phi_{\theta}y)^{-1} D\Phi(\theta,y)\\
&=  \begin{pmatrix}
\frac{1}{ a ^2}  & 0^\top \\[3pt]
0 & \frac{1}{ \ep ^2} J_{\theta}(y)^\top  J_{\theta}(y)
\end{pmatrix}.
\end{aligned}
\end{equation}
We summarize this discussion in Figure~\ref{fig:metrics}.
\begin{figure}[htb]
	\centering
	
	\includegraphics[width = \textwidth]{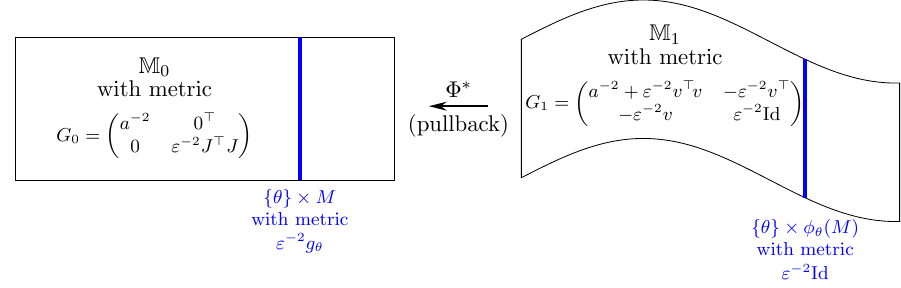}
	\caption{The metric $G_1$ associated with the time-diffused process on $\set{M}_1$ and its pullback~$G_0$ on~$\set{M}_0$. The induced metrics on the $\theta$-time-slices of the manifolds (blue vertical lines) are~$\frac{1}{\ep^2}g_{\theta}$ and $\frac{1}{\ep^2} \mathrm{Id}$, respectively. Note that the pullback decouples the metric---and hence the diffusion coefficient of the process---into purely temporal and spatial components.}
	\label{fig:metrics}
\end{figure}

Recall from the introduction that $g_\theta$ is the pullback of the Euclidean metric by $\phi_\theta$ to~$M$, namely $\phi_\theta^*e$.
In Euclidean coordinates we have~$g_{\theta}(y) = J_{\theta}(y)^\top  J_{\theta}(y)$.
Thus, denoting by $\Delta_{g_\theta}$ the Laplace--Beltrami operator with respect to the metric $g_\theta$ and using~\eqref{eq:LB} and \eqref{eq:G_as_pullback}, the Laplace--Beltrami operator of the Riemannian manifold $(\set{M}_0,G_0)$ is given by
\begin{equation} \label{eq:LB_G}
\Delta_{G_0} F(\theta,\cdot) =  a^2\, \partial_{\theta}^2 F(\theta,\cdot) +  \ep ^2\, \Delta_{g_{\theta}} F(\theta,\cdot) =: 2\cL_0^*F(\theta,\cdot),
\end{equation}
because $\mathrm{det}\,J_{\theta} \equiv 1$. Again, the associated boundary conditions are homogeneous Neumann.
Note that in analogy to the previous subsection and Proposition~\ref{prop:M1metric} we denote in~\eqref{eq:LB_G} the generator of the standard Brownian motion on~$(\set{M}_0, G_0)$ by~$\cL_0^* = \frac12 \Delta_{G_0}$. Similarly to Corollary~\ref{cor:generator1} we obtain:
\begin{cor} \label{cor:generator0}
The generator of the pulled-back time-diffused process $\rv{Y}_t$ governed by \eqref{eq:AugStratonPullback} on $\mathbb{M}_0$ is given by~$\smash{ \cL_0^* = \frac12 \Delta_{\Sigma_0^{-1}} }$. Hence, $\rv{Y}_t$ is equivalent in law to the standard Brownian motion on the Riemannian manifold~$(\set{M}_0, G_0)$, where $G_0 = \Sigma_0^{-1}$.
\end{cor}

\begin{remark}
The pullback to the initial time $\theta=0$ is simply for convenience. Let $\phi_{\vartheta,\theta}$ denote the flow of the ODE $\dot{x}_t = v(t,x_t)$ from time $\vartheta$ to time~$\theta$, i.e., $\phi_{\vartheta,\theta} = \phi_{\theta}\circ \phi_{\vartheta}^{-1}$. We note that we could use the manifold of states at any time $\vartheta \in [0,\tau]$, $M_{\vartheta} := \phi_{0,\vartheta}(M)$, to define $\set{M}_{\vartheta} := \bigcup_{\theta\in [0,\tau]} \{\theta\} \times M_{\vartheta}$, and pull back the metric $G_1$ to $\set{M}_{\vartheta}$ using
\[
\Phi_{\vartheta}: (\theta,z) \mapsto (\theta, \phi_{\vartheta,\theta} (z)).
\]
Then the formulas \eqref{eq:G_as_pullback} and \eqref{eq:LB_G} hold with $J_{\theta} = \partial_z\left( \phi_{\vartheta,\theta}(z) \right)$ and $g_{\theta}$ modified accordingly. In particular, the spectra of the Laplace--Beltrami operators $\Delta_{G_0}$ and $\Delta_{\Phi_{\vartheta}^* G_1}$ coincide, and the associated eigenfunctions can be obtained by coordinate transformation with~$z = \phi_{0,\vartheta}(y)$.
\end{remark}

Next we will see how the time-diffused process connects to the currently established notion of coherent sets, and how it extends this notion to semi-material coherent sets.

\section{The inflated dynamic Laplace operator and connections to the dynamic Laplace operator}\label{sec:dynLapConns}

In the following, the dependence of the metric $G_0$ and of related objects on the parameter $a$ is going to be of central interest. Hence, we will explicitly denote this dependence by writing~$\Goa$.
The parameter $\epsilon$ is the strength of the spatial diffusion;  see (\ref{eqn:fp}) and (\ref{eq:SDE}).
As discussed in the introduction, this is usually taken to be small \cite{F13}, or in the vanishing limit when constructing the dynamic Laplace operator $\Delta^D$~\cite{Fro15}.
Our inflated dynamic Laplace operator \eqref{eq:LB_G} corresponding to the SDE \eqref{eq:AugStratonPullback} has the additional parameter~$a$.
In this section we will show that for any fixed $\epsilon>0$, in the limit as $a\to\infty$ the properties of $\Delta_{G_{0,a}}$ mimic those of~$\Delta^D$.
Therefore, for simplicity from now on we almost exclusively set $\epsilon=1$ and retain only the parameter~$a$.

\subsection{The dynamic Laplace operator}
\label{ssec:dynLap}

The dynamic Laplace operator $\Delta^D$ in (\ref{DL}) arose from the desire to find (material) coherent sets~\cite{Fro15} using purely geometric constructions.
Coherent sets $A\subset M$ maximally inhibit mixing due to their boundary size remaining small under evolution by $\phi_t$ relative to enclosed volume.
An evolving boundary remaining small is a common measure of advective mixing \cite{pierrehumbert91,ottino} because in the presence of small diffusion the dispersion of mass outside $\phi_t(A)$ is proportional to the boundary size of $\phi_t(A)$.
This small evolving boundary geometry has been encoded in statements for a new dynamic isperimetric theory \cite{Fro15}.
In particular, the classical Cheeger inequality, which relates the geometry of a Riemannian manifold to the first nontrivial eigenvalue of its Laplace--Beltrami operator \cite{cheeger70}, and the classical Federer--Fleming theorem \cite{FF}, which equates the Cheeger constant and the Sobolev constant, were each extended to the dynamic situation in \cite{Fro15,FrKw17} and applied to defining and discovering finite-time coherent sets.
One aspect we will repeatedly use is that leading nontrivial eigenvalues of both $\Delta^D$ and $\Delta_{G_{0,a}}$ are strongly connected with coherence in our dynamical system.
Links between these dynamic isoperimetric quantities and the spectrum of $\Delta^D$, and their counterparts in our time-expanded geometry will be taken up in this section.

\subsection{The dynamic Laplace operator as the averaging-limit}
\label{ssec:avging}

We show that the dynamic Laplace operator $\Delta^D$ arises from $\Delta_{\Goa}$ in the limit of infinite temporal diffusion.
The intuition from the metric point of view is to note that as $a\to\infty$ the distance according to $G_{0,a}$ between two points $(\theta_1,x)$ and $(\theta_2,x)$ on the same deterministic trajectory in $\set{M}_0$ goes to zero;  see the left-hand image in Figure~\ref{fig:metrics}.
Therefore, in the $a\to\infty$ limit the importance of temporal displacements vanishes and all relevant information is captured by the temporal average (\ref{DL}) defining $\Delta^D$.

To demonstrate this convergence formally we consider the statistical properties of the process \eqref{eq:AugStratonPullback} in the $ a\to \infty$ limit.
To this end recall that the generator~$\cL_0^*$ of \eqref{eq:AugStratonPullback} is the half Laplace--Beltrami operator~$\Delta_{\Goa}$ from~\eqref{eq:LB_G}.
Using~\eqref{foot:ItoFPE}, we see that an SDE in It{\^o} form having generator~$\cL_0^*$~\eqref{eq:LB_G} is given by
\begin{equation} \label{eq:AugIto2slowtime}
\begin{aligned}
d\rv{\theta}_t &=  a\, d\rv{w}_t, \\
d\rv{y}_t &= \frac{1}{2} \nabla_y\cdot \left( J_{\rv{\theta}_t}^{-1} J_{\rv{\theta}_t}^{-\top} \right)(\rv{y}_t)\,dt + J_{\rv{\theta}_t}(\rv{y}_t)^{-1}\,d\rv{b}_t,
\end{aligned}
\end{equation}
The temporal process $\rv{\theta}_t$ is much faster than the spatial process when $ a\gg 1$, and thus the results of averaging can be applied.
More precisely, it follows from \cite[Remark 10.2, Section 10.7, Section 17.4]{PaSt08} that the slow process $\rv{y}_t$ converges weakly (i.e., in distribution) to the averaged process
\begin{equation} \label{eq:avgprocess}
d\bar{\rv{y}}_t = \bar{v}(\bar{\rv{y}}_t)\,dt + \bar{\sigma}(\bar{\rv{y}}_t)\,d\rv{b}_t,
\end{equation}
as $ a\to \infty$, where $\bar{v} := \frac12 \nabla\cdot \bar{\Sigma}$,  $\bar{\sigma}$ is any matrix-valued function satisfying~$ \bar{\sigma}\bar{\sigma}^\top  = \bar{\Sigma}$, and
\begin{equation}
\label{barSigma}
\bar{\Sigma}(y) := \frac{1}{\tau}\int_0^{\tau}  \left( J_{\theta}^{-1} J_{\theta}^{-\top} \right) (y)\,d\theta.
\end{equation}
Note that $\bar{\Sigma}$ is symmetric and positive definite as an integral of symmetric positive-definite matrices, thus one can find a (nonunique) $\bar{\sigma}$ such that~$ \bar{\sigma}\bar{\sigma}^\top  = \bar{\Sigma}$.

We see by~\eqref{foot:ItoFPE} that the limiting slow process $\bar{\rv{y}}_t$ has the forward generator $\bar{\cL}^*:H^2(M)\to L^2(M)$ given by
\[
\bar{\cL}^* f = \nabla\cdot\Big( (-\bar v + \underbrace{\tfrac12 \nabla\cdot \bar\Sigma}_{=\bar{v}} ) f + \tfrac12 \bar\Sigma\,  \nabla f \Big)   = \frac12 \nabla \cdot \left( \bar{\Sigma}\, \nabla f\right),
\]
which does not depend on the particular choice of~$\bar{\sigma}$.
Recall that $g_\theta(y)$ has matrix representation $J_{\theta}^{-1} J_{\theta}^{-\top}(y)$, so by (\ref{eq:LB}) we have $\Delta_{g_{\theta}} f = \nabla \cdot \big( J_{\theta}^{-1} J_{\theta}^{-\top}\,\nabla f \big)$.
We can now immediately see from~\eqref{DL} and (\ref{barSigma}) that $\frac12 \Delta^D = \bar{\cL}^*$.
In other words, the dynamic Laplace operator is (twice) the generator of the dominant spatial process $\rv{y}_t$ in (\ref{eq:AugIto2slowtime}) as~$a \to \infty$.
The expression for $\bar{\Sigma}$ appears as a harmonic mean of the metrics $g_t$ in \cite{KK};  here it arises naturally in the limit of speeding up time in the temporal diffusion.

Let us briefly re-introduce the parameter $\ep>0$ into \eqref{eq:AugIto2slowtime} to compare the two situations as at the end of section~\ref{sect:diff-time}, where we first considered the the limit $\ep\to 0$ for fixed $a>0$, and then the limit $a\to\infty$ for fixed $\ep>0$:
\begin{align*}
	d\rv{\theta}_t &=  a\, d\rv{w}_t, \\
d\rv{y}_t &= \frac{\ep^2}{2}\, \nabla_y\cdot \left( J_{\rv{\theta}_t}^{-1} J_{\rv{\theta}_t}^{-\top} \right)(\rv{y}_t)\,dt + \ep J_{\rv{\theta}_t}(\rv{y}_t)^{-1}\,d\rv{b}_t\,.
\end{align*}
This SDE is in law equivalent to the pullback of~\eqref{eq:difftimeproc}.
The theory of averaging~\cite{PaSt08} allows one to draw an equivalence between these two limits: the limiting evolution of the (slow) $y$-component is in law governed by~\eqref{eq:avgprocess} in both cases, for $\ep\to 0$ on the timescale $t = \mathcal{O}(\ep^{-2})$, and for $a\to\infty$ on the timescale $t = \mathcal{O}(1)$, respectively.
Hence, the two situations can be transformed into one another via a suitable rescaling of time.
We remark that $\Delta^D$ can also be obtained \cite{KaSch21} as the leading-order term for $\ep\to 0$ in the Fokker--Planck equation~\eqref{eqn:fp} viewed in Lagrangian coordinates~\cite{thiffeault2003advection} (i.e.\ where all times $t\in[0,\tau]$ are pulled back to $t=0$), using operator-averaging techniques adapted from~\cite{Krol91}.

\subsection{Cheeger and Sobolev inequalities}
\label{ssec:cheeger}

We now begin to analyse the Laplace--Beltrami operator $\Delta_{\Goa}$ and the associated metric~$\Goa$.
In the sequel we rarely consider time-evolution of the associated augmented process~\eqref{eq:AugStratonPullback}, and so we revert to denoting the temporal coordinate of the augmented manifold~$\set{M}_0$ by~$t$ (instead of~$\theta$), to stress its temporal character.
The purpose of this section is to link the Cheeger and Sobolev constants of our time-augmented manifold with the dynamic Cheeger and dynamic Sobolev constants of \cite{Fro15,FrKw17}.

Suppose that $\Gamma\subset M$ is a co-dimension 1 $C^\infty$ surface disconnecting $M$ into the disjoint union $M=A_1\cup\Gamma\cup A_2$, with $A_1, A_2$ connected submanifolds.
\new{Let $\ell$ denote Lebesgue measure on $M$,} let $\iota:\Gamma\xhookrightarrow{}M$ denote the inclusion map, $\iota^*g_t$ the induced metric on $\Gamma$ arising from $g_t$, and $V_{\iota^*g_t}$ the corresponding volume form on $\Gamma$.
Recall the dynamic Cheeger constant \cite[equation (20)]{Fro15} or \cite[equations (3.6)--(3.7)]{FrKw17}, which we can write as
\begin{equation}
\label{dynche}
h^D:=\inf_\Gamma\frac{\frac{1}{\tau}\int_0^\tau V_{\iota^*g_t}(\Gamma)\ dt}{\min\{\ell(A_1),\ell(A_2)\}}.
\end{equation}
In the expression~(\ref{dynche}), we select a $\Gamma$ disconnecting $M$ and follow its forward evolution under the nonlinear dynamics~$\phi_t$.
We wish to find the initial disconnector $\Gamma$ whose average evolved size is least, relative to the volumes of the two connected components of $M$, as this represents the potential boundary of a finite-time coherent set.
The dynamic Cheeger inequality \cite{Fro15,FrKw17} states that
\begin{equation}
\label{dynamiccheeger}
h^D\le 2\sqrt{-\lambda^D_2},
\end{equation}
where $\lambda^D_2$ is the first nontrivial eigenvalue of $\Delta^D$.

To motivate the next construction, we note the evolution of any surface $\Gamma$ in (\ref{dynche}) that disconnects $M$ will trace out a surface $\bbGamma':=\bigcup_{t\in [0,\tau]} (\{t\} \times \phi_t\Gamma)$ that disconnects our ``forward-evolved'' time-augmented manifold~$\mathbb{M}_1$.
With $\Phi$ from~\eqref{eq:Phi} we may pull back such a traced-out surface to obtain a surface $\bbGamma=\Phi^{-1} \bbGamma'$ disconnecting~$\mathbb{M}_0$.
Of course by the construction of $\bbGamma'$ and the definition of $\Phi$, the surface $\bbGamma$ has a constant section on each time fibre, namely~$\Gamma$.
That is, $\bbGamma\cap (\{t\}\times M)=\Gamma$ for each~$t\in [0,\tau]$.
In summary, the minimising disconnector $\Gamma$ from (\ref{dynche}) provides a particular ``constant-in-time'' disconnector $\bbGamma$ of our augmented Riemannian manifold~$(\mathbb{M}_0,G_{a})$.
However, we may also consider more general disconnectors of $(\mathbb{M}_0,G_{a})$, which would represent a relaxation of materialness in the standard definition of coherent sets~\cite{Fro15}.
This can be accomplished using the standard Cheeger constant for the manifold $(\mathbb{M}_0,G_{a})$ for suitable~$a$.
Let $\mathbb{\incmap}:\bbGamma\xhookrightarrow{}\set{M}_0$ denote the inclusion map, $\mathbb{\incmap}^*\Goa$ the induced metric on $\bbGamma$, and $V_{\mathbb{\incmap}^*\Goa}$ the corresponding volume form on~$\bbGamma$. Set
\begin{equation}
\label{augche}
H_a:=\inf_{\footnotesize\bbGamma}\frac{V_{\mathbb{\incmap}^*\Goa}(\bbGamma)}{\min\{V_{\Goa}(\mathbb{A}_1),V_{\Goa}(\mathbb{A}_2)\}},
\end{equation}
where $\bbGamma\subset \mathbb{M}_0$ is a co-dimension 1 $C^\infty$ surface disconnecting $\mathbb{M}_0$ into the disjoint union $\mathbb{M}_0=\mathbb{A}_1\cup\bbGamma\cup \mathbb{A}_2$, with $\mathbb{A}_1, \mathbb{A}_2$ connected submanifolds.
Such a $\bbGamma$ is a potential boundary of a semi-material coherent set with parameter $a$.
The standard Cheeger inequality in this situation reads
\begin{equation}
\label{cheegeraug}
H_a\le 2\sqrt{-\Lambda_2},
\end{equation}
where $\Lambda_2$ is the first nontrivial eigenvalue of $\Delta_{G_{0,a}}$.
We will shortly address the relationship between $h^D$ and $H_a$, and the behaviour of $H_a$ with increasing $a$.

Similarly, we recall the dynamic Sobolev constant \cite[equation (21)]{Fro15} or \cite[\S3.2]{FrKw17}:
\begin{equation}
\label{dynsob}
s^D:=\inf_{f\in C^\infty(M)} \frac{\frac{1}{\tau}\int_0^\tau \int_M \|\nabla_{g_t}f\|_{g_t}\ d\ell dt}{\inf_{\alpha\in\mathbb{R}}\int_M |f-\alpha|\ d\ell},
\end{equation}
where $\nabla_g f$ is the unique vector field on $M$ satisfying
\begin{equation}
\label{eq:grad}
g(\nabla_g f,w)=w(f)
\end{equation}
for all vector fields $w:M\to\mathbb{R}^d$, and $w(f)$ is the Lie derivative of $f:M\to\mathbb{R}$.
Let us \new{interpret} these definitions in coordinates.
In coordinates, $w(f)(x)=\sum_{i=1}^d w_i(x)\cdot \frac{\partial f}{\partial x_i}(x)$, the directional derivative at $x$ of $f$ in the direction $w(x)$;  we will denote by $[\partial f(x)]$ the matrix of partial derivatives evaluated at $x$.
We will occasionally, but not universally, use square brackets around objects like $\partial f$, $g$, and $G$ to emphasise that an object is to be interpreted as a matrix.
The expression $g(\nabla_g f,w)(x)$ in coordinates is $\nabla_g f(x)^\top [g(x)]w(x)$, where $[g(x)]$ is the coordinate matrix representation of $g$ at $x\in M$.
Thus, in coordinates \eqref{eq:grad} becomes $\nabla_g f(x)^\top [g(x)]w(x) = [\partial f(x)]^\top w(x)$.
Since this holds for all vector fields $w$ we have
\begin{equation}
\label{eq:grad2}
\nabla_g f(x)=[g(x)]^{-1} \partial f (x),
\end{equation}
 using symmetry and invertibility of $[g(x)]$.
Finally we have
\begin{eqnarray}
\nonumber \|\nabla_{g}f(x)\|^2_{g} & = & g(\nabla_g f(x),\nabla_g f(x)) \\
\nonumber & = & \nabla_g f(x)^\top [g(x)] \nabla_g f(x) \\
\label{eq:normgrad}
& = & [\partial f(x)]^\top [g(x)^{-1}] [\partial f(x)]\qquad\mbox{by (\ref{eq:grad2})}.
\end{eqnarray}
We may also define a Sobolev constant for the Riemannian manifold $(\mathbb{M}_0,\Goa)$ in the usual way:
\begin{equation}
\label{augsob}
S_a:=\inf_{F\in C^\infty(\mathbb{M}_0)} \frac{\int_{\mathbb{M}_0} \|\nabla_{\Goa}F\|_{\Goa}\ dV_{\Goa}}{\inf_{\alpha\in\mathbb{R}}\int_{\mathbb{M}_0} |F-\alpha|\ dV_{\Goa}},
\end{equation}
where $V_{\Goa}$ denotes the volume measure with respect to~$\Goa$. Note that by~\eqref{eq:G_as_pullback} we have~$dV_{\Goa} = \frac{1}{a} dt\,d\ell$.

\begin{proposition}
\label{isoperprop}
One has
\begin{enumerate}
\item $H_a=S_a\le s^D=h^D$ for all $a\ge 0$,
\item $H_a$ and $S_a$ are nondecreasing in $a\ge 0$.
\end{enumerate}
\end{proposition}
\begin{proof}
\,
\begin{enumerate}
\item The fact that $H_a=S_a$ follows from the Federer--Fleming Theorem (e.g.\ \cite[p.~131]{chavel_eigenvalues}).  The fact that $s^D=h^D$ follows from the dynamic Federer--Fleming Theorem (Theorem 3.1 \cite{Fro15}, Theorem 3.3 \cite{FrKw17}).

    We now treat the inequality $S_a\le s^D$.
    For $f\in C^\infty(M)$, denote by $s^D(f)$ the infimand of (\ref{dynsob}).
For $\epsilon>0$ let $f_\epsilon\in C^\infty(M)$ be such that $s^D(f_\epsilon)\le s^D+\epsilon$.
Define $F_\epsilon\in C^\infty(\mathbb{M}_0)$ by $F_\epsilon(t,x)=f_\epsilon(x)$ for $(t,x)\in \mathbb{M}_0$.
Then
\begin{eqnarray*}
S_a(F_\epsilon) &:=& \frac{\int_{\mathbb{M}_0} \|\nabla_{\Goa}F_\epsilon \|_{\Goa}\ dV_{\Goa}}{\inf_{\alpha\in\mathbb{R}}\int_{\mathbb{M}_0} |F_\epsilon-\alpha|\ dV_{\Goa}}\\
&=&\frac{\int_{[0,\tau]}\int_M \left( a^2(\partial_tF_\epsilon)^2+[\partial_xF_\epsilon]^\top(J_t^\top J_t)^{-1}[\partial_x F_\epsilon] \right)^{1/2}\, \frac{1}{a}dt\ d\ell}{\inf_{\alpha\in\mathbb{R}}\int_{[0,\tau]}\int_M |F_\epsilon-\alpha|\ \frac{1}{a}dt\ d\ell}\\
&=&\frac{\frac{1}{\tau}\int_{[0,\tau]}\int_M \left([\partial_xF_\epsilon]^\top(J_t^\top J_t)^{-1}[\partial_x F_\epsilon] \right)^{1/2}\, dt\ d\ell}{\frac{1}{\tau}\left(\int_{[0,\tau]}\ dt \right) \inf_{\alpha\in\mathbb{R}}\int_M |F_\epsilon-\alpha|\  d\ell}\\
&=&s^D(f_\epsilon)\le s^D+\epsilon.
\end{eqnarray*}
Since $\epsilon$ was arbitrary, this implies~$S_a=\inf_{F\in C^\infty(\mathbb{M}_0)} S_a(F)\le s^D$.
\item By (\ref{eq:G_as_pullback}), (\ref{eq:normgrad}), and (\ref{augsob}) one has
\[
S_a=\inf_{F\in C^\infty(\mathbb{M}_0)} \frac{\int_{[0,\tau]}\int_M \left( a^2(\partial_tF)^2+[\partial_xF]^\top(J_t^\top J_t)^{-1}[\partial_x F] \right)^{1/2}\, \frac{1}{a}dt\ d\ell}{\inf_{\alpha\in\mathbb{R}}\int_{[0,\tau]}\int_M |F-\alpha|\ \frac{1}{a}dt\ d\ell}.
\]
The result follows by noting that the integrand is nondecreasing in~$a$.
By the equality $H_a=S_a$ in Part 1 we obtain that $H_a$ is nondecreasing in~$a$.
\end{enumerate}
\end{proof}

\begin{remark}
\label{cheegerremark}
The values $H_a$ and $S_a$ quantify the maximum level of coherence present:  low $H_a=S_a$ indicates strong coherence.
Proposition \ref{isoperprop} says that increasing $a$ leads to greater boundary lengths relative to volume on $\set{M}_0$ and therefore lower coherence.
Referring to (\ref{eq:G_as_pullback}), with increasing $a$, the numerator in (\ref{augche}) can be reduced by aligning the tangent spaces of $\bbGamma$ with the time axis (recall we are always working in $\set{M}_0$).
Thus, as one increases $a$, we expect the minimising $\bbGamma$ to become increasingly material;   for example, in the lower left panel of Figure \ref{fig:topsum}) the boundary of the pale blue set  will become more horizontal.
In summary, there is a trade-off between materiality and coherence, with the former increasing and the latter decreasing with increasing $a$.
\end{remark}

\subsection{Spectrum and eigenfunctions}
\label{ssec:eigenvalues}

The spectrum of the dynamic Laplace operator and our proposed inflated dynamic Laplace operator characterises the strength of coherence and suggests natural numbers of coherent sets.
Eigenvalues near to zero indicate the presence of strong coherence, and their corresponding eigenfunctions encode the location of coherent sets in the phase space.

\subsubsection{Spectrum of the dynamic Laplace operator}

We begin by recalling the variational characterisation of eigenvalues of the dynamic Laplacian and then link these to our inflated dynamic Laplace operator $\Delta_{\Goa}$ on augmented space.
We consider the dynamic Laplacian $\Delta^D$ as defined in \cite[equation (28)]{Fro15} and~\cite[equation (4.12)]{FrKw17}.
The boundary condition on $M$ is the natural one for the dynamic Laplacian and corresponds to a ``dynamic Neumann boundary condition''; see \cite[equation (30)]{Fro15} for an explicit representation.

By \cite[Theorem 4.1]{Fro15} and \cite[Remark 4.2]{Fro15} or \cite[Theorem 4.4]{FrKw17} and the discussion in \cite[Section 4.2 ]{FrKw17}, the dynamic Laplacian has a countable discrete spectrum $0=\lambda^D_1>\lambda^D_2\ge\lambda^D_3\cdots$ with the corresponding eigenfunctions denoted~$\mathbf{1}_M \equiv f_1,f_2,\ldots \in C^\infty(M)$.
Let $S_0=L^2(M)$ and for $k\ge 1$ let
$S_{k}=\{f\in L^2(M):\langle f,f_i\rangle=0, 1\le i\le k\}$.
By \cite[equation (34)]{Fro15},
one has the following variational representation of $\lambda^D_k$ for $k\ge 1$:
\begin{equation}
\label{eq:dl-lamk}
\lambda^D_k
=-\inf_{f:M\to\mathbb{R}, f\in S_{k-1}} \frac{\int_0^\tau \int_M \|\nabla_{g_t}f(x)\|^2_{g_t}\ d\ell(x) \, dt}{\int_0^\tau\int_M f(x)^2\ d\ell(x) \,dt}.
\end{equation}

\subsubsection{Spectrum of the inflated dynamic Laplace operator}
\label{sec:specaug}
We recall that $\Delta_{\Goa}$ is  equipped with homogeneous Neumann boundary conditions.
By standard theory (e.g.\ \cite{lablee}) $\Delta_{\Goa}$ has a discrete spectrum $0=\Lambda_{1,a}>\Lambda_{2,a}\ge\Lambda_{3,a}\cdots$ with eigenfunctions $\mathbf{1}_{\set{M}_0} \equiv F_1,F_2,\ldots\in C^\infty(\mathbb{M}_0)$.
Some of the eigenfunctions are easily identifiable: for $k\ge 1$, the functions $F_k^\temp(t,\cdot) := \cos ( k\pi t/\tau)$ are clearly eigenfunctions with eigenvalue $\smash{ \Lambda_{k,a}^{\mathrm{temp}} := -(a\pi k / \tau)^2 }$.
We call these eigenfunctions \emph{temporal modes} or \emph{temporal eigenfunctions} because they are constant in space and vary only in time.
Define $W_0=\{f\in L^2([0,\tau]): \int_0^\tau f(t)\, dt=0\} = \mathbf{1}_{[0,\tau]}^{\perp}$ and $\mathbb{S}_0^\temp:=\left\{ f\mathbf{1}_M \, : \, f\in W_0 \right\} \subset L^2(\set{M}_0)$,  a subspace containing all temporal eigenfunctions.

The operator $\Delta_{\Goa}$ is symmetric on its domain in $L^2(\set{M}_0)$ and so its eigenfunctions are $L^2$-orthogonal.
Therefore if $F \in L^2(\set{M}_0)$ is a non-temporal eigenfunction, then $F \perp \big( (t,x) \mapsto \cos (k\pi t/\tau) \big)$ for all $k\ge 1$.
As the temporal eigenfunctions are dense in $\mathbb{S}_0^\temp$, we have that $F$ is orthogonal to every function in $\mathbb{S}_0^\temp$. One thus has $0 = \int_{\set{M}_0}F\, f\mathbf{1}_M\, d\ell dt = \int_0^{\tau} f(t) \int_M F(t,\cdot)\,d\ell\, dt$ for all $f\in W_0$, giving $\int_M F(\cdot,x)\,d\ell(x) \in W_0^{\perp} = \mathrm{span}( \mathbf{1}_{[0,\tau]} )$, which implies that $F$ has constant spatial means:
\begin{equation}
\label{eq:spatial_meanzero}
\int_M F(\cdot,x)\,d\ell(x) = \text{const} \text{ a.e.\ on } [0,\tau].
\end{equation}
We denote the subspace of all such functions by $\mathbb{S}_0^\spat=(\mathbb{S}_0^\temp)^\perp$.
The non-temporal eigenfunctions will be called \emph{spatial eigenfunctions} or \emph{spatial modes}. In general, they will vary both in space and in time. The associated eigenvalues will be denoted by $0=\Lambda_{1,a}^{\spat}>\Lambda_{2,a}^{\spat}\ge \Lambda_{3,a}^{\spat}\cdots$. The spatial and temporal eigenvalues partition the spectrum $\sigma(\Delta_{\Goa})$.
By~\eqref{eq:spatial_meanzero}, an eigenfunction $F$ is a spatial mode if and only if its spatial mean $t\mapsto \int_M F(t,\cdot)\,d\ell$ is an a.e.\ constant function of time.
We will later use this distinct behavior to numerically distinguish between temporal and spatial modes.

\subsubsection{Behaviour of the spectrum of $\Delta_{\Goa}$ with increasing $a$}
We next address the behavior of the eigenvalues of $\Delta_{\Goa}$ with increasing $a>0$, linking them to the eigenvalues of the dynamic Laplace operator.
Let $\mathbb{S}_0=L^2(\mathbb{M}_0)$ and for $k\ge 1$ let $\mathbb{S}_{k}=\{F\in L^2(\mathbb{M}_0):\langle F,F_i\rangle=0, 1\le i\le k\}$.
For $k\ge 1$, one has the standard variational characterisation of eigenvalues of Laplace--Beltrami operators (recall that the volume form $V_{\Goa}$ is given by $dV_{\Goa} = \frac{1}{a}d\ell$)
\begin{equation}
\label{eq:LB-lamk}
\Lambda_{k,a}=-\inf_{F:\mathbb{M}_0\to\mathbb{R}, F\in \mathbb{S}_{k-1}} \frac{\int_{\mathbb{M}_0} \left\|\nabla_{\Goa}F\right\|^2_{\Goa}\ dV_{\Goa}}{\int_{\mathbb{M}_0} F^2\ dV_{\Goa}}.
\end{equation}
Further, denoting the eigenfunction corresponding to $\Lambda^\spat_{k,a}$ by $F^\spat_{k}$ let us denote $\mathbb{S}^\spat_k=\{F \in \mathbb{S}_0^\spat:\langle F,F^\spat_i\rangle=0, 1\le i\le k\}$.
We then have the variational characterisation of spatial eigenfunctions:
\begin{equation}
\label{eq:LB-lamkspat}
\Lambda_{k,a}^\spat
=-\inf_{F:\mathbb{M}_0\to\mathbb{R}, F\in \mathbb{S}_{k-1}^\spat} \frac{\int_{\mathbb{M}_0} \left\|\nabla_{\Goa}F\right\|^2_{\Goa}\ dV_{\Goa}}{\int_{\mathbb{M}_0} F^2\ dV_{\Goa}}.
\end{equation}

Recall from \eqref{eq:G_as_pullback} that the matrix representation of $\Goa$ is
\begin{equation}
\label{eq:Gmat}
[\Goa(t,x)]=\left(
         \begin{array}{cc}
           1/a^2 & 0 \\
         0 & g_t(x) \\
         \end{array}
       \right).
\end{equation}
Thus, using (\ref{eq:normgrad}) we have
\begin{eqnarray}
\nonumber\|\nabla_{\Goa}F(t,x)\|^2_{\Goa}&=&[\partial F(t,x)]^\top[\Goa(t,x)]^{-1}[\partial F(t,x)]\\
\nonumber&=&a^2(\partial_t F(t,x))^2 + [\partial_x F(t,x)]^\top[g_t(x)]^{-1}[\partial_x F(t,x)]\\
\label{eq:Gdecomp}&=&a^2(\partial_t F(t,x))^2 + \|\nabla_{g_t}F(t,x)\|^2_{g_t}\,.
\end{eqnarray}

\new{Before stating our main result for this subsection, we note the following: for a fixed $a$, because the spectrum of $\Delta_{\Goa}$ (counting multiplicity) can be written as the union $\smash{ \sigma(\Delta_{\Goa})=\bigcup_{k\ge 1}\Lambda_{k,a}^\temp\cup\bigcup_{k\ge 1}\Lambda_{k,a}^\spat }$,} the indexing of the elements of $\sigma(\Delta_{\Goa})$ immediately yields $\Lambda_{k,a}^\temp\le \Lambda_{k,a}$ and $\Lambda_{k,a}^\spat\le \Lambda_{k,a}$ for $k\ge 1$.

\begin{theorem}
\label{thm:eigbounds}
\quad
\begin{enumerate}
\item For each $k\ge 1$ and $a > 0$ one has $\lambda_k^D \le \Lambda_{k,a}^{\mathrm{spat}}$.
\item For each $k\ge 1$, $\Lambda_{k,a}$, $\Lambda_{k,a}^{\mathrm{temp}}$, and $\Lambda_{k,a}^{\mathrm{spat}}$ are nonincreasing in $a\ge 0$,
\item For each $k\ge 1$, $\lim_{a\to\infty} \Lambda_{k,a}^\temp \to-\infty$.
\item For each $k\ge 1$, $\lim_{a\to\infty} \Lambda_{k,a}^{\mathrm{spat}} = \lim_{a\to\infty} \Lambda_{k,a} = \lambda_k^D$.
\end{enumerate}
\end{theorem}
\begin{proof}
See Appendix \ref{app:spectrum}.
\end{proof}

As $a$ increases, part 2 of Theorem~\ref{thm:eigbounds} states that $\Lambda_{k,a}^{\mathrm{temp}}$ and $\Lambda_{k,a}^{\mathrm{spat}}$ monotonically decrease.
This is intuitive because a larger $a$ leads to larger a value of
$\|\nabla_{\Goa} F\|_{\Goa}^2$ in \eqref{eq:LB-lamk} as we increasingly penalise variation of $F$ in the temporal direction.
We note that as $a$ increases, the ordering of eigenvalues $\Lambda_{k,a}$ in the full spectrum will change, and therefore the index $k$ is implicitly a function of~$a$.
Temporal eigenvalues are demoted to lower positions in the full spectrum as $a$ increases, leaving only spatial eigenvalues in the leading part of the full spectrum for sufficiently large~$a$.
Parts 1, 2, and 4 of Theorem~\ref{thm:eigbounds} are illustrated numerically in Figure~\ref{fig:ChiSow_gy-mxmx_evals_vs_a} for the Childress--Soward system from section~\ref{ssec:ChiSow_gy-mxmx}.
\begin{figure}[htb]
\centering
\includegraphics[width=0.5\textwidth]{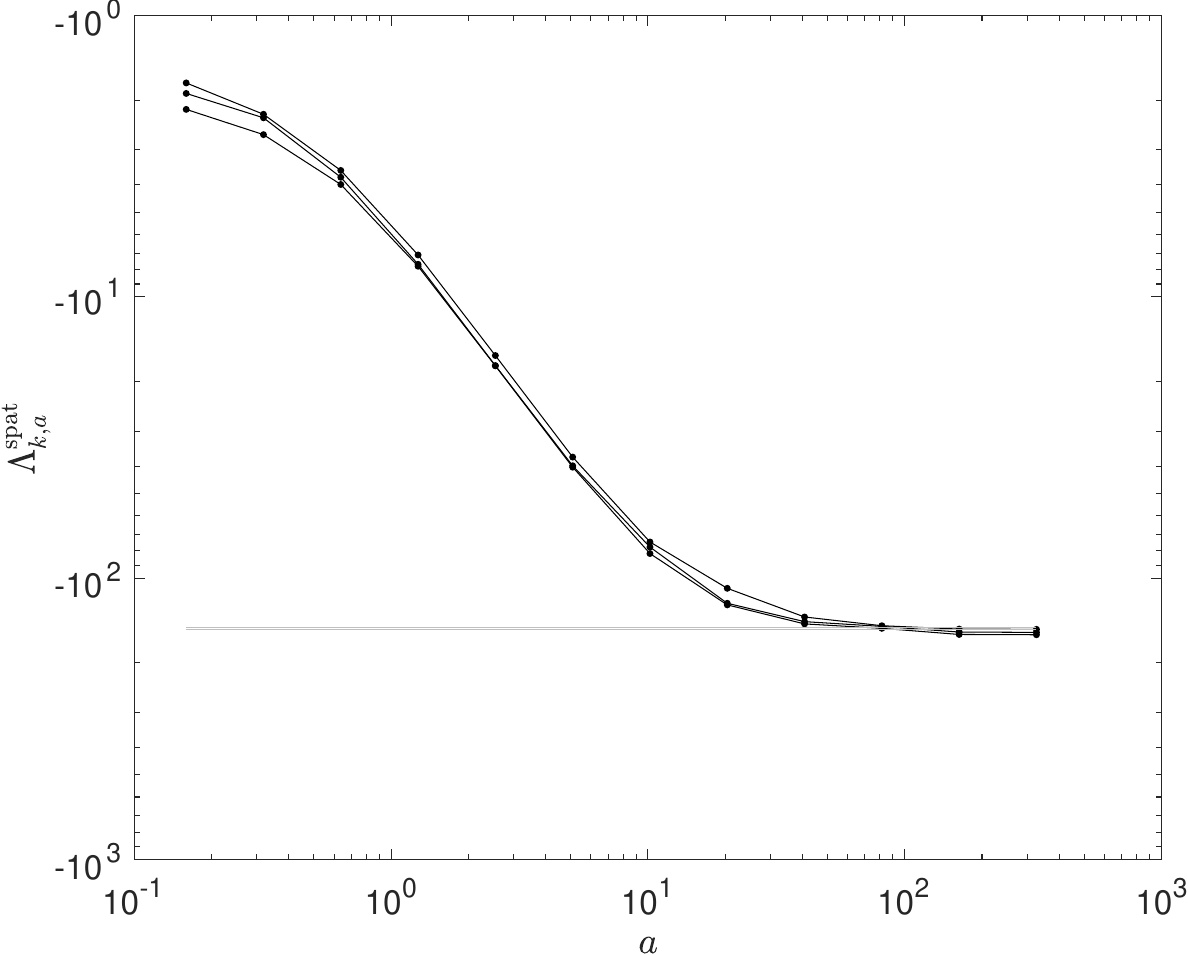}
\caption{The three subdominant spatial eigenvalues $\Lambda^\spat_2,\Lambda^\spat_3,\Lambda^\spat_4$ of $\Delta_{\Goa}$ versus~$a$. The associated system is discussed in section~\ref{ssec:ChiSow_gy-mxmx}. The grey horizontal lines (indistinguishable in this figure) indicate the values of the three subdominant eigenvalues of the dynamic Laplace operator~$\Delta^D$ for this system.}
\label{fig:ChiSow_gy-mxmx_evals_vs_a}
\end{figure}

For fixed $k$ and increasing $a$, we expect the eigenfunction $F_{k,a}^{\spat}$ to become more regular in the temporal direction as the infimum in (\ref{eq:LB-lamkspat}) seeks to reduce the combination of gradients in the temporal and spatial directions.
In the limit as $a\to\infty$, there will be vanishing variation in the temporal direction and we will recover the $k^{\mathrm{th}}$ eigenfunction of the dynamic Laplacian $\Delta^D$, copied across time.
In the other direction, as $a\to 0$, and the penalisation of the temporal variation diminishes, we expect $\smash{ F_{k,a}^{\spat}(t,\cdot) }$ to simply encode the spatial structure of $(M,g_t)$;  that is, $\smash{ F_{k,a}^{\spat}(t,\cdot)\approx f_{k,t} }$, where $f_{k,t}$ is the $k^{\mathrm{th}}$ eigenfunction of~$\Delta_{g_t}$.

From the above discussion we see that level sets of the eigenfunctions $F_{k,a}^{\spat}$, from which we will create our semi-material finite-time coherent sets, will interpolate from being strictly material (in the $a\to\infty$ limit) to rather non-material (for small~$a$).
This is consistent with the discussion of the behaviour of stochastic trajectories at the conclusion of section \ref{sect:diff-time}, and the behaviour of the minimising disconnectors $\bbGamma$ (which form boundaries of FTCS or part thereof) in Remark \ref{cheegerremark}.

\section{Semi-material coherent sets from the inflated dynamic Laplacian}
\label{ssec:practical}

We shall now discuss how we can utilize the previous theoretical considerations to identify coherent sets and their lifetimes.
In particular, we explain how we can identify different dynamical regimes---existence of coherent sets or global mixing---from the eigenmodes of the inflated dynamic Laplace operator~$\Delta_{\Goa}$.
We illustrate these ideas using the partially coherent Childress--Soward system described in full detail in section~\ref{ssec:ChiSow_gy-mxmx}.

\subsection{Choosing the temporal diffusion parameter $a$}
\label{sssec:choosing_a}

For large $a$, temporal diffusion in $\Delta_{\Goa}$ will dominate and because of the variational (minimisation) characterisation of the eigenvalues, any temporal variation in the eigenfunctions will be heavily penalised.
Therefore we expect eigenfunctions corresponding to eigenvalues early in the spectrum to be purely spatial.
More precisely, from section~\ref{ssec:avging} and Theorem~\ref{thm:eigbounds}, for large $a$ we expect spatial eigenfunctions of $\Delta_{\Goa}$ to be approximately ``time-copied'' versions of the eigenfunctions of the dynamic Laplace operator.
In the other direction, for small $a$ there is very low temporal diffusion and different time fibres of eigenfunctions $F(t,\cdot)$ will approximately decouple and depend almost entirely on the spatial metric $g_t$ on the $t^{\mathrm{th}}$ time fibre.
If one were to attempt to extract coherent sets through level sets of $F$ in this small $a$ regime, the coherent sets could be highly non-material.

We aim for a sweet spot for $a$ somewhere in between these extremes.
We would like to have the dominant eigenfunctions of $\Delta_{\Goa}$ consisting mostly of spatial eigenfunctions, because it is these we are primarily interested in, but also including a small number of temporal eigenfunctions, so that such an $a$ allows some temporal variation in the spatial eigenfunctions.
The latter point is crucial for being able to discriminate between coherent and mixing regimes over our full time domain.

We now discuss a heuristic to select a lower bound for $a$.
The largest nonzero eigenvalue from the purely temporal component of $\Delta_{\Goa}$ is $\Lambda_1^\temp=-a^2\pi^2/\tau^2$, where $\tau$ is the flow duration.
Assuming a rectangular domain $M$ with (maximal) side length $l$, the largest nontrivial eigenvalue of the Laplace--Beltrami operator on $(M,e)$ is $-4\pi^2/l^2$ for periodic boundary conditions, and $-\pi^2/l^2$ for homogeneous Neumann boundary conditions, respectively.
The spatial eigenvalues of $\Delta_{\Goa}$ will in general be larger in magnitude (more negative) than these values because of the presence of dynamics.
Thus, if we desire the contribution from the temporal component to be about the same as the spatial component (with no dynamics), in the periodic case we want $a^2\pi^2/\tau^2\approx 4\pi^2/l^2$, so we set $a_{\min}=2\tau/l$ as the lower bound for $a$.
Similarly, for homogeneous Neumann boundary conditions, we set $a_{\min} = \tau/l$.
In section \ref{ssec:ChiSow_gy-mxmx}, this leads to $a_{\min}=2/\pi$ for our numerical example.
We recommend beginning with $a_{\min}$ computed in this way and then increasing~$a_{\min}$.
Using this heuristic for the partially coherent Childress--Soward system in Section \ref{sec:examples}, one obtains a spectrum as shown in Figure~\ref{fig:ChiSow_gy-mxmx_evals}.
As predicted, we see that $\Lambda_2^\spat<\Lambda_1^\temp$
\begin{figure}[htb]
\centering
\includegraphics[width = 0.5\textwidth]{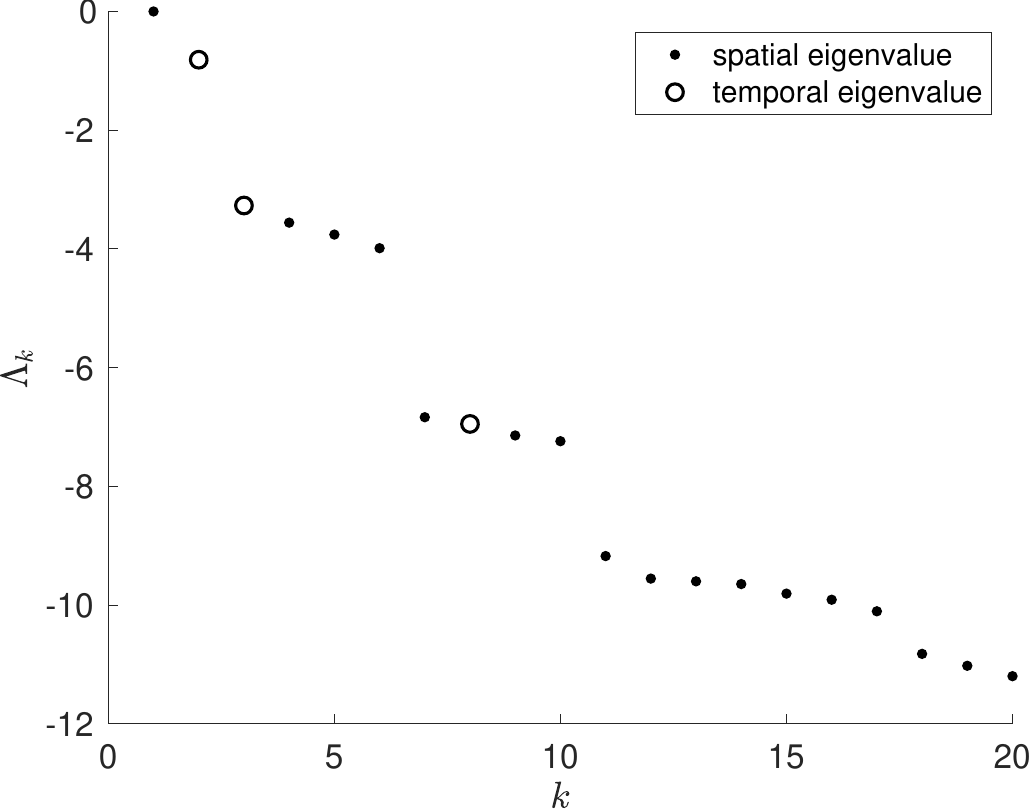}
\caption{Eigenvalues of $\Delta_{\Goa}$ for the partially coherent Childress--Soward system for ${ a = 2 / \pi }$, estimated by the FEM-based discretisation dscribed in section~\ref{sec:FEMDL}. Spatial modes are indicated by dots, temporal ones by circles.}
\label{fig:ChiSow_gy-mxmx_evals}
\end{figure}

The above discussion suggests that in theory one may make $\tau$ as large as one likes and the temporal parameter $a_{\min}$ can be scaled linearly with $\tau$ to compensate.
For numerical reasons it is better to choose $\tau$ larger than, but not much larger than, the expected maximal lifetime of the finite-time coherent sets of interest.
We pick up this point again in Section~\ref{ssec:a_timescale}.

\subsection{Distinguishing spatial and temporal eigenfunctions}
\label{sssec:distinguishing}

From section~\ref{sec:specaug} we know that spatial eigenfunctions $F$ of $\Delta_{\Goa}$ have time fibres $F(t,\cdot)$ with constant mean, cf.~\eqref{eq:spatial_meanzero}.
The temporal eigenfunctions have the form $F_k(t,x)=C\cos(k\pi t/\tau)$, $k\ge 1$.
We will numerically compute the variance of the means of the time fibres $F(t,\cdot)$;  if this variance is close to zero, the eigenfunction is spatial.
We now analytically determine the variance of the temporal eigenfunctions $F_k$, $k\ge 1$.
Let us normalise so that $\|F_k\|^2_{L^2(\mathbb{M}_0)}=\tau\ell(M)=\|\mathbf{1}\|^2_{L^2(\mathbb{M}_0)}$;
this implies $C = \sqrt{2}$.
The mean $s(t)$ of $F(t,\cdot)$ on the $t^{\mathrm{th}}$ time fibre is
\[
s(t):=\frac{1}{\ell(M)}\int_M C\cos(k\pi t/\tau)\ d\ell=\sqrt{2}\cos(k\pi t/\tau).
\]
Toward computing the temporal variance of the spatial means $s(t)$,  we note the mean of $s(t)$ is $\frac{1}{\tau}\int_0^\tau s(t)\ dt=0$.
Thus, the variance of $s(t)$ is
\[
\frac{1}{\tau} \int_0^\tau (\sqrt{2}\cos(k\pi t/\tau))^2\ dt = \frac{1}{\tau}\cdot 2\frac{\tau}{2} = 1.
\]
Therefore, with the above normalisation we have a simple numerical procedure for distinguishing spatial from temporal eigenfunctions by computing the variance of $s(t)$.
If the variance of $s(t)$ is zero (or near zero), the eigenfunction is spatial and if the variance of $s(t)$ is 1 (or near 1), the eigenfunction is temporal.
This scheme was used to categorise the spectrum shown in Figure~\ref{fig:ChiSow_gy-mxmx_evals}.

\subsection{Distinguishing coherent flow regimes from mixing regimes}
\label{sssec:fibrenorms}

The (signed) mass of a spatial mode $F$ has to distribute itself over its time fibres $F(t,\cdot)$ because
\[
\|F\|_{L^2(\mathbb{M}_0)}^2=\int_0^\tau \|F(t,\cdot)\|_{L^2(M)}^2\ dt=1.
\]
If the temporal diffusion coefficient $a$ is suitably chosen, we will be able to distinguish temporal regions where coherent dynamics is present or absent using the $L^2$ norms of time fibres of \emph{subdominant} eigenfunctions $F(t,\cdot)$, $t\in[0,\tau]$.
It is important to recall that $\int_M F(t,\cdot)\ d\ell=0$ for a.e.~$t\in[0,\tau]$, which we have by~\eqref{eq:spatial_meanzero} and the fact that subdominant spatial modes are also orthogonal to~$\mathbf{1}_{\set{M}_0}$. This implies that the only way for $F(t,\cdot)$ to be constant on the $t^{\mathrm{th}}$ time fibre is~$F(t,\cdot)\equiv 0_{\{t\}\times M}$.

For $t$ in intervals where coherent dynamics is present, the norm of $F(t,\cdot)$ may be relatively large, with $F(t,\cdot)$ taking large positive (say) values in the coherent region in space and negative values in the complement of the coherent region.
Within each coherent region, $F(t,\cdot)$ should be approximately constant to achieve small values of $\|\|\nabla_{g_t}F(t,\cdot)\|_{g_t}\|_{L^2(M)}$.
On the other hand, during periods of intense global mixing in space, for sufficiently large $a$ it is likely that $\|F(t,\cdot)\|_{L^2(M)}^2$ will be small.
This is because the metric $g_t$ is rapidly varying in time and in order to achieve a minimal eigenvalue in the variational characterisation of eigenvalues (i.e.\ low values of $\|\|\nabla_{g_t}F(t,\cdot)\|_{g_t}\|_{L^2(M)}$) the eigenfunction $F$ should also be rapidly varying in time to adapt to~$g_t$.
In opposition to this effect, if $a$ is large enough, rapid variation of $F$ in time will be costly in the temporal direction (i.e.\ large values of $|\partial_t F|$).
The way out is for $F(t,\cdot)$ to be constant (i.e.\ zero) when strong globally mixing is present.
This pushes
the (signed) mass of $F$ onto the most coherent time fibres and minimises the $L^2$ norm on strongly mixing time fibres; see Figure~\ref{fig:ChiSow_gy-mxmx_slicenorms_vs_a}.
Of course, the above analysis is strictly for spatial eigenfunctions $F$ because the $L^2$ norms of time-fibres of temporal eigenfunctions vary dramatically in time.
\begin{center}
\begin{minipage}{0.8\textwidth}
\emph{In summary, as a basic indicator to discriminate between coherent vs mixing regimes we use the relative values of the $L^2$ norms of the time fibres of dominant spatial eigenfunctions.}	
\end{minipage}
\end{center}
This intuition is further formalized section~\ref{sec:characterisation}.

\begin{figure}[htb]
\centering
\begin{subfigure}{0.45\textwidth}
\includegraphics[width = \textwidth]{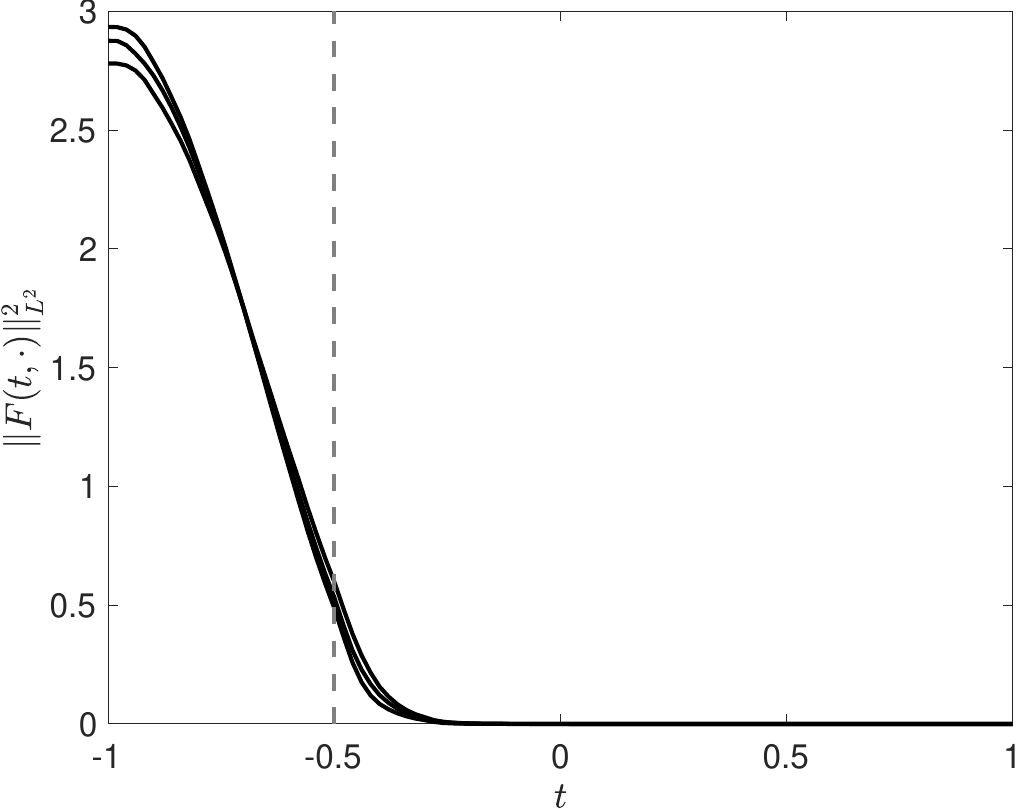}
\subcaption{}
\end{subfigure}
\begin{subfigure}{0.45\textwidth}
\includegraphics[width=\textwidth]{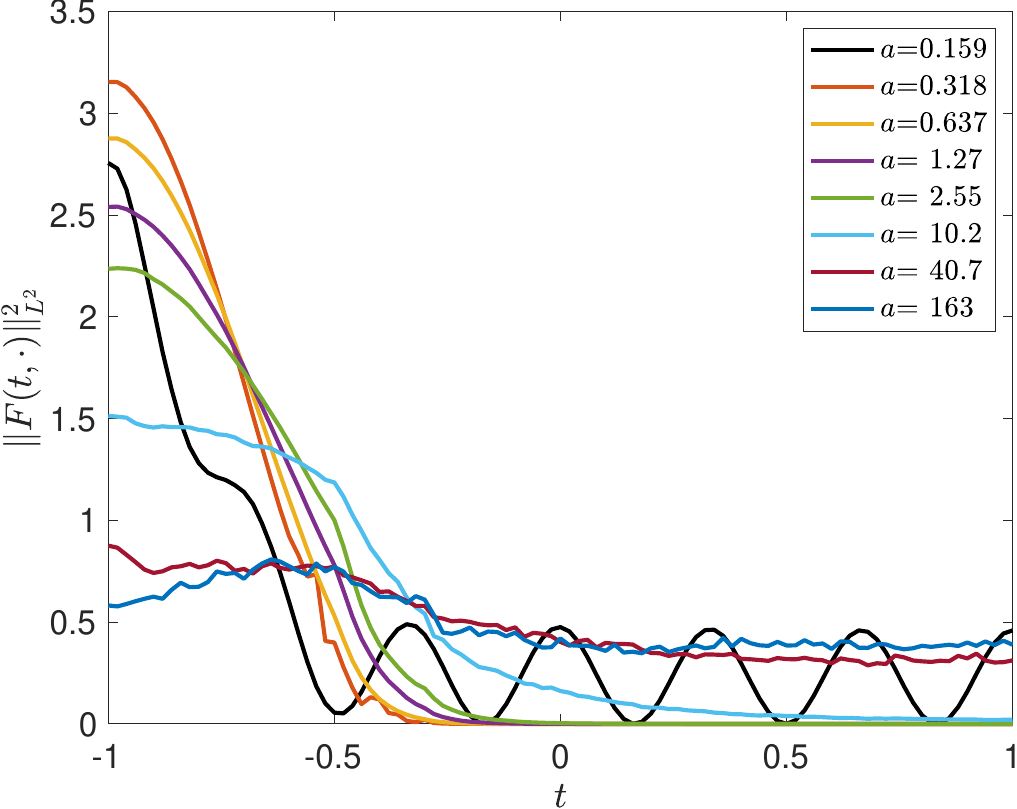}
\subcaption{}
\end{subfigure}
\caption{Slicewise squared $L^2$ norms of subdominant spatial eigenfunctions for the partially coherent Childress--Soward system, which has a coherent regime for~$t\in [-1,-0.5]$ and is mixing for $t\in [-0.5,1]$. (a) The function $t \mapsto \|F(t,\cdot)\|_{L^2}^2$ for the first three subdominant spatial modes~$F$ for $a= \frac{2}{\pi}$, having the eigenvalues~$\Lambda_k^{\mathrm{spat}} = -3.5517, -3.7559, -3.9847$, $k=2,3,4$. (b) Slicewise squared $L^2$ norms of the $4^{\mathrm{th}}$ spatial eigenmode $\Lambda_4^{\mathrm{spat}}$ of $\Delta_{\Goa}$ for several different choices of the temporal diffusion strength~$a$. We observe that the best distinction between the coherent and mixing regimes is obtained between $a = \frac{1}{\pi} \approx 0.32$ and~$a = \frac{4}{\pi} \approx 1.27$ (recall our heuristic from subsection \ref{sssec:choosing_a} suggested increasing $a$ from the value $a_{\min}=\frac{2}{\pi}$). For $a \gtrsim \frac{8}{\pi}$, the temporal variation of $F$ is too small and for $a \lesssim \frac{1}{2\pi}$ it is too large.}
\label{fig:ChiSow_gy-mxmx_slicenorms_vs_a}
\end{figure}

\subsection{A posteriori considerations regarding $a$}
\label{ssec:a_timescale}

The previous three subsections are sufficient to select a reasonable value for the parameter $a$, to separate temporal and spatial eigenfunctions, and to use the spatial eigenfunctions to find regimes of coherent behaviour.
In this final subsection we take a closer look at the relationship between $a$ and where in the spectrum a coherent set of a certain spatial regularity and temporal duration might be expected to appear.
We will do this by comparing the dynamic Cheeger constant of a specified finite-time coherent set with the Cheeger constants of sets extracted from level sets of temporal eigenfunctions.

Consider a set $A\subset M$ that remains coherent in the time interval $0<\tau_1<\tau_2<\tau$.
By volume-preservation of the dynamics we write $\ell$ for the volume on $M$ and later also for the volume on each time slice $\{t\}\times M$.
Assuming that $\ell(A)\le \ell(M)/2$ and following (\ref{dynche}), but without taking the infimum, the dynamic Cheeger constant of the disconnector $\Gamma=\partial A$ on the interval $[\tau_1,\tau_2]$ is:
\[
h^D(\partial A) = \frac{\frac{1}{\tau_2-\tau_1}\int_{\tau_1}^{\tau_2} V_{i^*g_t}(\partial A)\,dt}{\ell(A)},
\]
where $i:A\hookrightarrow M$ is the inclusion map.
The set $A$ naturally defines a space-time set of the form~$\set{A} = [\tau_1,\tau_2]\times A \subset \set{M}_0$.
Because $\ell(A)\le \ell(M)/2$, we have $V_{\Goa}(\set{A})\le V_{\Goa}(\set{M}_0)/2$, and thus the Cheeger constant of $\bbGamma = \partial\set{A}$ is:
\begin{equation}
	\label{eq:Cheeger_bbA}
	H_a(\partial\set{A}) = \frac{2\ell(A) + \frac{1}{a}\int_{\tau_1}^{\tau_2}V_{i^*g_t}(\partial A) \,dt}{(\tau_2-\tau_1)\ell(A)/a} = \frac{2a}{\tau_2-\tau_1} + h^D(\partial{A}).
\end{equation}

We now turn to the temporal eigenfunctions.
We wish to construct a superlevel set from the $k^{\mathrm{th}}$ temporal eigenfunction $F^\temp_k(t,x)=\cos(k\pi t/\tau)$ that has least Cheeger constant;  this will occur for the level set at 0.
We therefore define the superlevel set $\set{A}_k^\temp:=\{(t,x)\in\set{M}_0:F^\temp_k(t,x)>0\}$.
The boundary $\partial\set{A}_k^\temp = (F^\temp_k)^{-1}(0)$, consists of $k$ copies of~$M$. The associated spatiotemporal Cheeger constant is
\[
H_a(\partial\set{A}_k^\temp) = \frac{k\ell(M)}{(\tau/2)\ell(M) / a)} = \frac{2ka}{\tau}.
\]
We now wish to compare $H_a(\partial\set{A})$ with $H_a(\partial\set{A}_k^\temp)$ and so we equate these two values.
Solving the resulting equality for $k$ yields
\begin{equation}
	\label{eq:a_bound}
k=\tau\left(\frac{1}{\tau_2-\tau_1}+\frac{h^D(\partial A)}{2a}\right).
\end{equation}
We note a few points for fixed $a$.
\begin{itemize}
\item Coherent sets with shorter duration will tend to appear further down the spectrum because the term $\frac{\tau}{\tau_2-\tau_1}$ increases with shrinking duration~$\tau_2-\tau_1$.
\item A coherent set with a larger dynamic Cheeger constant on $[\tau_1,\tau_2]$ will appear further down the spectrum.
\item Because eigenfunctions with larger index $k$ are more difficult to accurately estimate numerically, we see from \eqref{eq:a_bound} that it is better to choose $\tau$ not too much larger than~$\tau_2-\tau_1$.
\end{itemize}
If we consider varying $a$:
\begin{itemize}
\item When $a$ is small it is predominantly the dynamic Cheeger constant that governs where the set appears in the spectrum, and when $a$ is larger, the temporal duration of the coherence is the important factor.
\item Rearranging (\ref{eq:a_bound}) to solve for $a$, we obtain
$$a=\frac{h^D(\partial A)}{2}\left(\frac{k}{\tau}-\frac{1}{\tau_2-\tau_1}\right)^{-1}.$$
This provides a rough indication of a choice of $a$ to pick up the coherent set $A$ in a spatial eigenfunction appearing approximately nearby the $k^{\mathrm{th}}$ temporal eigenfunction in the eigenvalue ordering;  note that smaller duration $\tau_2-\tau_1$ will force larger $k$ to maintain positivity of the second term above.
\end{itemize}

For the Childress--Soward flow from section~\ref{ssec:ChiSow_gy-mxmx} let us consider the second spatial mode, which identifies two vortices next to each other as a coherent set $A$;  see Figure \ref{fig:ChiSow_gy-mxmx_EVs}(a).
This set has perimeter $4\pi$ because the domain is periodic, and area $2\pi^2$; therefore $h^D(\partial A) = 2/\pi$.
Further, $\tau=2$ and $\tau_2-\tau_1=1/2$, and we note that the boundary of $\partial\set{A}$ at $t=-1$ does not enter the Cheeger constant calculations because of the Neumann boundary conditions on the temporal faces of $\set{M}_0$.
Thus, $H_a(\bbGamma) = \frac{a}{\tau_2 - \tau_1} + h^D(\partial A)$ in \eqref{eq:Cheeger_bbA} and in~\eqref{eq:a_bound} we replace $\tau_2-\tau_1$ by~$2(\tau_2-\tau_1)$.
The equality \eqref{eq:a_bound} becomes $k = 2+\frac{2}{a\pi}$. This can be satisfied for $k\ge 3$, and with $k=3$ it yields~$a = 2/\pi$. This is the same value the heuristic from section \ref{sssec:choosing_a} suggested.

\section{A one-dimensional surrogate model}
\label{sec:characterisation}

In this section we construct a reduced one-dimensional eigenproblem from the inflated dynamic Laplace eigenproblem by integrating out the spatial dynamics.
The analysis of this reduced problem further formalizes our intuition from the previous section on how to deduce regimes of coherence and mixing.

\subsection{Derivation of a surrogate 1D model}
\label{ssec:surr_deriv}

Let $\Delta_{\Goa}F=\Lambda F$ with $\|F\|_{L^2(\set{M}_0)}=1$. Using (\ref{eq:LBab}), we multiply both sides by $F$ and integrate over the $t^{\mathrm{th}}$ fibre $\{t\}\times M$:
\begin{align*}
\Lambda\,\|F(t,\cdot)\|_{L^2(M)}^2 = \Lambda\int_{M}F^2\ d\ell &= a^2\int_M(\partial_{tt}F)\cdot F\ d\ell+\int_M\Delta_{g_t}F\cdot F\ d\ell\\
&= a^2\int_M \frac12 \partial_{tt}(F^2)-(\partial_t F)^2\ d\ell-\int_M \|\nabla_{g_t}F\|^2_{g_t}\ d\ell \\
&= \frac{a^2}{2}\partial_{tt}\int_M F^2\ d\ell-a^2\int_M(\partial_t F)^2\ d\ell - \int_M \|\nabla_{g_t}F\|^2_{g_t}\ d\ell
\end{align*}
which we write, using $u(t)=\|F(t,\cdot)\|_{L^2(M)}^2$, as
\begin{equation*}
	\frac{a^2}{2}u''-\left(a^2\frac{\int_M(\partial_t F)^2\ d\ell}{\|F(t,\cdot)\|_{L^2(M)}^2} + \frac{\int_M \|\nabla_{g_t}F\|^2_{g_t}}{\|F(t,\cdot)\|_{L^2(M)}^2}\right) u = \Lambda u.
\end{equation*}

Making the obvious substitutions for the (time-dependent) temporal and spatial Rayleigh-type coefficients, we write this as
\begin{equation}
\label{eq:surrogate}
\frac{a^2}{2}u''(t)=\left(\Lambda+\left[a^2\rho^\temp(t)+\rho^\spat(t)\right]\right)u(t),\quad\mbox{for $t\in (0,\tau)$}, \quad u'(0)=u'(\tau)=0.
\end{equation}
Our reduced equation (\ref{eq:surrogate}) describes the expected behaviour of $u(t)=\int_M F(t,\cdot)^2\ d\ell$, the square of the spatial norm of the eigenfunction $F$ on the $t^{\mathrm{th}}$ time slice.
On the $t^{\mathrm{th}}$ time fibre, the decay experienced due to the irregularity of $F$ is $a^2\rho^\temp(t)+\rho^\spat(t)$.
We interpret $\Lambda$ as the average space-time decay that the eigenfunction $F$ experiences on all of $\set{M}_0$.
Indeed, by the variational form~\eqref{eq:LB-lamk} we have that
\[
-\Lambda = \int_0^{\tau} \left(  a^2\rho^\temp(t)+\rho^\spat(t)\right) u(t)\, dt,
\]
and $\int_0^{\tau} u(t)\,dt=1$ by our choice of normalization of~$F$.

Recall that in section \ref{sssec:fibrenorms} we used the relative size of the fibre norms $u(t)$ to distinguish coherent flow regimes from incoherent ones.
There are two fundamental regimes:
\begin{enumerate}
\item $t\in[0,\tau]$ for which $\Lambda+(a^2\rho^\temp(t)+\rho^\spat(t))<0$.
    For such $t$ the \emph{local decay is less than the average decay}, indicative of $F$ encoding relatively \emph{coherent dynamics}.
    Because $u''<0$ and $u>0$, $u$ has a local maximum.  In other words there is a local peak in the norm of $\|F(t,\cdot)\|_{L^2(M)}$, consistent with the discussion in section \ref{sssec:fibrenorms}.
\item $t\in[0,\tau]$ for which $\Lambda+(a^2\rho^\temp(t)+\rho^\spat(t))>0$.
    For such $t$ the \emph{local decay is greater than the average decay}, indicative of $F$ encoding \emph{relatively mixing dynamics}.
      Because $u''>0$ and $u>0$, $u$ (and therefore $\|F(t,\cdot)\|_{L^2(M)}$) has a local minimum, consistent with the discussion in section \ref{sssec:fibrenorms}.
\end{enumerate}

The above two regimes partition $[0,\tau]$ into time intervals where the eigenfunction $F$ encodes dynamics that is more coherent or less coherent, respectively, than the average coherence over all of $[0,\tau]$.
One could also define subintervals of $\tau$ with more extreme coherence relative to $F$ by introducing a threshold $c>0$.
For example, the sets $\{t\in[0,\tau]: \Lambda+(a^2\rho^\temp(t)+\rho^\spat(t))<-c\}$ and $\{t\in[0,\tau]: \Lambda+(a^2\rho^\temp(t)+\rho^\spat(t))>c\}$ indicate stronger coherence and stronger mixing, respectively, with increasing $c$.
On the former interval, $\|F(t,\cdot)\|_{L^2(M)}$ has a local maximum and on the latter, $\|F(t,\cdot)\|_{L^2(M)}$ has a local minimum.
In the next subsection we analyse the shape of, and transitions between, these maxima and minima.

\subsection{Analysis of a surrogate 1D model}
\label{ssec:surrogate}

In the previous subsection, the coefficient function $a^2\rho^\temp(t)+\rho^\spat(t)$ arose directly from the eigenfunction~$F$.
We now heuristically investigate replacing this exact coefficient function with a function denoted simply $\rho(t)$, whose form is suggested by properties of the flow, in an effort to infer something about $F$.
To this end, we write a schematic version of (\ref{eq:surrogate}),
\begin{equation}
\label{eq:surrogate2}
\frac{a^2}{2}u''(t)-\rho(t)u(t) = \nu\, u(t),\quad\mbox{for $t\in (0,\tau)$}, \quad u'(0)=u'(\tau)=0,
\end{equation}
where $\rho(t)\ge 0$  is meant to describe the ``relative mixing strength'' (larger $\rho$, greater mixing) that the flow inflicts on the supposed unknown function~$F$ on the $t^{\mathrm{th}}$ time fibre.

We assume that all we know in \eqref{eq:surrogate2} is $a$ and~$\rho$, and so this equation amounts to a \emph{Sturm--Liouville eigenproblem}.
By the theory of Sturm--Liouville eigenproblems~\cite[\S5.3--\S5.4, pp. 153 and 164, and Thm.~5.17]{Teschl2012}, if $\rho$ is integrable, \eqref{eq:surrogate2} has a countable spectrum of distinct eigenvalues $0 \ge \nu_0 > \nu_1 > \cdots$ all having multiplicity one, and the associated (up to constant scaling unique) eigenfunctions~$u_i$ are mutually orthogonal in $L^2([0,\tau])$ and have exactly $i$ zeros,~$i\ge 0$.
Since $u$ models the squared norm of time slices of $F$, only solutions~$u_i\ge 0$ are of interest, which leaves $u_0$ as the unique meaningful solution.
We note that if we have an eigenfunction $F$, applying the above remarks to (\ref{eq:surrogate}), which we obtain by substituting $\rho =  a^2\rho^\temp+\rho^\spat$ and $\nu=\Lambda$ into~\eqref{eq:surrogate2}, shows that the solution $u(t)=\|F(t,\cdot)\|_{L^2(M)}^2$ is the unique solution.

Returning to our heuristic discussion, for the Childress--Soward flow, from the discussion in section \ref{sssec:fibrenorms} and Figure \ref{fig:ChiSow_gy-mxmx_slicenorms_vs_a}, we expect $\rho(t)$ arising from eigenfunctions $F$ that highlight the coherent sets from time $-1$ to $-0.5$ to be small until the mixing regime begins at $t=-0.5$, after which $\rho(t)$ should rise to a much larger value.
A simple approximation of such a $\rho$ is a step function with two values $Z\gg z>0$ in the coherent and mixing regimes, respectively.
This step function form of $\rho(t)$ permits finer analysis of the surrogate model~\eqref{eq:surrogate2}.

For simplicity, in the following we set~$\tau=1$.
It is straightforward to compute a one-to-one correspondence between the following two homogeneous Neumann boundary value problems, one on $[0,\tau]$, and one on~$[0,1]$:
\begin{equation}
\label{eq:surrogate_scaling}
\frac{\tilde{a}^2}{2} \tilde{u}'' - \tilde{\rho} \tilde{u} = \nu\, \tilde{u} \text{ on } (0,\tau) \quad \raisebox{-15pt}{\tikz \draw[<->, thick] (-2,0) -- (2,0) node[below, pos=0.5] {\scriptsize  $\tilde{a} = a\tau, \, \tilde{\rho}(t) := \rho(t/\tau)$} node[above, pos=0.5] {\scriptsize $\tilde{u}(t):= u(t/\tau)$};} \quad \frac{a^2}{2} u'' - \rho u = \nu\, u \text{ on } (0,1)	
\end{equation}
The solution of the problem on $[0,\tau]$ is obtained from the solution on $[0,1]$ with scaled variables, if the the temporal diffusion strength is also scaled by~$\tau$.

As in the Childress--Soward flow in Section \ref{sec:examples} we assume that the velocity field of the system is such that there is coherence in the first $0<p<1$ fraction of the time interval, and then there is strong mixing.
Of course, the coherent regime could be located wherever in the time interval, our choice is merely for simplicity.
We set
\begin{equation}
	\label{eq:rho}
	\rho(t) = \left\{
\begin{array}{ll}
z, & t\in [0,p],\\
Z, & t\in (p,1],
\end{array}
\right.
\end{equation}
with $Z \gg z > 0$.
We recognise that replacing the a priori unknown coefficient function $\rho$ in~\eqref{eq:surrogate} by a $\rho$ taking only two values is a strong simplification.
The solutions of~\eqref{eq:surrogate2} can now be determined analytically, and the numerical results in Figure~\ref{fig:surrogate_example}(a) and Figure~\ref{fig:ChiSow_gy-mxmx_EV4} show that the profile of $u$ predicted by the surrogate model with this idealised $\rho$ is surprisingly accurate.
\begin{proposition} \label{prop:int_cohmix}
The solutions to \eqref{eq:surrogate2} with mixing rate function $\rho$ as in~\eqref{eq:rho} are
\begin{equation*}
u(t) = \left\{
\begin{array}{ll}
\alpha \cosh ( \omega_z t), & t\in [0,p],\\
\alpha \frac{\cosh (\omega_z p)}{\cosh(\omega_Z(1-p))}\cosh (\omega_Z (1-t)), & t\in [p,1],
\end{array}
\right.
\end{equation*}
where $\alpha$ is an arbitrary scaling constant, $\omega_z = \omega_z(\nu) = \frac{\sqrt{2(\nu+z)}}{a}$, $\omega_Z = \omega_Z(\nu) = \frac{\sqrt{2(\nu+Z)}}{a}$, and $\nu$ is the eigenvalue satisfying
\begin{equation} \label{eq:int_cohmix_eval}
f(\nu): = \frac{\omega_z(\nu)}{\omega_Z(\nu)} \tanh (\omega_z(\nu)\, p) + \tanh (\omega_Z(\nu)\, (1-p)) = 0.
\end{equation}
\end{proposition}
\begin{proof}
See Appendix \ref{app:proof_int_cohmix}.
\end{proof}

Recall from above that we are only interested in the dominant mode of~\eqref{eq:surrogate2}, i.e., the eigenfunction associated with~$\nu_0$, the largest eigenvalue. Properties of the eigenproblem are discussed in Appendix~\ref{app:int_cohmix_evals}.
In particular, we show that $\nu_0 \in (-Z,-z)$, and this implies $\omega_z(\nu_0) \in \mathrm{i}\R$ and $\omega_Z(\nu_0) \in \R$. Consequently, for the associated eigenfunction~$u$ we obtain a cosine-profile on the first (coherent) part $[0,p]$ of the time interval because $\cosh(\omega_z t) = \cos(|\omega_z|t)$, and an exponentially decaying cosh-profile on the second (mixing) part~$[p,1]$.

The differences in magnitude of~$u$ will, ideally, be indicative for the difference between the regimes. In particular, we expect strong exponential decay of $u$ in a strongly mixing regime. Nevertheless, ambiguity for the intermediate times when the system is shifting from coherent to mixing, is still expected, as for most (realistic) systems this is a continuous and not an abrupt transition.

We show the dominant eigenfunction $u$ from Proposition~\ref{prop:int_cohmix} for parameters $a^2/2 \in \{ 1/\pi^2, 1/100\pi^2\}$, $p=0.25$, $z=2$, $Z=40$ in Figure~\ref{fig:surrogate_example}(a).
\begin{figure}[htb]
\centering

\begin{subfigure}{0.4\textwidth}
\includegraphics[width=\textwidth]{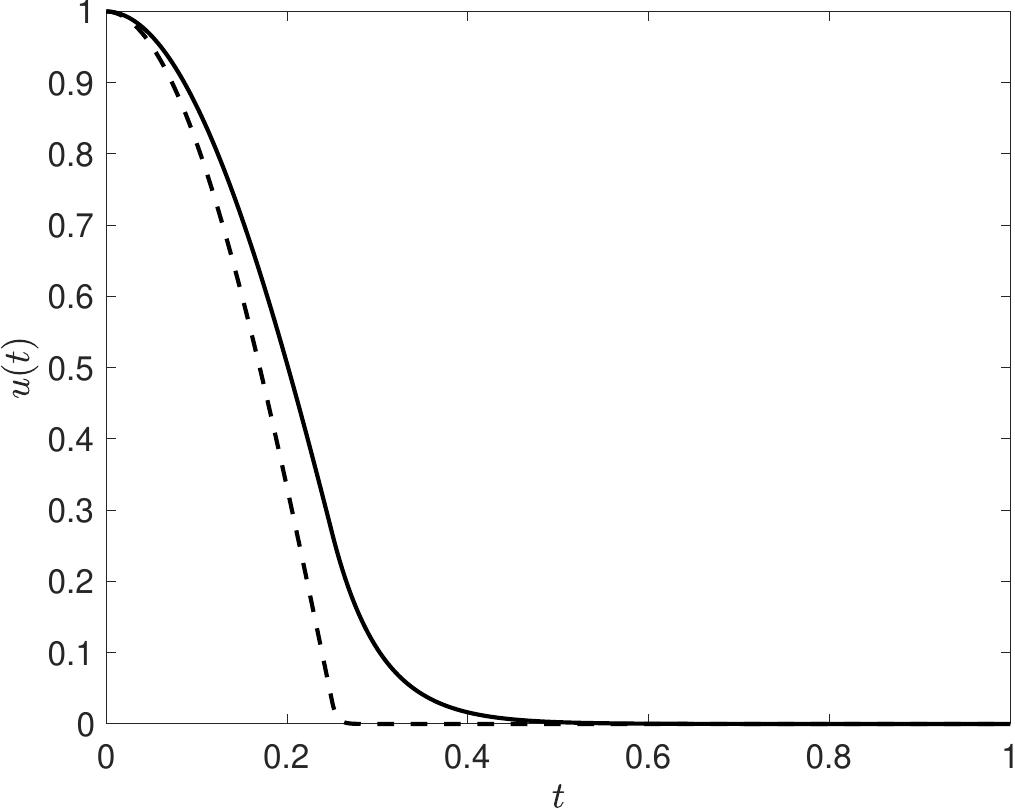}
\subcaption{}
\end{subfigure}
\hfill
\begin{subfigure}{0.465\textwidth}
\includegraphics[width=\textwidth]{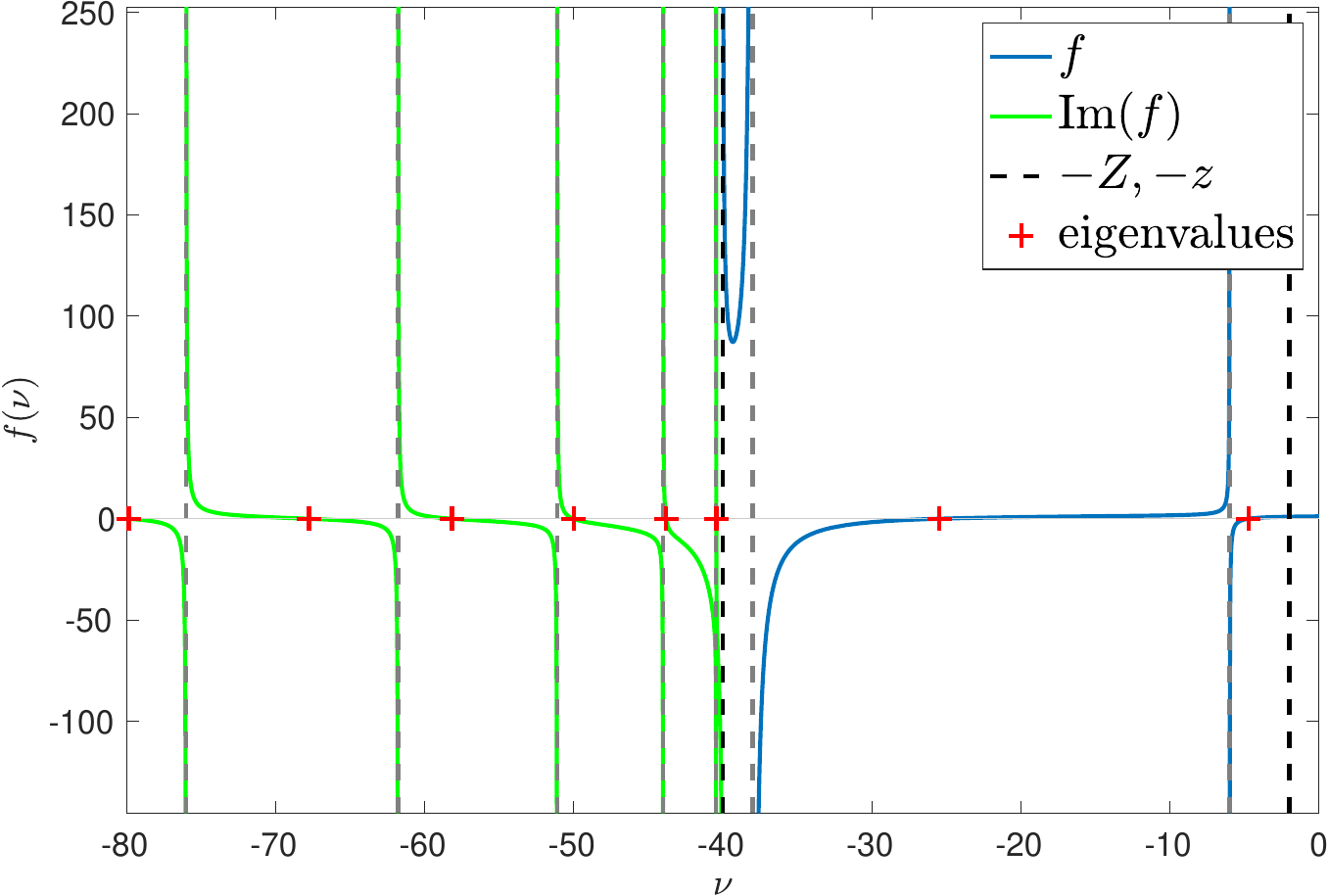}
\subcaption{}
\end{subfigure}

\caption{(a) Example solution of~\eqref{eq:surrogate2}, as in Proposition~\ref{prop:int_cohmix} below, for parameters $a^2/2 = 1/\pi^2$, $p=0.25$, $z=2$, $Z=40$ (solid).  This matches well with the full solutions shown in Figure~\ref{fig:ChiSow_gy-mxmx_slicenorms_vs_a};  note that the time axis has been linearly rescaled from $[-1,1]$ to $[0,1]$ in this figure (see (\ref{eq:surrogate_scaling})).  Decreasing $a$ to $a^2/2 = 1/100\pi^2$ (dashed) leads to a sharper transition in the surrogate solution. (b) Eigenvalues of~\eqref{eq:surrogate2} as zeros of~\eqref{eq:int_cohmix_eval} for~$a^2/2 = 1/\pi^2$. The black dashed lines indicate the values~$-Z=-40$, $-z=-2$. Gray dashed lines indicate singularities of~$f$.}
\label{fig:surrogate_example}
\end{figure}

We find that the decay of $u$ is increasingly rapid on~$[p,1]$ as $a$ decreases.
Note that this analysis assumes that $\rho$ is independent of $a$ in~\eqref{eq:surrogate}.

The surrogate model gives a good qualitative approximation, as shown in Figure~\ref{fig:ChiSow_gy-mxmx_EV4}. Subfigure (a) tests the cosine-profile by looking at whether $\smash{ \arcsin\left(\|F(t,\cdot)\|^2_{L^2}\right) }$ vs~$t$ is indeed linear for the coherent regime. We use arcsin instead of arccos to map near zero values to near zero values. Subfigure (b) tests approximately exponential decay of $\|F(t,\cdot)\|^2_{L^2}$ vs~$t$ in the mixing regime.
\begin{figure}[htb]
\centering

\begin{subfigure}{0.45\textwidth}
\includegraphics[width=\textwidth]{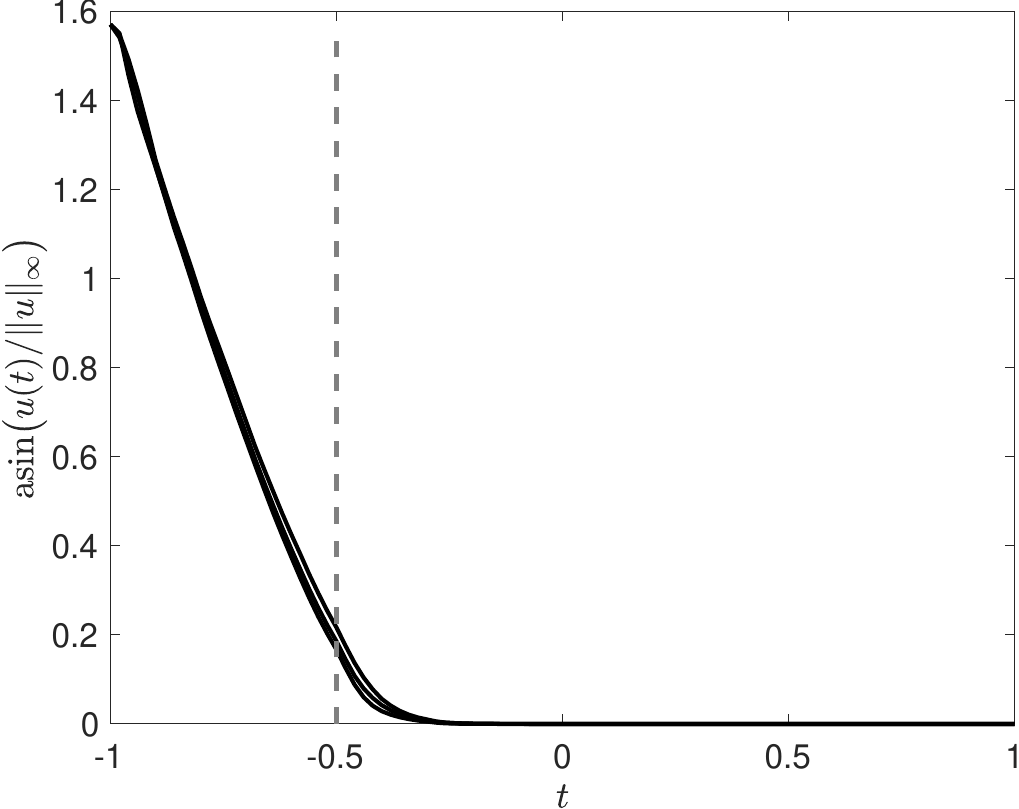}
\subcaption{}
\end{subfigure}
\hfill
\begin{subfigure}{0.45\textwidth}
\includegraphics[width=\textwidth]{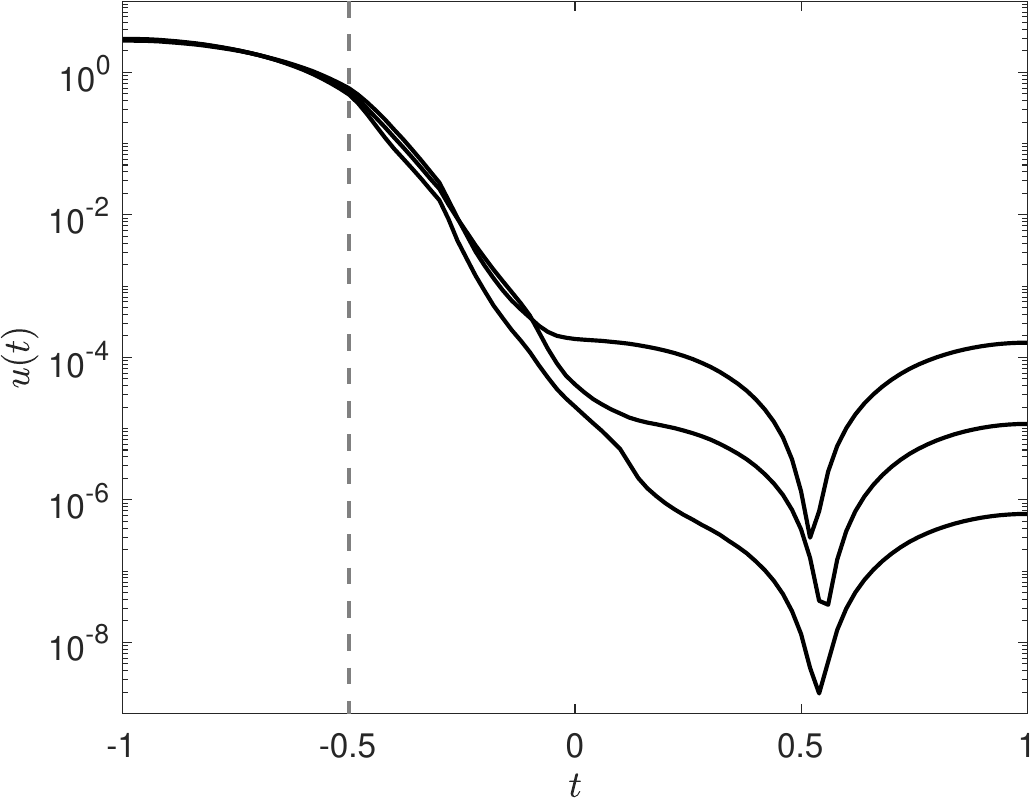}
\subcaption{}
\end{subfigure}

\caption{Analysis of the function $\smash{ u: t \mapsto \|F(t,\cdot)\|_{L^2(M)}^2 }$, for the first three spatial eigenmodes $F^\spat_k$, $k=2,3,4$, of~$\Delta_{\Goa}$ for the partially coherent Childress--Soward system with~$a=\frac{2}{\pi}$. (a) $t\mapsto \arcsin(u(t)/ \|u\|_{\infty})$. Dashed lines indicate the time $t=-0.5$, where the mixing regime starts. The approximately constant slope on $[-1,-0.5]$ indicates that $u$ is approximately a cosine on this interval where the velocity field is in the coherent regime. (b) The function $u$ depicted on a semilogarithmic scale. The strong exponential drop in the graph over $[-0.5,-1]$ indicates the onset of mixing. Compare also with Figure~\ref{fig:ChiSow_gy-mxmx_slicenorms_vs_a}.}
\label{fig:ChiSow_gy-mxmx_EV4}
\end{figure}

\section{FEM-based numerical discretisation of the inflated dynamic Laplacian}
\label{sec:FEMDL}

In this section we describe how to numerically approximate $\Delta_{\Goa}$ in (\ref{eq:LBab}).
We adopt and extend the approach of \cite{FJ18}, which derived a finite-element method discretisation of the dynamic Laplacian eigenproblem.
For $f:M\to\mathbb{R}$ we define the pushforward and pullback of $f$ under $\phi_t$ by $(\phi_t)_*f:=f\circ (\phi_t)^{-1}$ and $\phi_t^*f:=f\circ \phi_t$, respectively;  recall that~$g_t=\phi_t^*e$.
The right-hand side of (\ref{eq:LBab}) has two main components and we begin by discussing~$\Delta_{g_t}$.
We note that one has the alternative representation \cite{Fro15}, $\Delta_{g_t}= \phi_t^*\circ\Delta_{(\phi_t (M),e)}\circ(\phi_t)_*$,
where we have made explicit the fact that the Laplace operator on the right-hand side is acting on the ``future'' Riemannian manifold~$(\phi_t(M),e)$.
Similarly for $F\in H^1(\set{M}_0)$ (resp.\ $F\in H^1(\set{M}_1)$) we have $\Phi_*F=F\circ\Phi^{-1}\in H^1(\set{M}_1)$ (resp.\ $\Phi^*F=F\circ\Phi\in H^1(\set{M}_0)$).
We are now ready to construct a weak-form approximation of the eigenproblem $\Delta_{\Goa} F = \Lambda F$, where for the moment we assume homogeneous Neumann conditions on~$M$ ($M$ may also be boundaryless).
Neumann boundary conditions \cite{Fro15,FJ18} on $M$ allow us to find finite-time coherent sets that may share a boundary with $M$, while Dirichlet boundary conditions \cite{FJ18} on $M$ force the finite-time coherent sets to have boundaries away from the boundary of~$M$.
In the weak form, the only change required to solve the Dirichlet case is to change $H^1(\set{M}_0)$ to $H^1_0(\set{M}_0)$.
For $F,\tilde F\in H^1(\set{M}_0)$, multiplying our eigenproblem by $\tilde F$ and integrating both sides we obtain
\begin{equation}
\label{weakeig}
\int_0^\tau\int_M \Delta_{\Goa}F\cdot \tilde F\ d\ell\ ds=\Lambda\int_0^\tau\int_M F\cdot \tilde{F}\ d\ell\ ds.
\end{equation}
The left-hand side of (\ref{weakeig}) is
\begin{eqnarray}
\nonumber\lefteqn{ \hspace*{-5em}
\int_0^\tau \int_M \left(a^2\partial_{tt}|_{t=s}F(\cdot,x)\right)\tilde{F}(\cdot,x)\ d\ell(x)\, ds + \int_0^\tau \int_M \left(\Delta_{g_s}F(s,\cdot)\right) \tilde{F}(\cdot,s)\ d\ell\, ds}\\
\nonumber &=& \int_0^\tau \int_M \left(a^2\partial_{tt}|_{t=s}F(\cdot,x)\right)\tilde{F}(\cdot,x)\ d\ell(x)\, ds \\
\nonumber & & + \int_0^\tau \int_{\phi_s(M)}\left(\Delta_{(\phi_s(M),e)}\circ(\Phi_*F)(s,\cdot)\right)\cdot(\Phi_*\tilde{F})(s,\cdot)\ d\ell\, ds\\
\label{femhalfleft1}&=&-a^2\int_0^\tau \int_M \left(\partial_{t}|_{t=s}F(\cdot,x)\right)\cdot\left( \partial_{t}|_{t=s}\tilde{F}(\cdot,x)\right) \ d\ell(x)\, ds \\
\label{femhalfleft2} & & - \int_0^\tau \int_{\phi_s(M)} \nabla_x(\Phi_*F)\cdot \nabla_x(\Phi_*\tilde{F})\ d\ell\, ds.
\end{eqnarray}
Similarly, the right-hand side of (\ref{weakeig}) is
\begin{equation}
\label{femhalfright}
\Lambda\int_0^\tau \int_{\phi_s(M)} \Phi_*F\cdot \Phi_*\tilde{F}\ d\ell\ ds.
\end{equation}
Because of the differing roles of time and space we assume that our approximating basis $V\subset H^1([0,\tau]\times M)$ contains functions of the form $\xi(t)\eta(x)$, $t\in[0,\tau], x\in M$.
This ansatz enables a convenient decomposition across time, where at each time fibre we can leverage the spatial constructions from \cite{FJ18}.
It also allows for simple adjustment of the parameter $a$, without having to recompute any integrals.
We suppose our approximation space is built in this way using a finite number of basis elements $\xi_i$, $i=0,\ldots,T$ and $\eta_k$,~$k=1,\ldots,N$.
Inserting $F(t,x) = \xi_i(t)\eta_k(x), \tilde{F}(t,x) = \xi_j(t) \eta_l(x)$ into (\ref{femhalfleft1}) and \eqref{femhalfleft2}, we have
\begin{eqnarray}
\nonumber \eqref{femhalfleft1} + \eqref{femhalfleft2} &=& -a^2\int_0^\tau \int_M \eta_k\,\partial_{t}|_{t=s} \xi_i\cdot \eta_l\,\partial_{t}|_{t=s}\xi_j \ d\ell\, ds\\
	& & - \int_0^\tau \int_{\phi_s(M)} \xi_i(s)\nabla_x((\phi_s)_*\eta_k)\cdot \xi_j(s)\nabla_x((\phi_s)_*\eta_l)\ d\ell\, ds \\
\label{femfulltime}&=&a^2\int_0^\tau -\xi'_i(s)\cdot \xi'_j(s) \underbrace{\int_{\phi_s(M)} (\phi_s)_*\eta_k\cdot(\phi_s)_*\eta_l \ d\ell}_{=:M^s_{kl}}\, ds \\
\label{femfullspace}\qquad &&+ \int_0^\tau \!\! \xi_i(s)\xi_j(s) \underbrace{\int_{\phi_s(M)} \!\!\!\!-\nabla_x((\phi_s)_*\eta_k)\cdot \nabla_x((\phi_s)_*\eta_l)\ d\ell}_{=:D^s_{kl}}\, ds=:\mathbf{D}_{ij,kl}
\end{eqnarray}
Similarly, we have
\begin{equation}
\label{femfullright}
\frac{1}{\Lambda} (\ref{femhalfright}) =\int_0^\tau \xi_i(s) \xi_j(s) \underbrace{\int_{\phi_s(M)} ((\phi_s)_*\eta_k)\cdot ((\phi_s)_*\eta_l)\ d\ell}_{=M^s_{kl}}\, ds=:\mathbf{M}_{ij,kl}
\end{equation}
We now fix the $\xi_i$, $i=0,\ldots,T$ and $\eta_k$, $k=1,\ldots,N$ to be the standard piecewise linear one-dimensional and $d$-dimensional hat functions, respectively.
More precisely, we partition $[0,\tau]$ into intervals with endpoints $0=t_0 < t_1 < \cdots < t_T=\tau$ and let $\xi_i$ be the nodal hat function centred at node~$t_i$.
Given vertices $x_k\in M$, we mesh $M$ into simplices and define $\eta_k$ as the nodal hat function with node $x_k$;  this mesh is used to create $M^{t_0}_{kl}$ and~$D^{t_0}_{kl}$.
To define $M^{t_i}$ and $D^{t_i}$, $i=0,\ldots,T$ we refer the reader to \cite{FJ18}.
For each $i=0,\ldots,T$, the matrices $M^{t_i}_{kl}$ are the standard mass matrices from the finite-element method, and the $D^{t_i}_{kl}$ are the modified stiffness matrices discussed in \cite{FJ18}, and can be efficiently computed.

To estimate $\mathbf{D}_{ij,kl}$ and $\mathbf{M}_{ij,kl}$, for $s\in [t_i,t_{i+1}]$ we linearly interpolate to estimate $M^s\approx M^{t_i}+((s-t_i)/(t_{i+1}-t_i))(M^{t_{i+1}}-M^{t_{i}})$, similarly $D^s\approx D^{t_i}+((s-t_i)/(t_{i+1}-t_i))(D^{t_{i+1}}-D^t_i)$.
Using these estimates, all that remains is to analytically evaluate the one-dimensional integrals (\ref{femfulltime})--(\ref{femfullspace}) and (\ref{femfullright}).
We omit the elementary, but lengthy details and present here the resulting formulae, where we specialise to $t_i=i\tau/T$ and set $h=\tau/T$.
Note that because the one-dimensional functions $u_i$ only overlap when $i=j$ or $|i-j|=1$, it is only for these combinations of $i$ and $j$ that we obtain nonzero integrals.

\begin{equation}
\label{boldDeqn}
\renewcommand{\arraystretch}{1.25}
\mathbf{D}_{ij,kl}=\left\{
  \begin{array}{ll}
    -\frac{a^2}{2h}(M^{t_0}_{kl}+M^{t_1}_{kl})+\frac{h}{12}(3D^{t_0}_{kl}+D^{t_{i+1}}_{kl}), & \hbox{$i=j=0$;} \\
    -\frac{a^2}{2h}(M^{t_{i-1}}_{kl}+2M^{t_i}_{kl}+M^{t_{i+1}}_{kl}) & \multirow{2}{*}{$1 \le i=j \le T-1$;} \\
    \phantom{+} + \frac{h}{12}(D^{t_{i-1}}_{kl}+6D^{t_i}_{kl}+D^{t_{i+1}}_{kl}), & \\
    -\frac{a^2}{2h}(M^{t_T}_{kl}+M^{t_{T-1}}_{kl})+\frac{h}{12}(3D^{t_T}_{kl}+D^{t_{T-1}}_{kl}), & \hbox{$i=j=T$;} \\
    \frac{a^2}{2h}(M^{t_i}_{kl}+M^{t_{i+1}}_{kl})+\frac{h}{12}(D^{t_i}_{kl}+D^{t_{i+1}}_{kl}), & \hbox{$j=i+1, i\le T-1$.}
  \end{array}
\right.
\end{equation}
and
\begin{equation}
\label{boldMeqn}
\renewcommand{\arraystretch}{1.25}
\mathbf{M}_{ij,kl}=\left\{
  \begin{array}{ll}
    \frac{h}{12}(3M^{t_0}_{kl}+M^{t_{i+1}}_{kl}), & \hbox{$i=j=0$;} \\
    \frac{h}{12}(M^{t_{i-1}}_{kl}+6M^{t_i}_{kl}+M^{t_{i+1}}_{kl}), & \hbox{$1 \le i=j \le T-1$;} \\
    \frac{h}{12}(3M^{t_T}_{kl}+M^{t_{T-1}}_{kl}), & \hbox{$i=j=T$;} \\
    \frac{h}{12}(M^{t_i}_{kl}+M^{t_{i+1}}_{kl}), & \hbox{$j=i+1, i\le T-1$.}
  \end{array}
\right.
\end{equation}
The values for $j=i-1, i=1,\ldots,T$ are identical to the values for $j=i+1$ by symmetry.
Thus, numerically we solve the sparse, symmetric eigenproblem $\mathbf{D}\mathbf{w}=\lambda\mathbf{M}\mathbf{w}$, where $\mathbf{w}\in\mathbb{R}^{(T+1) N}$.
An approximate eigenfunction $F$ is then reconstructed as $F(t,x)=\sum_{i=0}^T\sum_{k=1}^N \mathbf{w}_{i,k}\xi_i(t)\eta_k(x)$.

\section{Numerical example}
\label{sec:examples}

\subsection{The Childress--Soward ``cat's-eye'' flow}
\label{ssec:ChiSow}

We consider the two-dimensional velocity field~\cite{ChSo89} $v:\mathbb{T}^2\to\mathbb{R}^2$, parameterised by $A>0$ and $-1\le r\le 1$:
\begin{equation} \label{eq:ChiSow}
v = A\cdot \left(\frac{\partial \psi}{\partial y}, -\frac{\partial \psi}{\partial x}\right), \mbox{ with streamfunction } \psi(x,y) = \sin x \sin y + r \cos x \cos y,
\end{equation}
where $\mathbb{T}^2$ is identified with~$2\pi S^1\times 2\pi S^1$, and $S^1$ is the circle of circumference 1.
For $|r| \approx 1$ the flow is a diagonal shear and for $r \approx 0$ the flow has four vortices; otherwise it possesses an intermediate ``cat's-eye'' structure.
Figure~\ref{fig:ChiSow_streamfcn} shows streamfunctions of the flow for different values of~$r$.

\begin{figure}[htb]
\centering

\includegraphics[width = 0.24\textwidth]{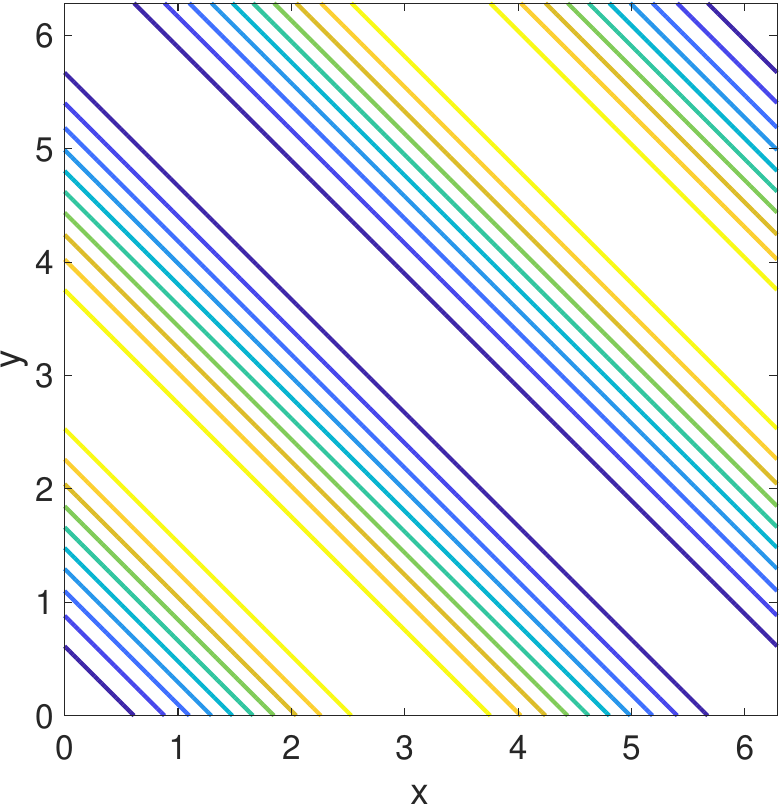}
\hfill
\includegraphics[width = 0.24\textwidth]{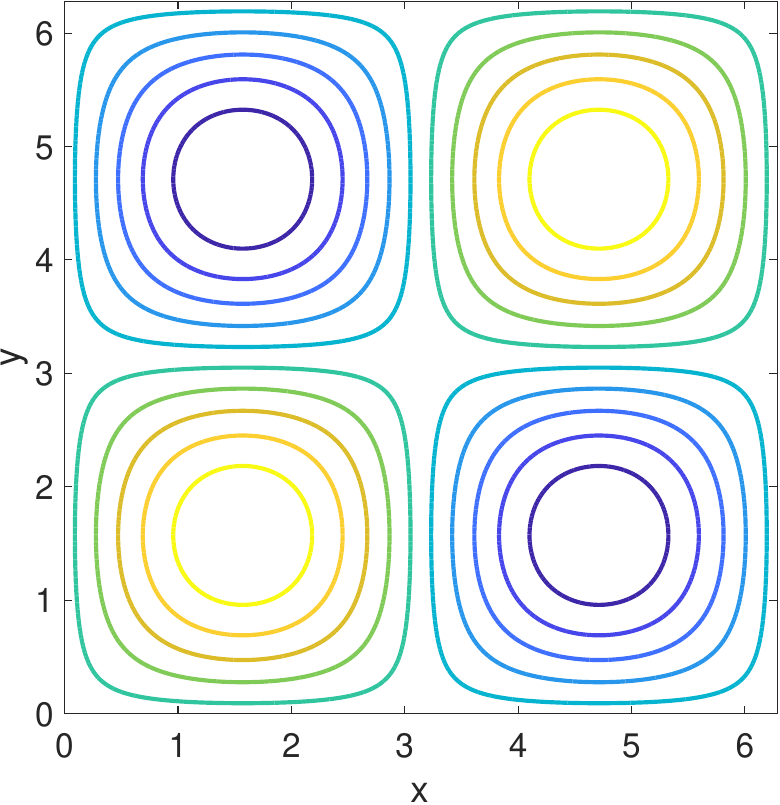}
\hfill
\includegraphics[width = 0.24\textwidth]{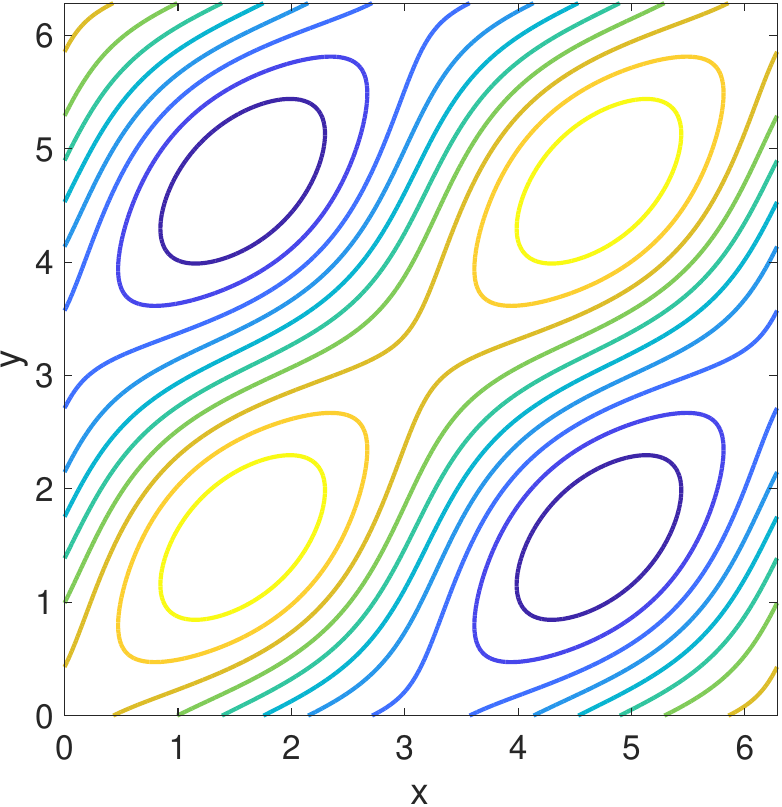}
\hfill
\includegraphics[width = 0.24\textwidth]{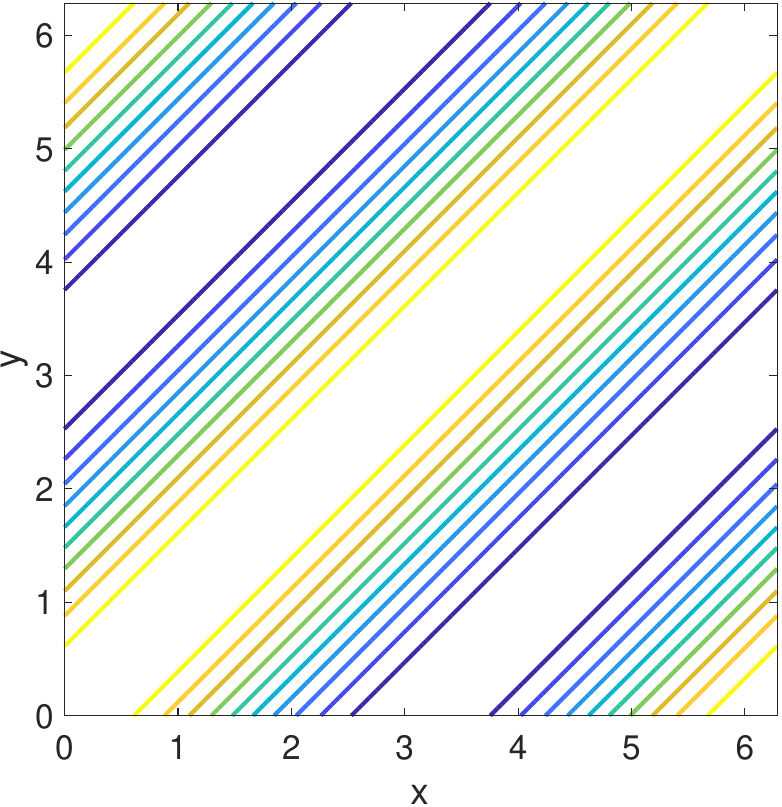}

\caption{Streamfunction contour lines of the Childress--Soward flow~\eqref{eq:ChiSow} for $r = -1,0,0.5,1$, left to right. Light yellow colors indicate larger values. At any point, the velocity field $v$ is tangential to contour lines of the streamfunction.}
\label{fig:ChiSow_streamfcn}
\end{figure}

\subsection{Extracting semi-material coherent sets from a flow exhibiting both coherent and incoherent regimes}
\label{ssec:ChiSow_gy-mxmx}

We consider a non-autonomous Childress--Soward flow~\eqref{eq:ChiSow} with time-dependent amplitude $A(t) = 40\, \mathds{1}_{[-1,-0.5]}(t) + 30\, \mathds{1}_{(-0.5,1]}(t)$ and time-dependent parameter modulation
$$
r(t) = \left\{
\begin{array}{ll}
0, & t\in [-1,-0.5], \\
\mathrm{tanh}\left( 100 \cos\left( 5\pi t \right) \right), & t\in (-0.5,1].
\end{array}\right.
$$
This flow shows coherent behavior in the four vortices throughout the time interval $[-1,-0.5]$, while four alternating (perpendicularly) shearing periods on $[-0.5,1]$, each of length approximately $0.4$, result in a mixing regime on this latter time interval.

Standard LCS methods of coherent structure detection will \emph{fail to detect the coherent behaviour} because it only lasts for the first quarter of the time duration $[-1,1]$, and is then destroyed.
We will show that \emph{we can detect the coherent regime and corresponding coherent sets using the spectrum and eigenfunctions of $\Delta_{\Goa}$}.

We generate trajectories of the system on the time interval $[-1,1]$ sampled on a uniform grid of $101$ time instances, with a $35\times 35$ regular spatial grid of initial conditions. The approximation of $\Delta_{\Goa}$ is carried out by the FEM-based method described in section~\ref{sec:FEMDL}.
We select $a = 2/\pi$, based on the heuristic in section \ref{sssec:choosing_a}, and compute the leading 20 eigenvalues and corresponding eigenfunctions.
Next, we separate the temporal eigenfunctions from the spatial ones, as described in section~\ref{sssec:distinguishing}. Figure~\ref{fig:ChiSow_gy-mxmx_evals} shows the eigenvalues together with their types (spatial/temporal).
Our choice of $a$ seems to be appropriate because the leading few eigenvalues contain a small number of temporal eigenvalues within several spatial eigenvalues, the latter being our main interest.
From Figure~\ref{fig:ChiSow_gy-mxmx_evals} we immediately see that in addition to $F_1$, which is always spatial, the next spatial eigenfunctions are $F_4, F_5, F_6$, followed by a clear gap in the spectrum to the next spatial eigenfunction $F_7$.
Thus, we expect to see four dominating coherent sets of similar coherence strength due to the similarity of the values of $\Lambda_4, \Lambda_5, \Lambda_6$.

Following the discussion in section~\ref{sssec:fibrenorms}, the relative size of the $L^2$ norm between different timeslices of our three subdominant spatial eigenfunctions $F_k(t,\cdot)$, $k=4,5,6$ gives our first indication of the time durations over which we have coherent dynamics.
Figure \ref{fig:ChiSow_gy-mxmx_slicenorms_vs_a}(a) shows relatively large and approximately equal values for these norms for each $k=4,5,6$, in a time interval approximately equal to $[-1,-0.6]$, indicating possible coherence during this time interval.
A precise time interval for coherence is not clear, but we can confidently say (i) the flow contains highly coherent sets at $t=-1$, (ii) there are likely 4 highly coherent sets encoded in $F_1,F_4,F_5,F_6$, and (iii) that the coherence of these sets is lost by $t=-0.5$.
Figure~\ref{fig:ChiSow_gy-mxmx_slicenorms_vs_a}(b) shows the slicewise $L^2$ norms $t\mapsto u(t)$ of the dominant dynamic eigenmode for the choices $a = 2^l / \pi$, $l=-1,0,1,2,3,5,7,9$, illustrating that the best choices of $a$ lie between $1/\pi$ and $4/\pi$, consistent with our heuristics in sections \ref{sssec:choosing_a} and~\ref{sssec:distinguishing}.

Figure~\ref{fig:ChiSow_gy-mxmx_EVs} shows timeslices of the first three subdominant spatial eigenmodes $F^\spat_k$, $k=2,4,5$. The results for $k=3$ look analogously to those for $k=2$, only rotated by 90 degrees.
In this example, these correspond to the indices~$k=4,5,6$ in the global order; i.e.\ $F_2^\spat=F_4$, $F_4^\spat=F_6$, and~$F_5^\spat=F_7$.
\begin{figure}[htbp]
\centering
\begin{subfigure}{\textwidth}
\includegraphics[width=0.19\textwidth]{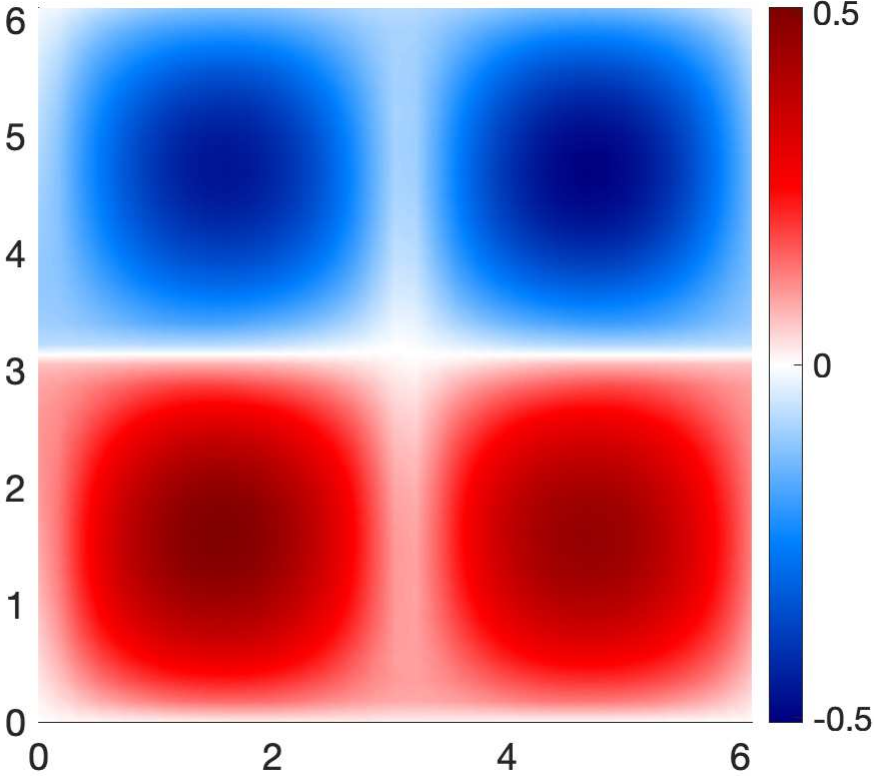}
\hfill
\includegraphics[width=0.19\textwidth]{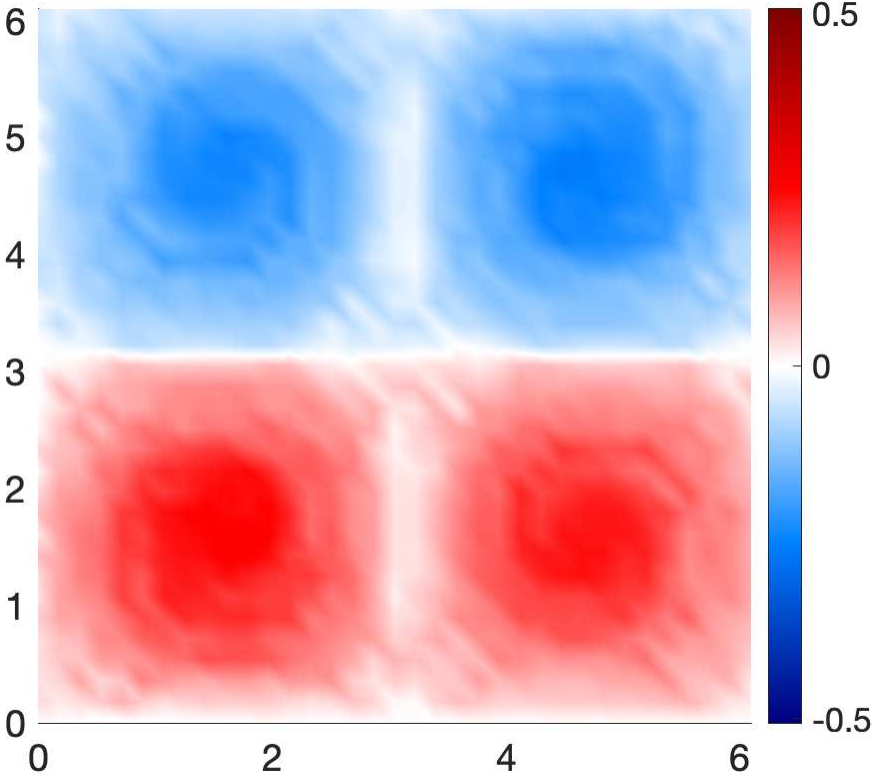}
\hfill
\includegraphics[width=0.19\textwidth]{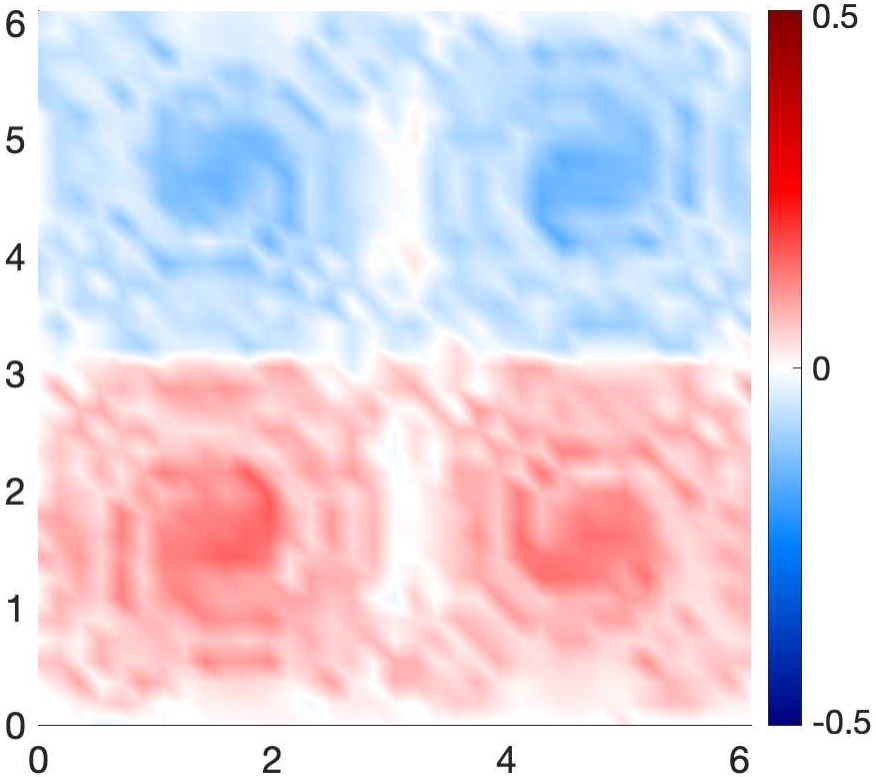}
\hfill
\includegraphics[width=0.19\textwidth]{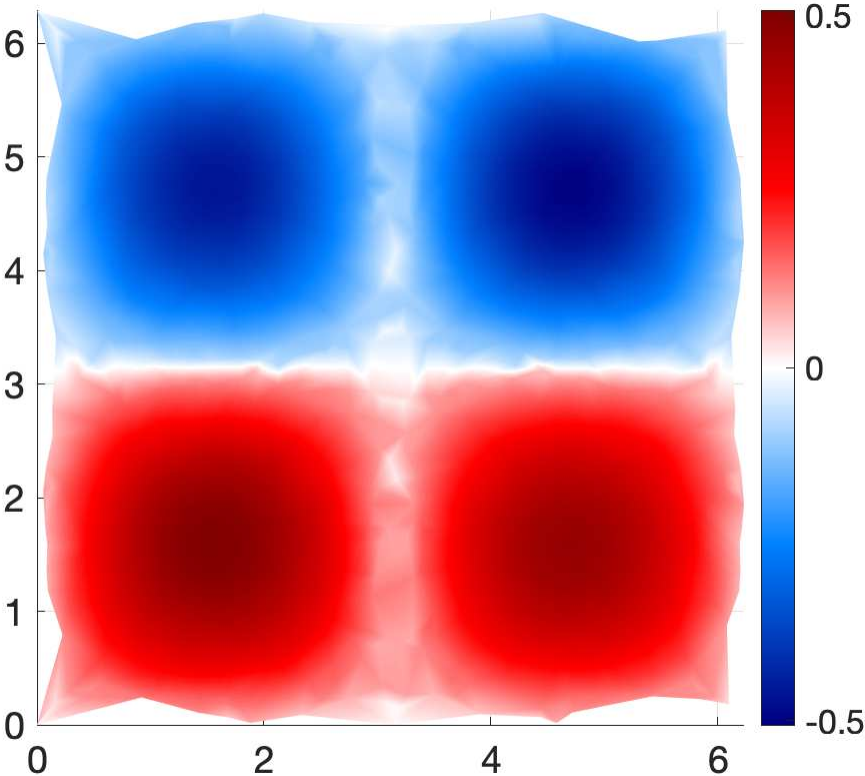}
\hfill
\includegraphics[width=0.19\textwidth]{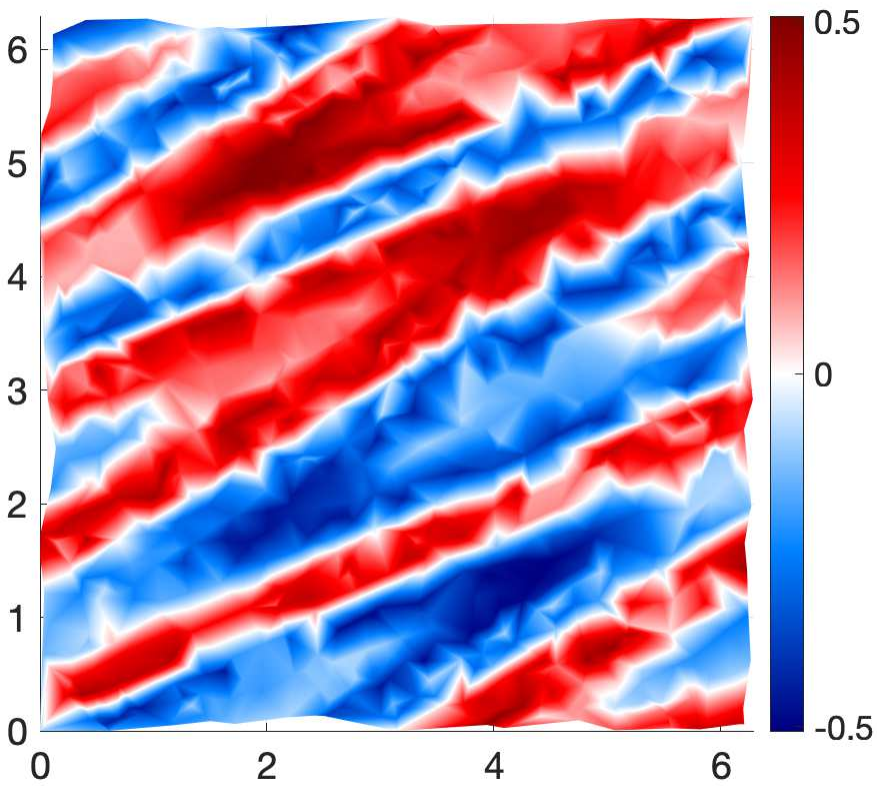}

\subcaption{Spatial mode $k=2$}
\end{subfigure}

\begin{subfigure}{\textwidth}
\includegraphics[width=0.19\textwidth]{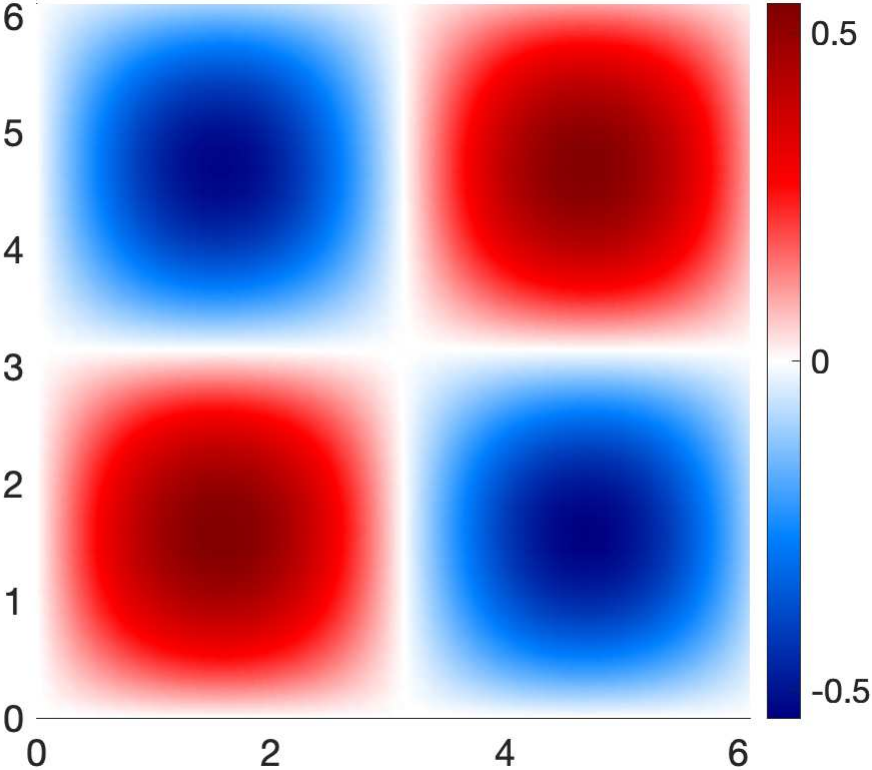}
\hfill
\includegraphics[width=0.19\textwidth]{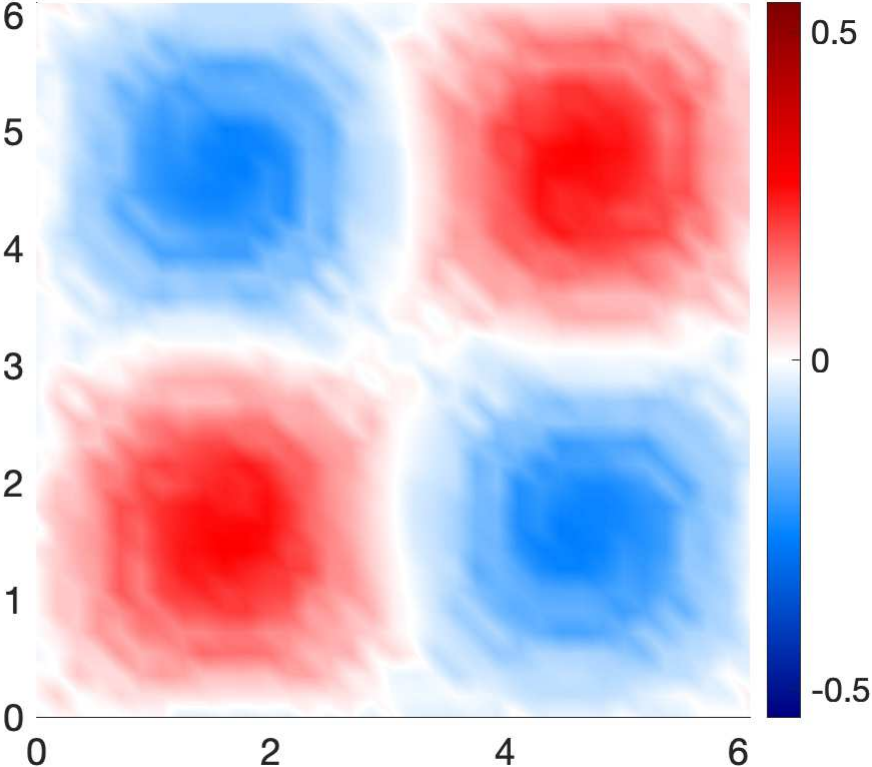}
\hfill
\includegraphics[width=0.19\textwidth]{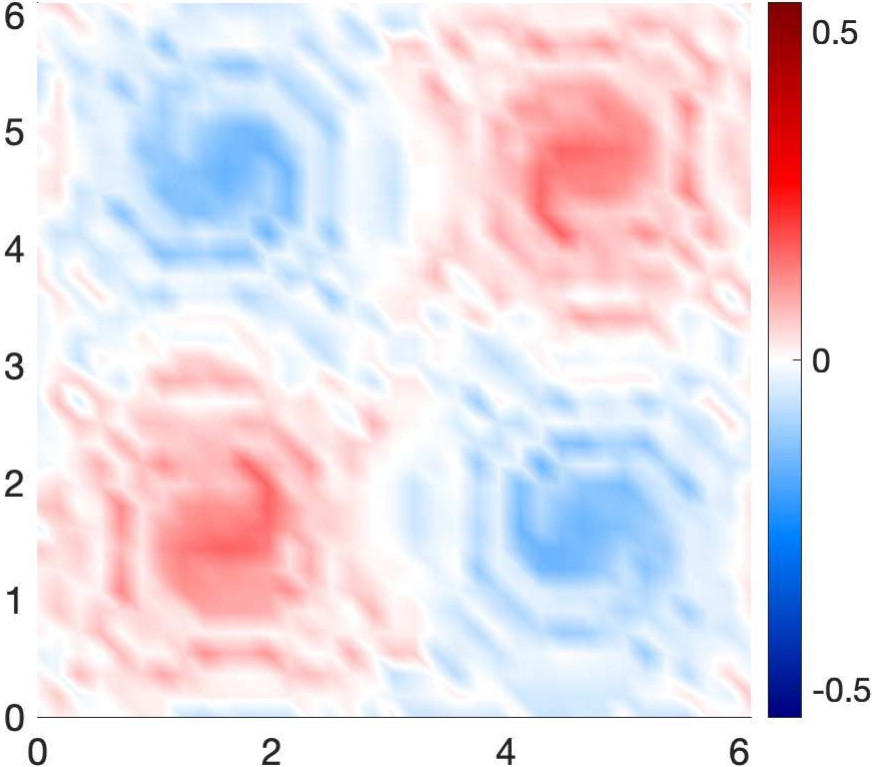}
\hfill
\includegraphics[width=0.19\textwidth]{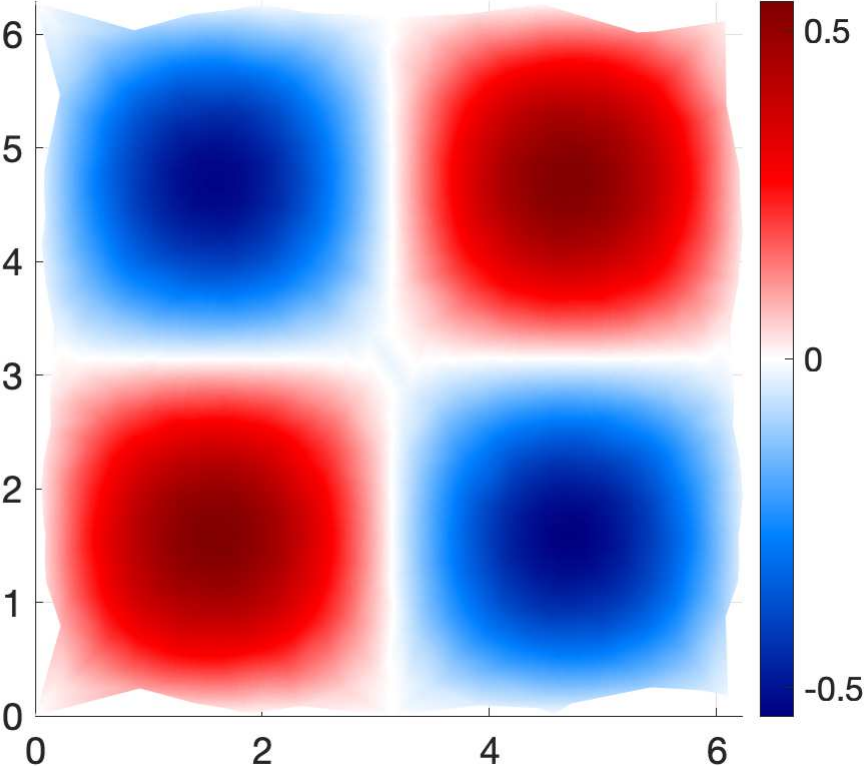}
\hfill
\includegraphics[width=0.19\textwidth]{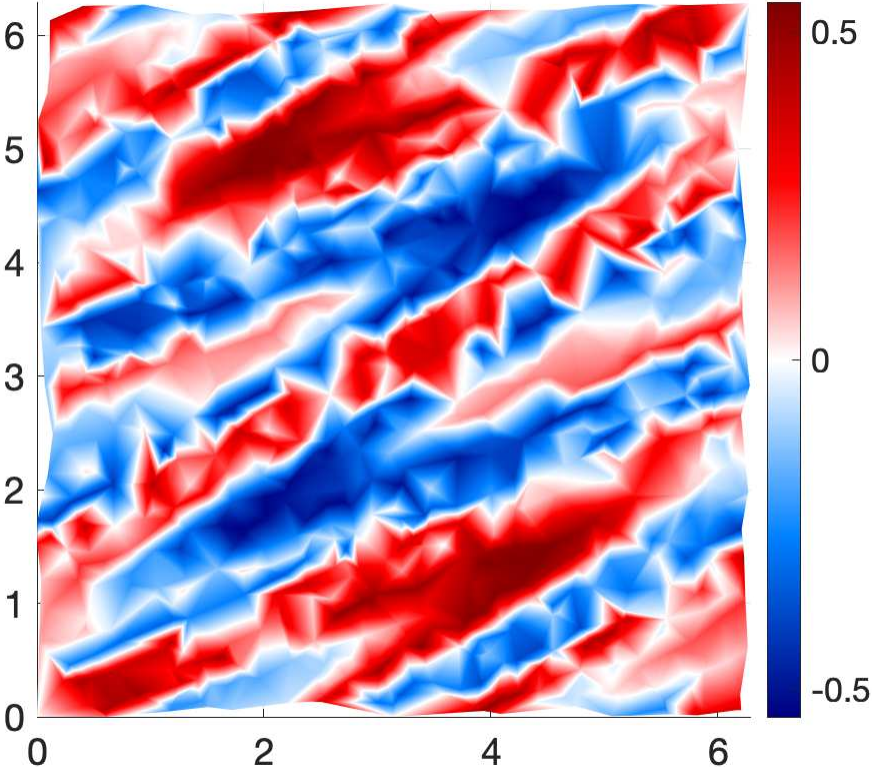}

\subcaption{Spatial mode $k=4$}
\end{subfigure}

\begin{subfigure}{\textwidth}
\includegraphics[width=0.19\textwidth]{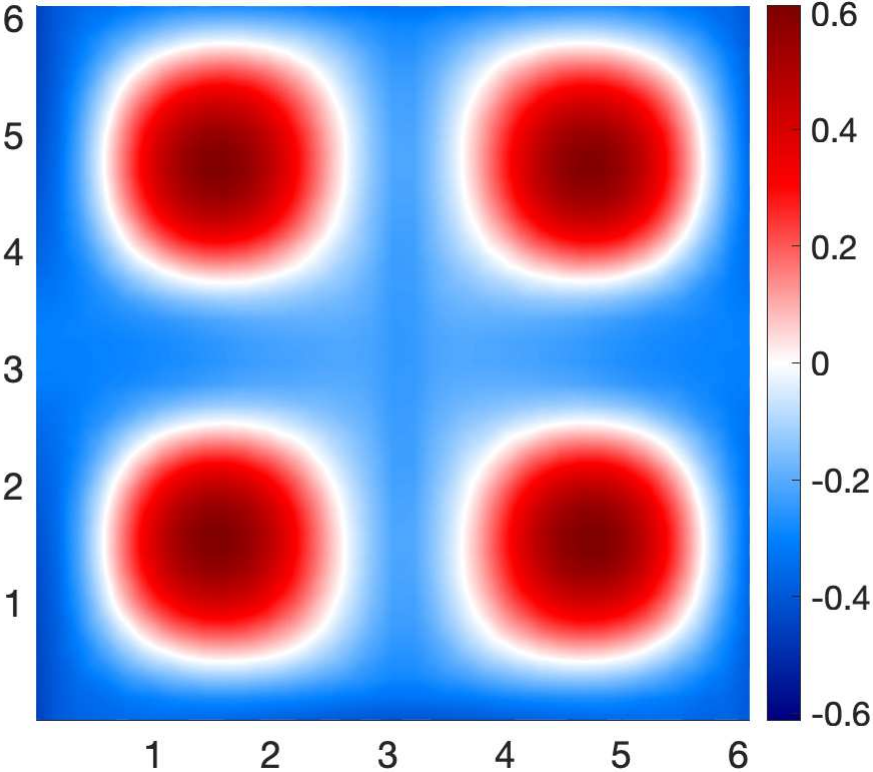}
\hfill
\includegraphics[width=0.19\textwidth]{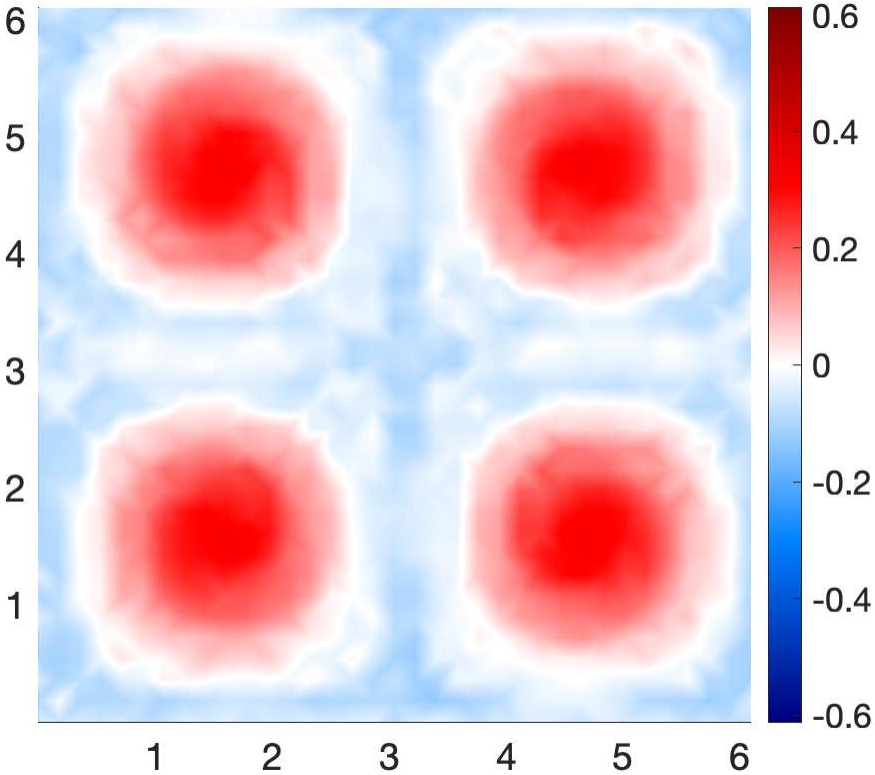}
\hfill
\includegraphics[width=0.19\textwidth]{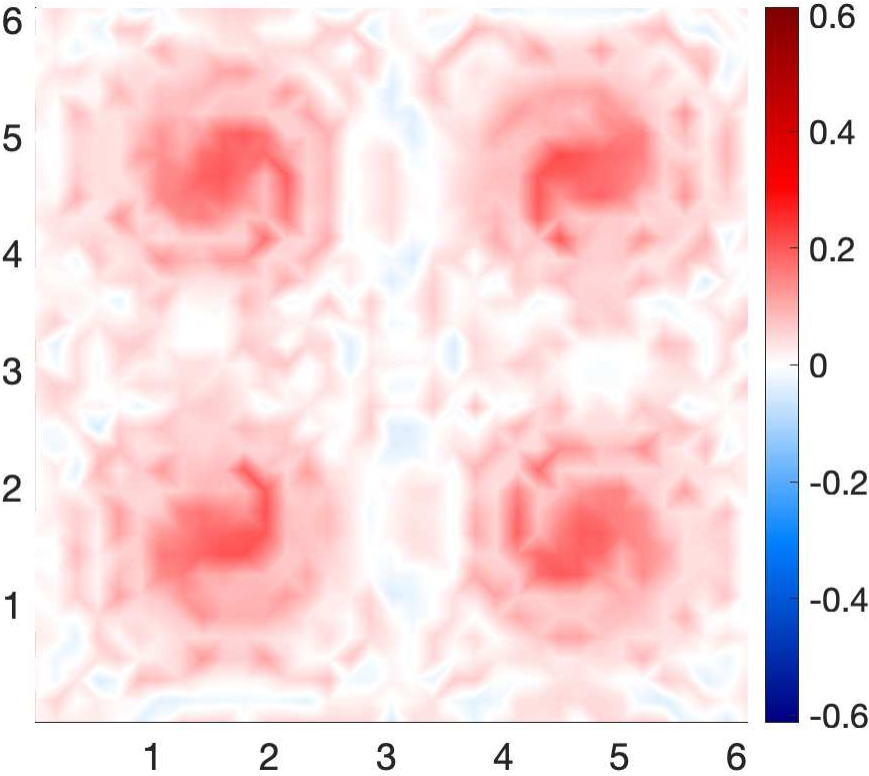}
\hfill
\includegraphics[width=0.19\textwidth]{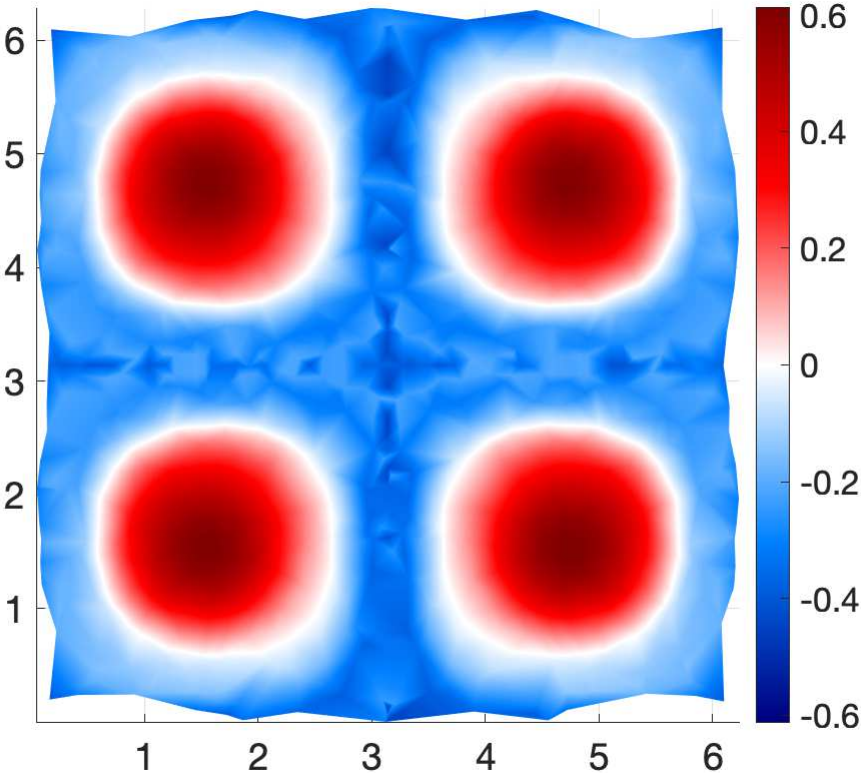}
\hfill
\includegraphics[width=0.19\textwidth]{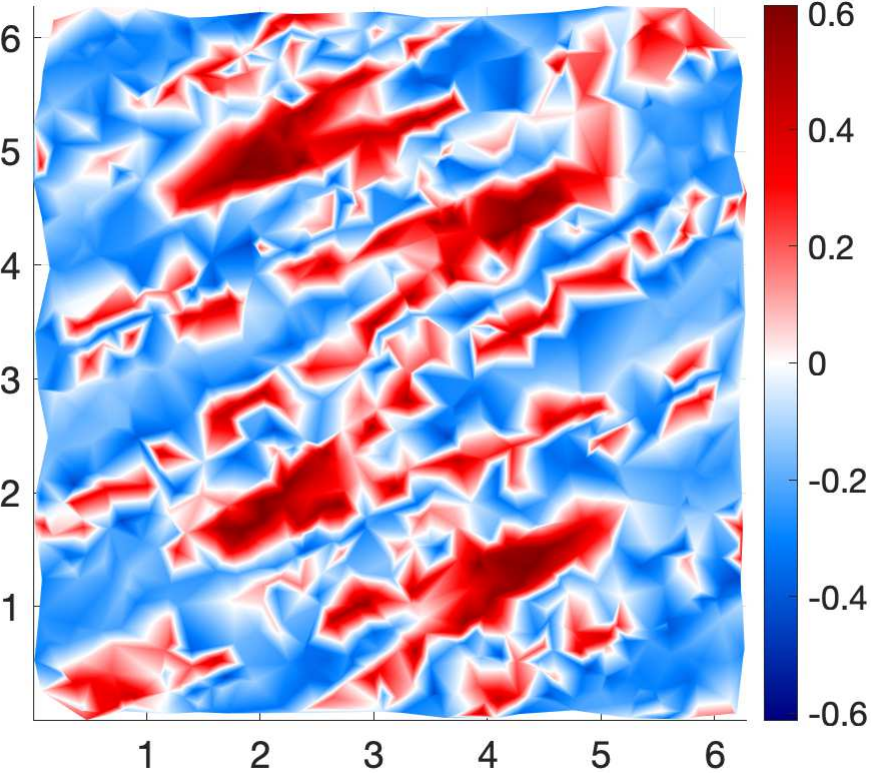}

\subcaption{Spatial mode $k=5$}
\end{subfigure}

\begin{subfigure}{\textwidth}
\includegraphics[width=0.19\textwidth]{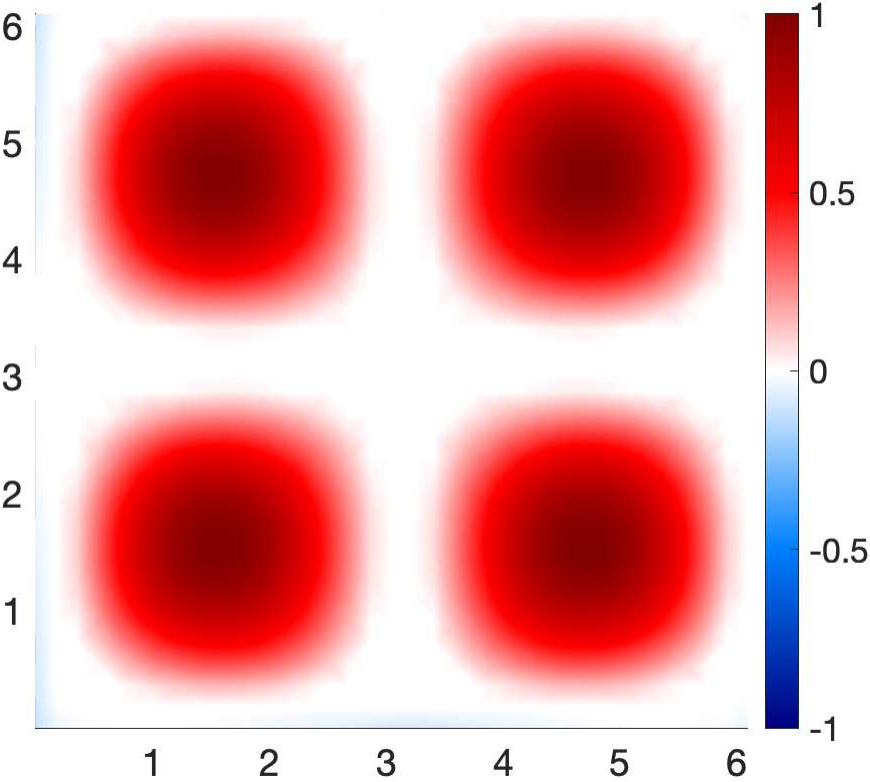}
\hfill
\includegraphics[width=0.19\textwidth]{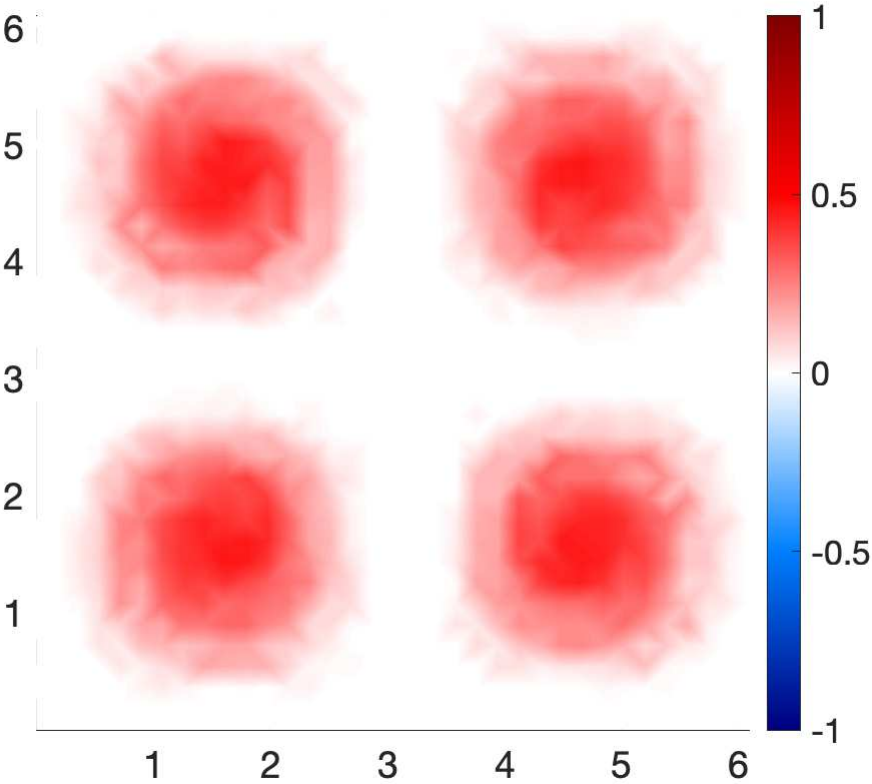}
\hfill
\includegraphics[width=0.19\textwidth]{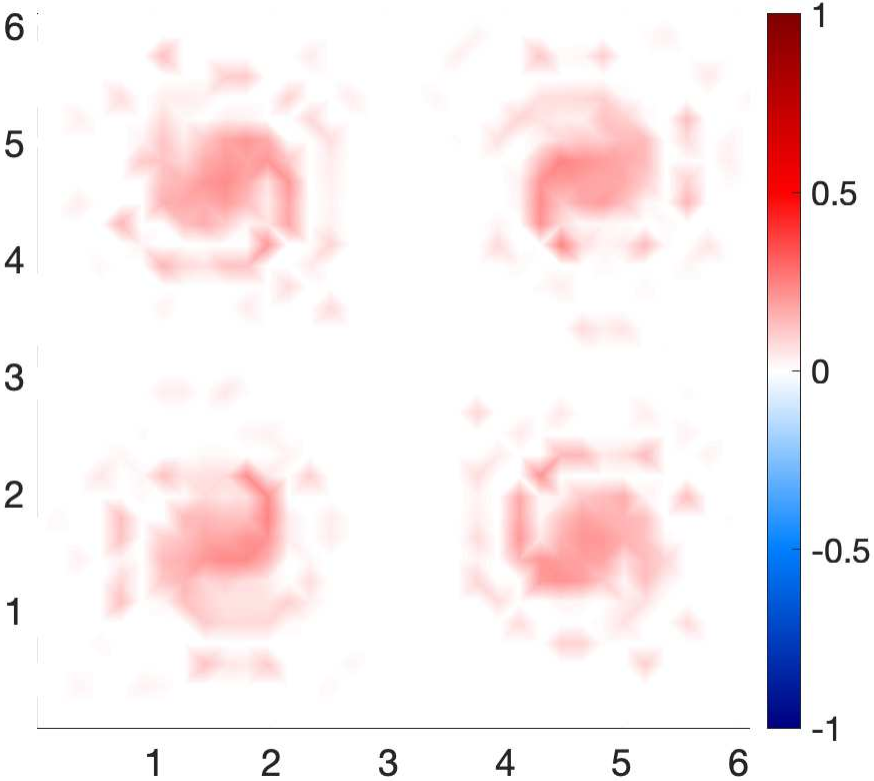}
\hfill
\includegraphics[width=0.19\textwidth]{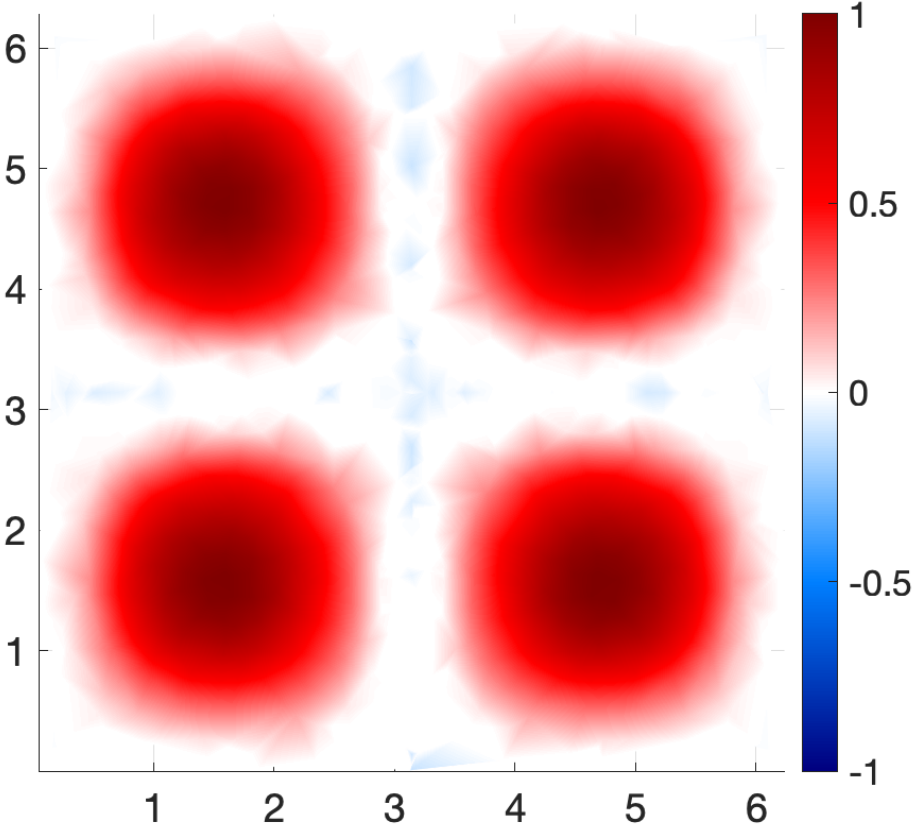}
\hfill
\includegraphics[width=0.19\textwidth]{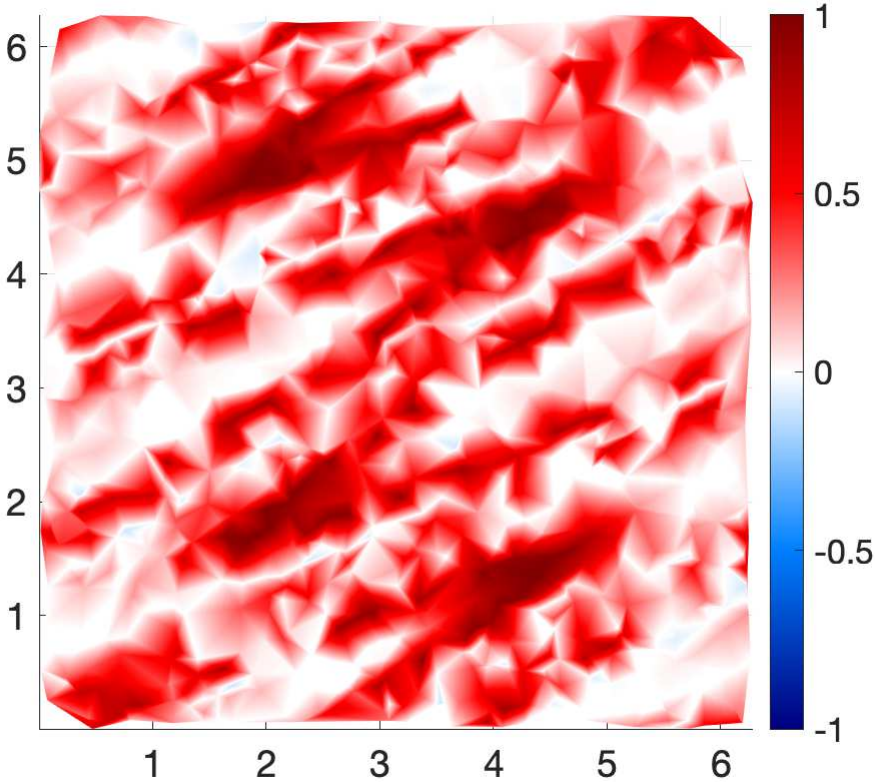}

\subcaption{Pointwise maximum of SEBA vectors}
\end{subfigure}

\caption{Visualization of timeslices $F_k^{\spat}(t,\cdot)$ of spatial eigenfunctions of $\Delta_{\Goa}$, for $k=2,4,5$ (rows (a), (b), and (c), respectively).
Colours indicate the value of the time slice eigenfunction $F_k^{\spat}$.
The first three columns show slices of a function on the $t^{\mathrm{th}}$ time fibre of $\set{M}_0$ for~$t=-1,-0.5, -0.4$, respectively. The last two columns depict $F_k^{\spat}(-1,\cdot)$ evolved forward to time fibres in $\set{M}_1$ at $t=-0.5$ and $-0.4$;  see the main text for a precise description.
Row (d) shows the images corresponding to the previous rows for the SEBA superposition~$S_{\mathrm{max}}$.}
\label{fig:ChiSow_gy-mxmx_EVs}
\end{figure}
We fix the colorscale from the first timeslice $F_k^\spat(-1,\cdot)$ as the slice-wise norms maximize there and thus indicate coherence.
The first column in Figure \ref{fig:ChiSow_gy-mxmx_EVs} displays the eigenvectors $F_k^{\spat}$ for $k=2,4,5$, in the first three rows, respectively.
An automated way to separate the different coherent sets encoded by groups of eigenvectors is implemented in the Sparse Eigenbasis Approximation (SEBA) algorithm~\cite{FrRoSa19}.
It computes a rotation of leading eigenvectors such that the resulting rotated vectors are maximally sparse. 
SEBA is applied to the first four nontrivial spatial eigenfunctions $F_2^\spat,\ldots,F_5^\spat$, to produce $S_1,\ldots,S_4$, which are functions of $t$ and~$x$.
Each of these SEBA functions should be supported on a \emph{single} semi-material coherent set (in space-time) and the value represents the relative strength of membership in one of the four coherent sets.
Note that in contrast to \cite{FrRoSa19} we do not use the leading (constant) spatial eigenfunction when applying SEBA because in the present setting we do not expect the union of the semi-material coherent sets in the time-expanded set $\set{M}_0$ to be all of $\set{M}_0$.
Having separated the semi-coherent sets, we define $S_{\max}:=\max\{S_1,\ldots,S_4\}$ into a single function via superposition.
The final row of the first column in Figure \ref{fig:ChiSow_gy-mxmx_EVs} shows $S_{\max}(-1,\cdot)$.

Column 2 of Figure \ref{fig:ChiSow_gy-mxmx_EVs} displays $F_2^\spat(-0.5,\cdot),F_4^\spat(-0.5,\cdot),F_5^\spat(-0.5,\cdot),$ and \linebreak[4] $S_{\max}(-0.5,\cdot)$, respectively.
Notice that by time $t=-0.5$, the point at which the dynamics enters its mixing phase, the relative magnitude of the eigenfunctions and SEBA superposition has begun to decrease;  compare this with Figure \ref{fig:ChiSow_gy-mxmx_slicenorms_vs_a}(a).
The third column of Figure \ref{fig:ChiSow_gy-mxmx_EVs} is the same as the second column, except $t$ has been advanced from $-0.5$ to $-0.4$, after which the mixing regime begins.
A significant decrease in the magnitude of the eigenfunctions and SEBA superposition can now be seen, in line with Figure \ref{fig:ChiSow_gy-mxmx_slicenorms_vs_a}(a).

The fourth column of Figure \ref{fig:ChiSow_gy-mxmx_EVs} shows the pushforwards $((\phi_{0.5})_* F_2^\spat(-1,\cdot)$, \linebreak[4] $(\phi_{0.5})_* F_4^\spat(-1,\cdot)$, $(\phi_{0.5})_* F_5^\spat(-1,\cdot)$, $(\phi_{0.5})_* S_{\max}(-1,\cdot)$, respectively.
In other words we take the functions in column 1 of Figure \ref{fig:ChiSow_gy-mxmx_EVs}, fix the colour of each point in $M$ and flow every point forward for $0.5$ time units without changing its colour.
The fourth column of Figure \ref{fig:ChiSow_gy-mxmx_EVs} clearly shows coherence on the time interval~$[-1,-0.5]$.
Advancing this flow a little further by $0.1$ time units we arrive at the fifth column of Figure~\ref{fig:ChiSow_gy-mxmx_EVs}.
One sees a dramatic difference, with rapid destruction of the coherent sets.
This is strongly consistent with the small slicewise norm values shown in the third column, indicative of a lack of coherence.

The eigenmodes $F_k^\spat$, $k=2,\ldots,5$ separate the four gyres from one another;  higher spatial modes subdivide the gyres into coherent rings and spiral-like structures (not shown).
These are less coherent than the gyre structures in Figure \ref{fig:ChiSow_gy-mxmx_EVs}, as indicated by the spectrum of~$\Delta_{G_{0,a}}$.
We summarise our approach in algorithm form below.

\paragraph{Algorithm (to extract semi-material coherent sets from trajectory data)}
\begin{compactenum}
\item Generate $N$ trajectories $\{x_t\}_{t\in\mathcal{T}}$, where $x_t\in \mathbb{R}^d$ and $\mathcal{T}\subset [0,\tau]$ has cardinality $T$.
\item Select the time diffusion strength $a$ according to the heuristic in section \ref{sssec:choosing_a}.
\item Construct the $N\times N$ matrices $M^t, D^t$ for $t\in\mathcal{T}$ as in section~\ref{sec:FEMDL}.
\item For the current choice of $a$ construct $\mathbf{M}$ and $\mathbf{D}$ as in section \ref{sec:FEMDL}, and solve the inflated dynamic Laplacian eigenproblem~$\mathbf{D}\mathbf{w} = \Lambda \, \mathbf{M}\mathbf{w}$.
\item Classify eigenfunctions as spatial or temporal by computing the temporal variance of spatial means, as in section \ref{sssec:distinguishing}.
\item If all leading nontrivial eigenfunctions are temporal, increase $a$ and return to step~4.  Aim for a value of $a$ with a small number of temporal eigenvalues early in the spectrum with most eigenvalues being spatial.
\item Plot the slicewise norms $\|F_k^{\spat}(t, \cdot)\|_2$ vs time $t$ as in Figure \ref{fig:ChiSow_gy-mxmx_slicenorms_vs_a};  large values indicate periods of coherence for the features encoded in $F_k^{\spat}$, while zero or near-zero values indicate strong mixing of the features encoded in $F_k^{\spat}$.
\item Apply SEBA to a collection of leading spatial eigenfunctions $F_k^{\spat}$, $k=2,\ldots,K+1$, where $K$ is determined by a spectral gap or other means, to produce a family of SEBA functions $S_k$, $k=1,\ldots,K$.
\item The eigenfunction families $F_k^{\spat}(t, \cdot)$ or SEBA function families $S_k(t, \cdot)$ may be spatially plotted in the pullback space $\mathbb{M}_0$ or the co-evolved space $\mathbb{M}_1$ as in Figure~\ref{fig:ChiSow_gy-mxmx_EVs}.
\end{compactenum}

\subsection{Semi-material coherent sets for a flow with multiple coherent and incoherent regimes}
\label{ssec:ChiSow_CMCM}
Between time $0$ and time~$\tau=4$ we consider a non-autonomous Childress--Soward flow~\eqref{eq:ChiSow} with time-dependent parameter modulation
\[
r(t) = \left\{
\begin{array}{ll}
0, & t\in [0,0.6] \cup [1.4,3.2], \\
\mathrm{sign}\left( \cos\left( 5\pi t \right) \right), & t\in (0.6,1.4) \cup (3.2,4]
\end{array}\right.
\]
and time-dependent amplitude $A(t) = 60\, \mathds{1}_{\{r=0\}}(t) + 40\, \mathds{1}_{\{r\neq 0\}}(t)$.
This flow shows coherent behavior in four vortices throughout the time interval $[0,0.6]$ and again in the interval $[1.4, 3.2]$.
The perpendicularly alternating shearing creates a mixing flow outside these two intervals.
\new{We have chosen a piecewise-continuous-in-time velocity field to stress test our numerics.  Such discontinuities do not affect the existence of the spectrum of the inflated dynamic Laplacian, nor the existence of eigenfunctions in a weak sense, as these quantities arise via integrals.  The finite-element approach to approximating the spectrum and eigenfunctions in Section \ref{sec:FEMDL} uses a weak formulation, which is also unaffected by piecewise-continuous inputs. The surrogate problem in Section \ref{sec:characterisation} only requires the ``relative mixing strength'' $t\mapsto\rho(t)$ to be integrable.}

Our heuristic from section \ref{sssec:choosing_a} suggests the choice $a = 4/\pi$, and we will discretize $\Delta_{\Goa}$ on a grid of trajectories that start on a regular $30\times 30$ mesh and are sampled at 151 equispaced times instances on~$[0,4]$. The computed dominant spectrum, its classification into spatial and temporal eigenvalues, and the slicewise squared $L^2$ norms of the eight leading nontrivial spatial eigenfunctions are shown in Figure~\ref{fig:ChiSow_CMCM_evals_fnorms}.

\begin{figure}[hbt]
	\begin{subfigure}{0.49\textwidth}
	\includegraphics[width=\textwidth]{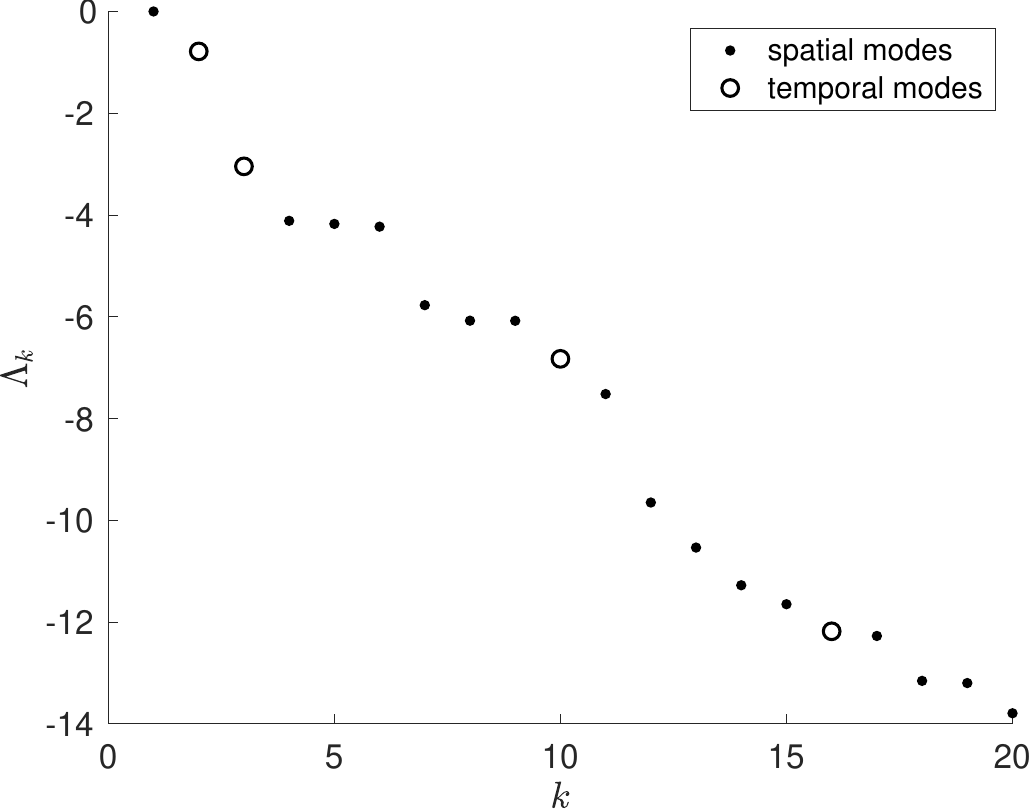}
	\subcaption{}
	\end{subfigure}
	\hfill
	\begin{subfigure}{0.49\textwidth}
	\includegraphics[width=\textwidth]{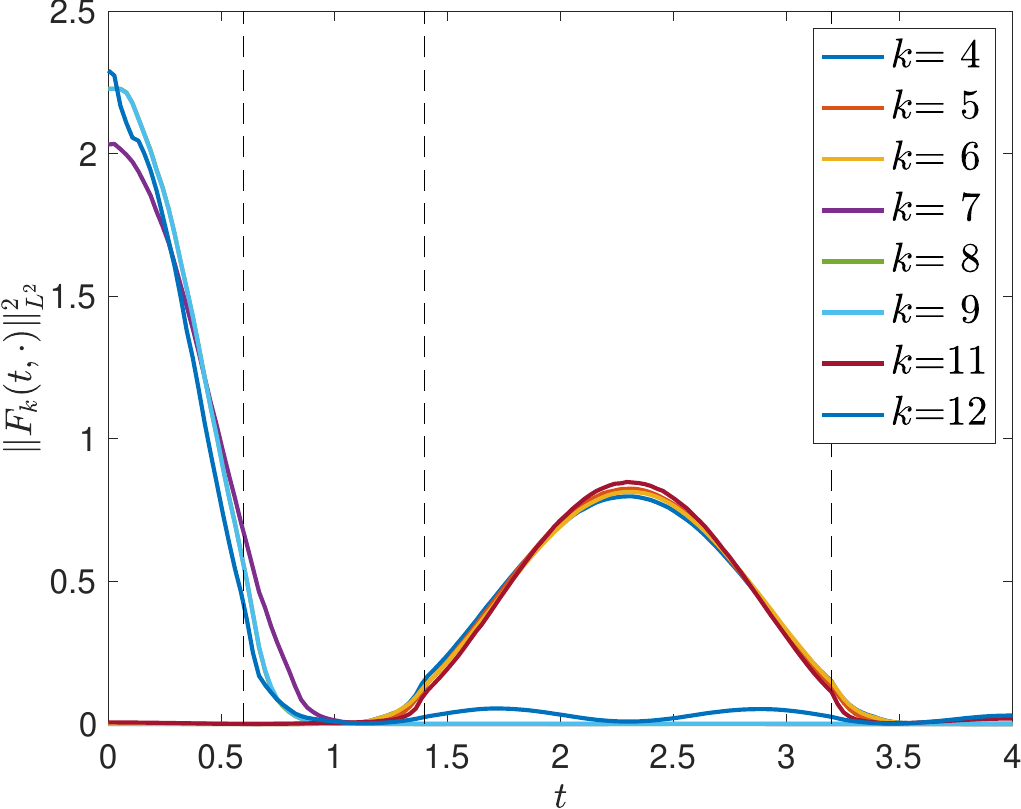}

	\subcaption{}
	\end{subfigure}
	\caption{(a) Eigenvalues. (b) Slicewise squared $L^2$ norms of 8 leading nontrivial spatial eigenfunctions for the Childress--Soward system with multiple coherent and mixing regimes ($k=8$ and $k=9$ are indistinguishable in this figure). The vertical dashed lines indicate the times $t = 0.6, 1.4, 3.2$, when coherent motion switches to shearing and vice versa.}
	\label{fig:ChiSow_CMCM_evals_fnorms}
\end{figure}

We observe that large timeslice norms are concentrated on time intervals of coherent behavior. Moreover, each spatial mode has large timeslice norms only on one of the coherent regimes.
With this we are able to easily identify the lifetimes of the corresponding eight semi-material coherent sets. Figure~\ref{fig:ChiSow_CMCM_largestslice} displays the leading eight nontrivial spatial modes at the timeslice of their respective maximal slicewise $L^2$ norm;  these eigenfunctions encode the spatial structure of the finite-time coherent sets.


\begin{figure}[hbt]
	\centering
	\includegraphics[width = 0.11\textwidth]{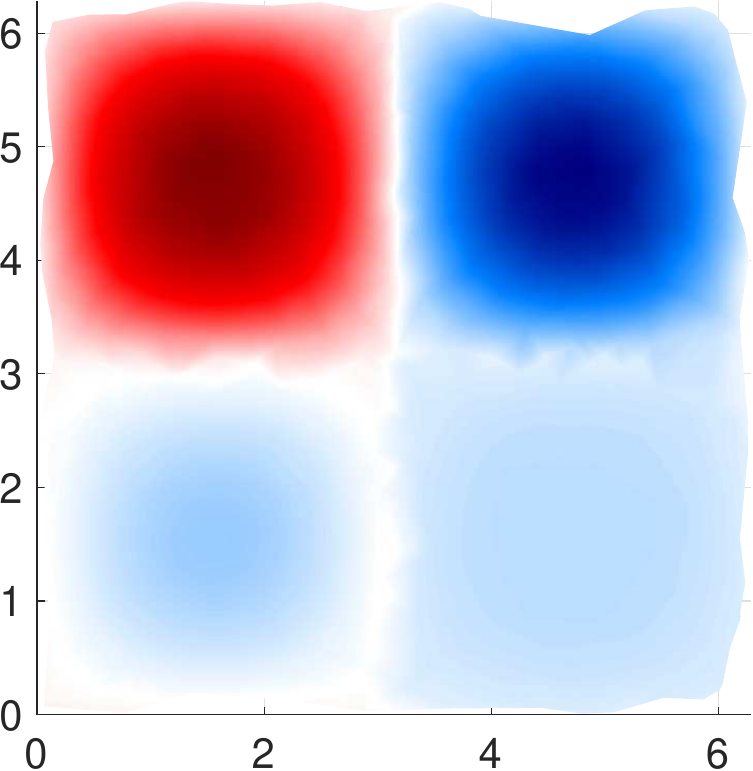}
	\hfill
	\includegraphics[width = 0.11\textwidth]{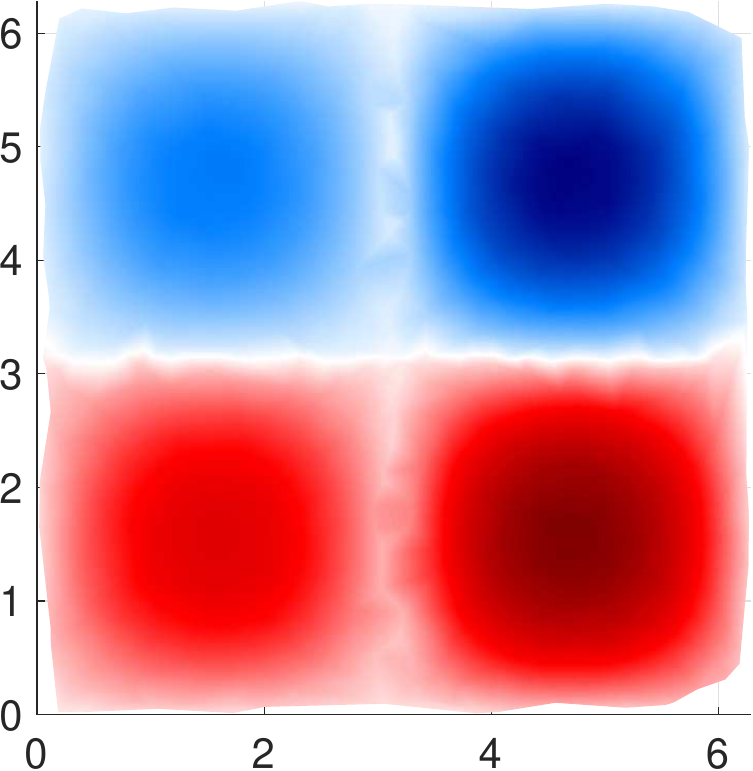}
	\hfill
	\includegraphics[width = 0.11\textwidth]{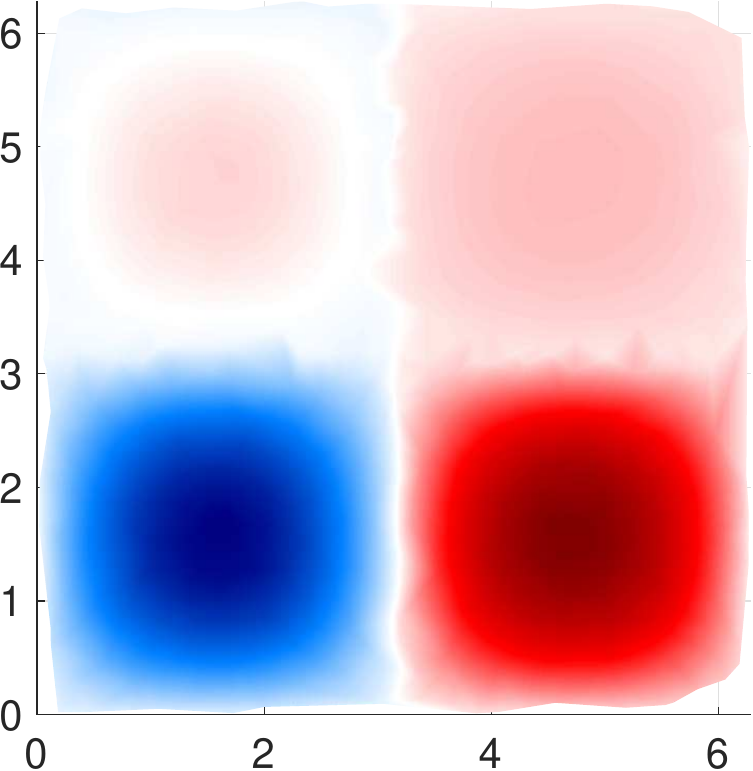}
	\hfill
	\includegraphics[width = 0.11\textwidth]{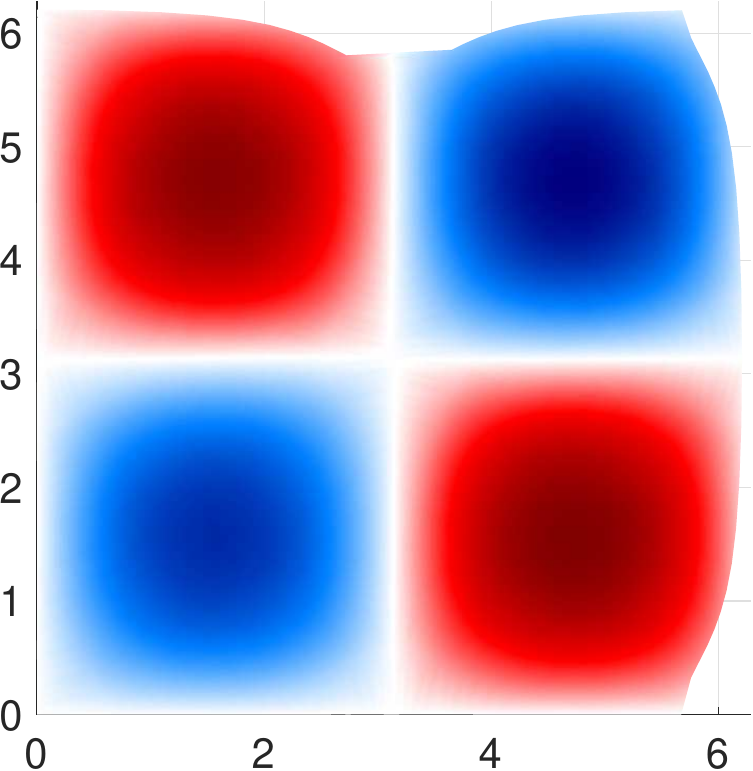}
	\hfill
	\includegraphics[width = 0.11\textwidth]{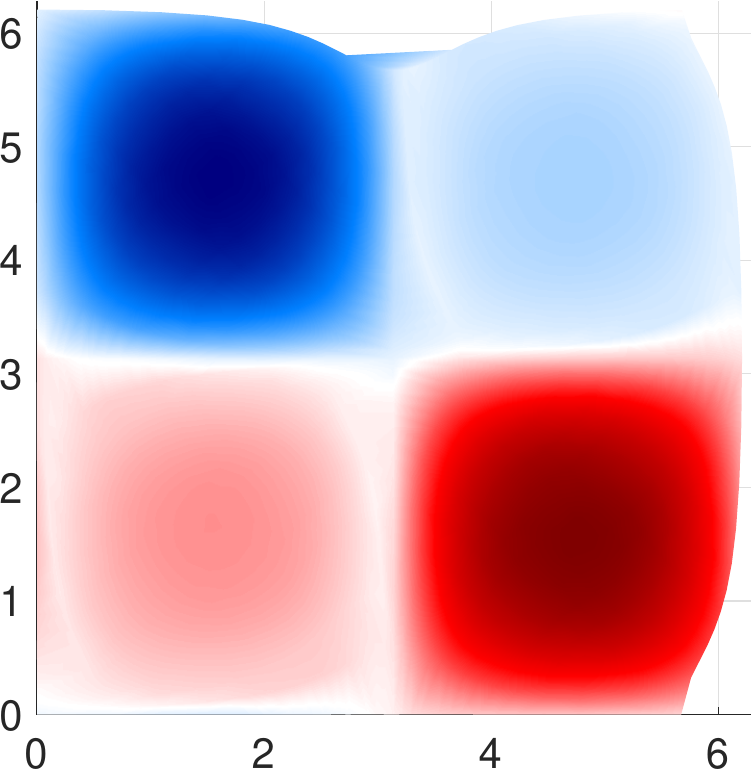}
	\hfill
	\includegraphics[width = 0.11\textwidth]{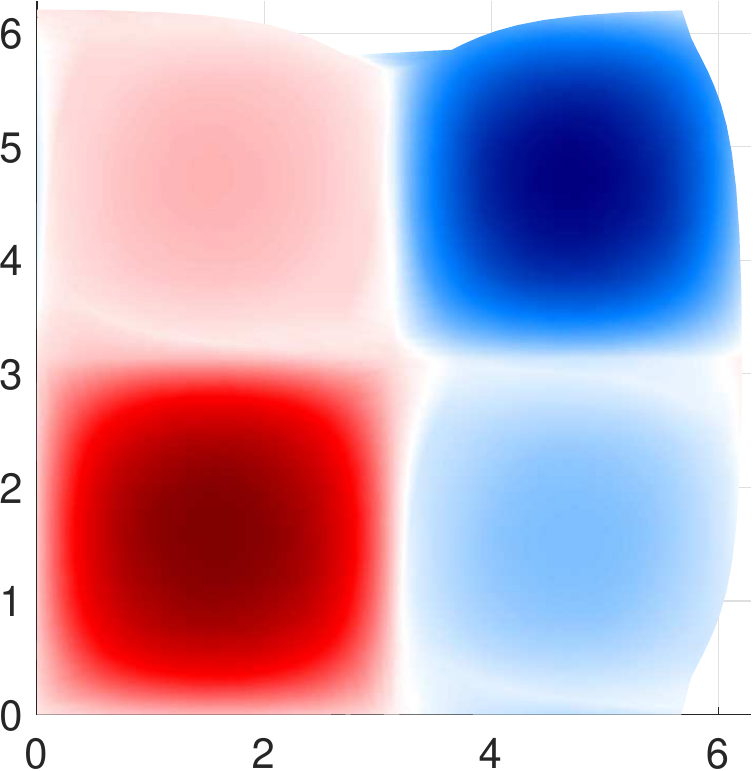}
	\hfill
	\includegraphics[width = 0.11\textwidth]{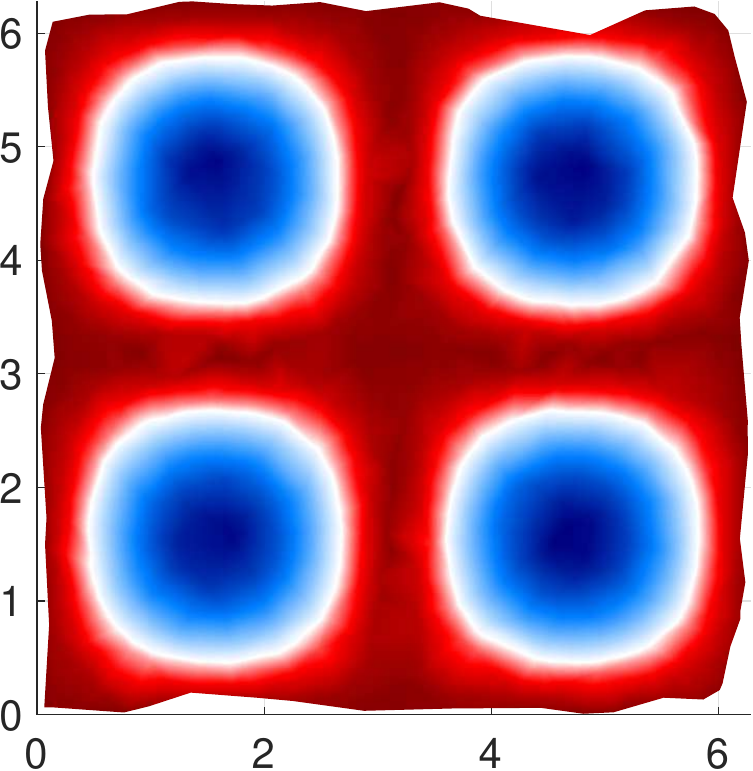}
	\hfill
	\includegraphics[width = 0.11\textwidth]{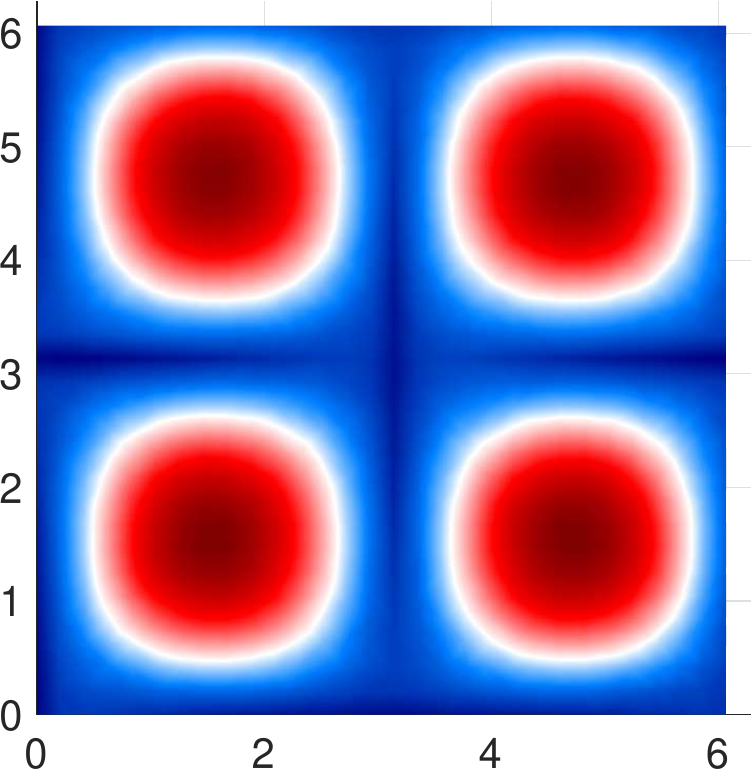}
	
	\caption{The eigenvectors listed in the legend of Figure~\ref{fig:ChiSow_CMCM_evals_fnorms} (b) shown as the timeslice $F_k(t,\cdot)$ where $t\mapsto \|F_k(t,\cdot)\|_{L^2(M)}$ is maximal, co-evolved by the flow to that timeslice (left to right). For $k=4,5,6,11$ this is for some $t\in [1.4,3.2]$, for the other eigenmodes this is for some~$t\in[0,0.6]$.}
	\label{fig:ChiSow_CMCM_largestslice}
\end{figure}

To distinguish the individual coherent sets, we apply SEBA to the eigenvectors~$F^{\spat}_k$ for~$k=2,\ldots,9$, and obtain the eight spacetime SEBA vectors~$S_1,\ldots, S_8$. Each one of them indicates one of the coherent gyres: four gyres in each of the two time intervals of coherent motion (not shown). We visualize the coherent gyres by depicting the pointwise maximum of SEBA vectors, $S_{\max} := \max_i S_i$, in coevolved spacetime $\set{M}_1$ in Figure~\ref{fig:ChiSow_CMCMC_SEBA}.

\begin{figure}[hbt]
	\centering
	\includegraphics[width = 0.6\textwidth]{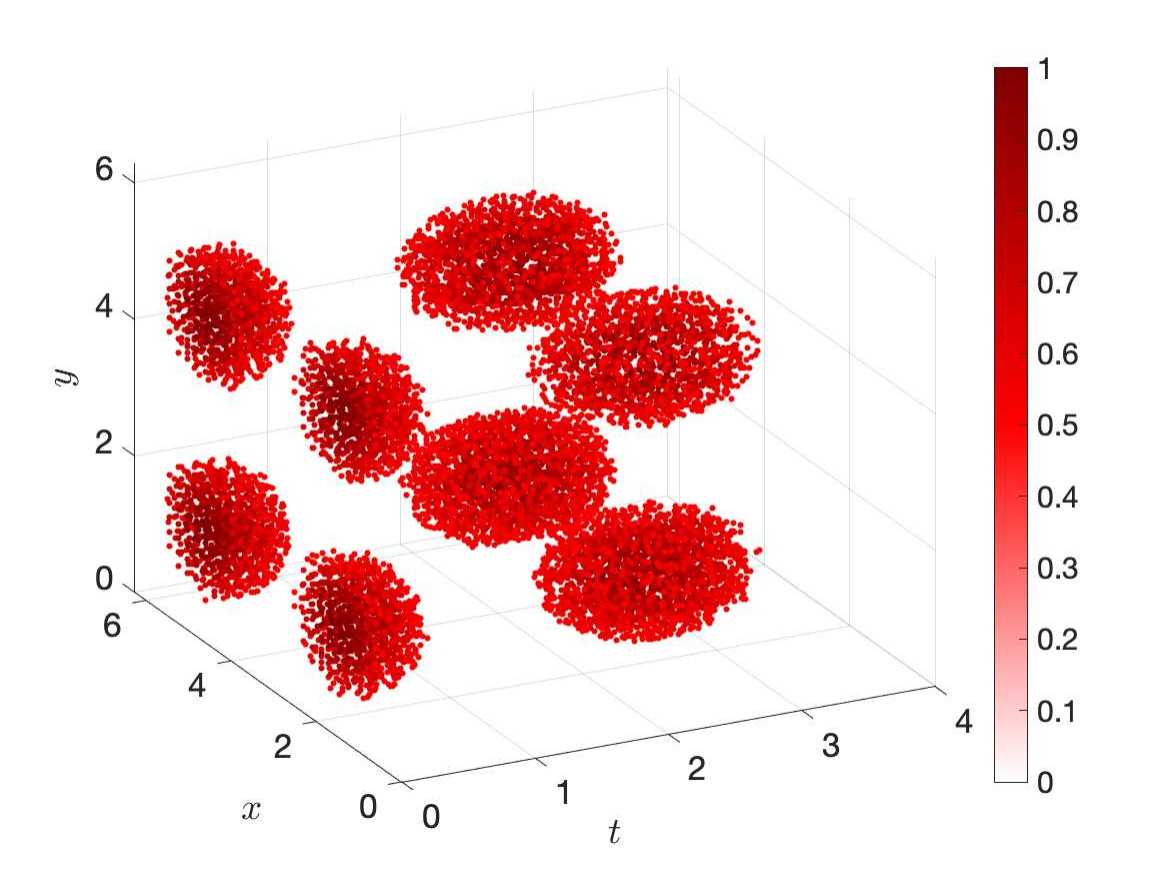}
	\caption{Maximum of SEBA vectors for the Childress--Soward system with multiple coherent and mixing regimes, plotted in spacetime $\set{M}_1$ for coevolved points. For visual clarity, only points with $S_{\max}$ value larger than $0.5$ are plotted. Note that the co-evolved points remain within each of their compact regions.}
	\label{fig:ChiSow_CMCMC_SEBA}
\end{figure}

\section*{Acknowledgments}

We thank Kathrin Padberg-Gehle for initial discussions on the idea of relaxing materiality, Ben Goldys and Georg Gottwald for helpful discussions on SDEs on manifolds, Carsten Gr\"aser and Martin Stahn for discussions on the $H^2$-regularity in Appendix~\ref{app:proof_int_cohmix}, and Shane Keating for pointing us to the Childress--Soward (``cat's eye'') flow.
We further thank Marcus Weber and Konstantin Fackelday for discussions on forward and backward transfer operators.

The research of GF was partially supported by an Australian Research Council Discovery Project.
PK thanks the UNSW School of Mathematics and Statistics for support through the Distinguished Visitor Scheme, where much of this research was carried out. PK was partially supported by Deutsche Forschungsgemeinschaft (DFG) through grant CRC 1114 ``Scaling Cascades in Complex Systems'', Project Number 235221301, Project A01 ``Coupling a multiscale stochastic precipitation model to large scale atmospheric flow dynamics''.
Both authors acknowledge the Australian-German Joint Research Cooperation Scheme, which supported visits to UNSW and FU Berlin.

\appendix

\section{Eigenvalue bounds and asymptotics}
\label{app:spectrum}
\begin{proof}[Proof of Theorem~\ref{thm:eigbounds}]
\quad

\paragraph{Part 1}
 If $k=1$, then $\lambda^D_1=\Lambda_{1,a}^\spat=0$ for all $a> 0$. For $k\ge 2$, note that a minimizer of \eqref{eq:dl-lamk} is the eigenfunction~$f_k$. Let~$\smash{ \tilde F_i(t,x) := f_i(x) }$ for $(t,x)\in\set{M}_0$ and~$i\ge 1$.
 Let $H^1:=H^1(\mathbb{M}_0)$ denote the Sobolev space of weakly differentiable $L^2(\mathbb{M}_0)$ functions with derivatives in $L^2(\mathbb{M}_0)$ and set $H^1_\spat=H^1(\mathbb{M}_0)\cap \mathbb{S}_0^\spat$.
 Note that~$\smash{ \tilde F_i\in H^1_{\mathrm{spat}} }$ with $\partial_t \tilde F_i \equiv 0$.
 We define $\smash{ \mathbb{S}_{k}' := \mathrm{span}\{\tilde F_1,\ldots,\tilde F_k\} \subset H^1_{\mathrm{spat}} }$, and note that for simplicity we will denote $ \nabla_{g_t} F(t,\cdot)$ by~$ \nabla_{g_t} F$. Then, by the (Courant--Fischer) min-max theorem for self-adjoint operators with no essential spectrum~\cite[sec.~4.3]{Teschl2009} applied to the spatial modes we obtain
\begin{align*}
	-\Lambda_{k,a}^\spat &= \min_{\substack{\mathbb{S} \subset H^1_{\mathrm{spat}}\\[2pt] \dim\mathbb{S}=k}} \ \max_{\substack{F\in \mathbb{S}\\ \|F\|_{L^2}=1}} \iint_{\set{M}_0} a^2 (\partial_t F)^2  + \| \nabla_{g_t} F\|_{g_t}^2\,d\ell\,dt \\
	&\le \max_{\substack{F\in \mathbb{S}_{k}'\\ \|F\|_{L^2}=1}} \iint_{\set{M}_0} a^2 \underbrace{ (\partial_t F)^2 }_{\mathclap{\quad\quad =0\text{ for }F\in\mathbb{S}_{k}' }} + \| \nabla_{g_t} F\|_{g_t}^2\,d\ell\,dt\\
	&= \max_{\substack{F\in \mathbb{S}_{k}'\\ \|F\|_{L^2}=1}} \iint_{\set{M}_0} \| \nabla_{g_t} F\|_{g_t}^2\,d\ell\,dt \stackrel{\eqref{eq:dl-lamk}}{=} -\lambda_k^D,
\end{align*}
where the inequality on the second line follows from bounding the minimum by the particular subspace~$\mathbb{S}_{k}'$, and the last equality comes from noting that the maximizer of the expression on its left-hand side is~$\tilde F_k$.

\paragraph{Part 2}
The result for $\Lambda_{k,a}^\temp$ is obvious using the explicit formula for $\Lambda_{k,a}^\temp$.
For $\Lambda_{k,a}$, the min-max characterisation states that
\begin{equation}
	\label{eq:minmax}
	-\Lambda_{k,a} = \min_{\substack{\mathbb{S} \subset H^1\\[2pt] \dim\mathbb{S}=k}} \ \max_{\substack{F\in \mathbb{S}\\[2pt] \|F\|_{L^2}=1}} \iint_{\set{M}_0} a^2 (\partial_t F)^2  + \| \nabla_{g_t} F\|_{g_t}^2\,d\ell\,dt .
\end{equation}
Because the integrand is nondecreasing in $a$, the result follows.
The argument for $\Lambda_{k,a}^\spat$ is similar.

\paragraph{Part 3}
Obvious, using the explicit formula for $\Lambda_{k,a}^\temp$.

\paragraph{Part 4}
For $a > 0$, let $F_a\in \mathbb{S}_{k}^\spat$ be the minimiser of (\ref{eq:LB-lamkspat}) of unit norm in~$L^2(\mathbb{M}_0)$. Recall from Section~\ref{sec:specaug} that these eigenfunction are smooth.
    By parts 1 and 2, we know that $a\mapsto\Lambda_{k,a}^\spat$ is a nonincreasing function bounded below by~$\lambda_k^D$.
    By (\ref{eq:Gdecomp}) and (\ref{eq:LB-lamkspat}), we must therefore have that
    \begin{equation}
    	\label{eq:t_deriv_vanishes}
    	\int_{\mathbb{M}_0}(\partial_tF_a)^2 \to 0 \quad\text{as}\quad a\to\infty.
    \end{equation}
   Since $H^1$ is reflexive and $\sup_{a\ge 1}\|F_a\|_{H^1}<\infty$, by Banach--Alaoglu there is a weak accumulation point $F_* \in H^1$ along a sequence $a_i\uparrow\infty$:
\begin{equation}
\label{eq:weak_conv}
F_{a_i} \rightharpoonup F_* \in H^1 \text{ as }a_i\to \infty.
\end{equation}

We see that $\partial_t F_* \equiv 0$ in the weak sense, since for all $\phi \in C^{\infty}(\set{M}_0)$ we have that
\[
\left| \int_{\set{M}_0} \partial_tF_*\, \phi \right| \stackrel{\eqref{eq:weak_conv}}{=}   \lim_{i\to \infty} \left| \int_{\set{M}_0}  \partial_t F_{a_i}\, \phi \right| \le \lim_{i\to\infty} \| \partial_t F_{a_i}\|_{L^2}  \| \phi\|_{L^2} \stackrel{\eqref{eq:t_deriv_vanishes}}{=} 0.
\]

Let us consider the previous constructions for all eigenvalues $\Lambda_{j,a}$ for $j=1,\ldots,k$. For the associated eigenfunctions $F_{j,a}$, which we assume to all be $L^2$-normalised, we obtain as in \eqref{eq:weak_conv} the weak limits $F_{j,*}$ that are constant in the temporal coordinates and build an $L^2$-orthonormal system\footnote{By the Rellich--Kondrachov compact embedding theorem, the weak limits $F_{j,*}$ in $H^1$ are also strong limits in $L^2$, possibly by passing to further subsequences~\cite[Example~1.2.11]{badiale2010semilinear}. That is, $F_{j,a_i} \to F_{j,*}$ in $L^2$ as $i\to\infty$ for $j=1,\ldots,k$. Note that without loss we can take the same subsequence $(a_i)_{i\in\mathbb{N}}$ for all~$j=1,\ldots,k$. Since the $F_{j,a}$ are $L^2$-normalised, so are thus the~$F_{j,*}$. Orthogonality follows also from the $L^2$ convergence and $\langle F_{k,*}, F_{\ell,*} \rangle_{L^2}  = \lim_i \langle F_{k,a_i}, F_{\ell,a_i} \rangle_{L^2} = 0$.}.
We define $\mathbb{S}_{k,*} := \mathrm{span}\{F_{1,*},\ldots, F_{k,*} \} \subset H^1$ and $\mathbb{S}^D := \left\{ F(t,x)= f(x) \, : \, f\in H^1(M) \right\} \subset H^1$.

We obtain firstly by omitting nonnegative temporal contributions and secondly using orthonormality of the eigenfunctions in the $L^2$ norm that
\begin{equation}
	\label{eq:split_inner_prod}
	\begin{aligned}
		-\Lambda_{k,a} &= \max_{ \substack{F \in \mathrm{span} \{ F_{1,a},\ldots,F_{k,a} \} \\ \|F\|_{L^2}=1} } \iint_{\set{M}_0} a^2 (\partial_t F)^2  + \| \nabla_{g_t} F\|_{g_t}^2\,d\ell\,dt \\
		&\ge \max_{c\in\R^k,\, \|c\|_2=1} \iint_{\set{M}_0} \Big\| \nabla_{g_t} \sum_{j=1}^k c_j F_{j,a} \Big\|_{g_t}^2\,d\ell\,dt.
	\end{aligned} 
\end{equation}
By part 2, $a\mapsto \Lambda_{k,a}$ is monotone, and so we obtain for the entire sequence (not only along $a_i$)
\begin{equation}
\label{eq:barFbound}
\liminf_{i\to\infty} \Lambda_{k,a_i} = \lim_{a\to\infty} \Lambda_{k,a} \ge \lim_{a\to\infty} \Lambda^{\mathrm{spat}}_{k,a} \stackrel{\text{part 1}}{\ge } \lambda_k^D.	
\end{equation}
Note that for an arbitrary sequence of real-valued functions $(h_i)_i$ and a function $h$ satisfying $h \le \limsup_i h_i$ pointwise one has\footnote{To see this, for any $z$ and $\epsilon>0$ we denote by $z_\epsilon$ a point with $h(z_\epsilon) \ge \sup_z h(z) - \epsilon$. Then one has $\sup_z h(z) - \epsilon \le h(z_\epsilon) \le \limsup_i h_i(z_\epsilon) \le \limsup_i \sup_z h_i(z)$.}
\begin{equation}
    \label{eq:limsup_max_swap}
    \limsup_{i\to\infty} \sup_z h_i(z) \ge \sup_z h(z).
\end{equation}
Below we will use \eqref{eq:limsup_max_swap} with $h_i(c) = \iint_{\set{M}_0} \| \nabla_{g_t} \sum_j c_j F_{j,a_i} \|_{g_t}^2\,d\ell\,dt$ and \linebreak[5] $h(c) = \iint_{\set{M}_0} \| \nabla_{g_t} \sum_j c_j F_{j,*} \|_{g_t}^2\,d\ell\,dt$. These functions satisfy the condition $h \le \limsup_i h_i$  by $F_{j,a_i} \rightharpoonup F_{j,*}$ and lower weak semicontinuity of norms.
From \eqref{eq:split_inner_prod}, we obtain
\begin{align*}
	\limsup_{i\to\infty} -\Lambda_{k,a_i} &\ge \limsup_{i\to\infty}
		\max_{c\in\R^k,\, \|c\|_2=1} \iint_{\set{M}_0} \| \nabla_{g_t} \sum_j c_j F_{j,a_i} \|_{g_t}^2\,d\ell\,dt \\
		&\stackrel{\eqref{eq:limsup_max_swap}}{\ge} \max_{c\in\R^k,\, \|c\|_2=1} \iint_{\set{M}_0} \| \nabla_{g_t} \sum_j c_j F_{j,*} \|_{g_t}^2\,d\ell\,dt \\
		&= \max_{\substack{F\in \mathbb{S}_{k,*}\\ \|F\|_{L^2}=1}} \iint_{\set{M}_0} \| \nabla_{g_t} F\|_{g_t}^2\,d\ell\,dt \\
		&\ge \min_{\substack{\mathbb{S} \subset \mathbb{S}^D\\[2pt] \dim\mathbb{S}=k}} \ \max_{\substack{F\in \mathbb{S}\\ \|F\|_{L^2}=1}} \iint_{\set{M}_0} \| \nabla_{g_t} F \|_{g_t}^2\,d\ell\,dt = -\lambda_k^D,
\end{align*}
where the first equality follows from the orthonormality of the $F_{j,*}$ and the last inequality follows from $\set{S}_{k,*} \subset \set{S}^D$, since the $F_{j,*}$ are constant in their temporal coordinate.
With \eqref{eq:barFbound}, the claim follows.
\end{proof}

\section{Properties of the surrogate model}
\label{app:int_cohmix}

\subsection{Solution of the surrogate problem}
\label{app:proof_int_cohmix}

\begin{proof}[Proof of Proposition~\ref{prop:int_cohmix}]
By classical theory we have
\begin{lemma}
	The solution $u$ of the eigenvalue equation~\eqref{eq:surrogate2} satisfies~$u\in H^2_{\mathrm{loc}}(0,1) \subset C^1(0,1)$.
\end{lemma}
\begin{proof}
Since classical, we will only sketch the steps here. Consider the bilinear form $B(\cdot,\cdot)$ associated to the differential operator~$Lu(t) = u''(t) -  \rho(t)u(t)$ with homogeneous Neumann boundary conditions. By the Lax--Milgram theorem~\cite[Thm.~6.2.1]{Evans10} and Poincar\'e's inequality~\cite[Thm.~5.8.1]{Evans10} there is a unique, well-defined solution operator $S: L^2\to H^1 \cap \mathbf{1}^{\perp},\ f\mapsto u$, of $B(u,v) = \langle f,v \rangle$ $\forall v\in H^1$~$(\ast)$. By the Sobolev embedding theorem~\cite[Thm.~5.6.6]{Evans10} $S$ is compact and thus has countable spectrum. If $u \in H^1$ is an eigenfunction for some eigenvalue $\nu\neq 0$, then $u$ solves $(\ast)$ with $f = \nu u \in L^2$, and by the regularity results for elliptic equations~\cite[Thm.~6.3.1]{Evans10} we obtain~$u\in H^2_{\mathrm{loc}}$.
Again by the Sobolev embedding theorem~\cite[Thm.~5.6.6]{Evans10} in one dimension (implying that $H^2$ is continuously embedded in $C^1$) we obtain that $u$ is continuously differentiable on~$(0,1)$.
\end{proof}

With the ansatz $u(t) = \alpha e^{\omega t}$ in~\eqref{eq:surrogate2} we obtain
\[
\renewcommand{\arraystretch}{1.25}
\omega^2= \left\{
\begin{array}{ll}
\omega_z^2 = 2\frac{\nu+z}{a^2}, & t\in [0,p],\\
\omega_Z^2 = 2\frac{\nu+Z}{a^2},  & t\in [p,1],
\end{array}
\right.
\]
thus the general solution has the form
\[
u(t)= \left\{
\begin{array}{ll}
u_z(t) := \alpha_1 e^{\omega_z t} + \alpha_2 e^{-\omega_z t}, & t\in [0,p],\\
u_Z(t) := \alpha_3 e^{\omega_Z t} + \alpha_4 e^{-\omega_Z t},  & t\in [p,1].
\end{array}
\right.
\]
The Neumann boundary conditions translate into
\[
t=0:\ \alpha_1\omega_z - \alpha_2 \omega_z = 0,\qquad
t=1:\  \alpha_3\omega_Z e^{\omega_Z} - \alpha_4 \omega_Z e^{-\omega_Z} = 0,
\]
that is $\alpha_1 = \alpha_2$ and $\alpha_4 = \alpha_3 e^{2\omega_Z}$.
Now we require continuity at the interface $t=p$, i.e.,
\[
\alpha_1 \left( e^{\omega_z p} + e^{-\omega_z p} \right) = \alpha_3 \left( e^{\omega_Z p} + e^{2\omega_Z} e^{-\omega_Z p} \right),
\]
equivalently $\alpha_1 \cosh (\omega_z p) = \alpha_3 e^{\omega_Z}  \cosh(\omega_Z (1-p))$. Solving this for $\alpha_3$ and substituting into $u(t)$ (replacing the free constant $\alpha_1$ by a general constant $\alpha/2$) gives $u_z(t) = \alpha \cosh (\omega_z t)$ and
\[
\begin{aligned}
u_Z(t) &= \frac{\alpha}{2} \frac{\cosh(\omega_z p)}{\cosh (\omega_Z (1-p))} e^{-\omega_Z} \left( e^{\omega_Z t} + e^{2\omega_Z} e^{-\omega_Z t} \right) = \alpha \frac{\cosh(\omega_z p)}{\cosh (\omega_Z (1-p))} \cosh(\omega_Z (1-t)).
\end{aligned}
\]
To obtain the equation characterising the eigenvalues, we invoke the continuity of the derivative of $u$ at $t=p \in (0,1)$, i.e., $u_z'(p) = u_Z'(p)$, and obtain
\[
\omega_z \sinh (\omega_z p) = -\omega_Z \frac{\cosh(\omega_z p)}{\cosh (\omega_Z (1-p))} \sinh(\omega_Z (1-p)),
\]
which gives $\frac{\omega_z}{\omega_Z} \tanh (\omega_z p) = - \tanh(\omega_Z (1-p))$,
i.e.,~\eqref{eq:int_cohmix_eval}. This concludes the proof.
\end{proof}

\subsection{Eigenvalue analysis of the surrogate problem}
\label{app:int_cohmix_evals}

Without loss, we can take~$a>0$. We consider the situation where the mixing rate is given by the function~$\rho$ from~\eqref{eq:rho}, i.e., we are here in the setting of Proposition~\ref{prop:int_cohmix}. In particular, we will analyse the solutions of~\eqref{eq:int_cohmix_eval}, looking at the cases $\nu > -z$, $-z > \nu > -Z$, and~$-Z > \nu$ separately. Our findings are summarized in Proposition~\ref{prop:zeros} below. To avoid overly complicated formulas, we will suppress the dependence of $\omega_z$ and~$\omega_Z$ on $\nu$ in the following.

\paragraph{The case $\nu > -z$} If $\nu > -z$, we have $\omega_z, \omega_Z\in\R$. Further, by the structure of \eqref{eq:int_cohmix_eval} we can then assume $\omega_z, \omega_Z > 0$ without loss, and it follows that every term on the left-hand side of \eqref{eq:int_cohmix_eval} is positive, hence the equation can not have a solution. Thus, every solution of~\eqref{eq:int_cohmix_eval} satisfies~$\nu < -z$, as $\nu = -z$ can be ruled out by similar arguments.

\paragraph{The case $\nu \in (-Z,-z)$} In this case $\omega_z \in \mathrm{i}\R,\ \omega_Z \in \R$. As before, the signs can be chosen such that $\omega_z = \mathrm{i}|\omega_z|,\ \omega_Z >0$. The equation~\eqref{eq:int_cohmix_eval} reads in this case as
\[
f(\nu) := \tanh(\omega_Z (1-p)) - \frac{|\omega_z|}{\omega_Z} \tan(|\omega_z| p) = 0.
\]
With $|\omega_z| = \frac{\sqrt{-2(\nu+z)}}{a}$, $\omega_Z = \frac{\sqrt{2(\nu + Z)}}{a}$, we observe that
\begin{itemize}
\item
$\nu \mapsto \tanh \Big( \frac{\sqrt{2(\nu+Z})}{a}(1-p) \Big)$ is increasing;
\item
$\nu \mapsto \frac{|\omega_z|}{\omega_Z} = \sqrt{\frac{-\nu-z}{\nu+Z}} = \sqrt{\frac{Z-z}{\nu+Z} -1} $ is decreasing; and
\item
$\nu \mapsto \tan\Big(\frac{\sqrt{-2(\nu+z)}}{a} p \Big)$ is decreasing between its singularities.
\end{itemize}
In summary, between its singularities, the function $f$ is continuous with values increasing from $-\infty$ to $\infty$; see Figure~\ref{fig:surrogate_example}(b). The singularities $\nu_k^*$ can be characterised by
\begin{equation} \label{eq:int_singularities}
\frac{\sqrt{-2(\nu_k^*+z)}}{a} p = \frac{\pi}{2} + k\pi,\quad k\in\mathbb{Z} \quad\Longleftrightarrow \quad \nu_k^* = -z - \frac{a^2}{2} \Big(\frac{2k+1}{2p}\pi \Big)^2,\quad k\in\mathbb{Z},
\end{equation}
thus there is exactly one eigenvalue in every open interval defined by two adjacent singularities inside~$(-Z,-z)$. The rightmost interval is not bounded by a singularity on the right, but by $-z$, as $f(-z)>0$ guarantees the existence of a zero of $f$ larger than~$\nu_0^*$.

What happens if $p$ is so small that the first singularity of the kind as in~\eqref{eq:int_singularities} satisfies~$\nu_0^* < -Z$? We note that then $\tan(|\omega_z| p) > 0$ for $\nu \in (-Z,-z)$, since it is a decreasing function in $\nu$ (see above) and it admits the value 0 at $\nu = -z$, while having its first singularity that is left of $-z$, i.e.\ $\nu_0^*$, smaller than~$-Z$. For $\nu \to -Z^+$ we have $|\omega_z| / \omega_Z \to +\infty$ and $\tanh(\omega_Z(1-p)) \to 0$, thus $\lim_{\nu\to -Z^+} f(\nu) = -\infty$. As $f$ is continuous and monotonically increasing on $(-Z,-z)$, we obtain that it has exactly one zero in $(-Z,-z)$ if~$\nu_0^* < -Z$.

For $p\to 0$ and fixed $\nu\in (-Z,-z)$ we have
\[
\tanh(\omega_Z (1-p)) \to \text{const}>0, \quad \frac{|\omega_z|}{\omega_Z} \tan(|\omega_z| p)\to 0.
\]
Thus, for any such fixed $\nu$ one has that $f(\nu)>0$ if $p$ sufficiently small, implying that
\begin{equation}
\label{eq:lam_p_to_0}
\text{the zero of $f$ in $(-Z,-z)$ converges to $-Z$ as~$p\to 0$.}
\end{equation}

For $a\to 0$ we have that the largest zero of $f$ converges agains $-z$, since by the above there is a zero in the interval~$(\nu_0^*,-z)$ and $\lim_{a\to 0}\nu_0^* = -z$ by~\eqref{eq:int_singularities}. To summarize, we have shown:
\begin{proposition} \label{prop:zeros}
The dominant (largest) eigenvalue~$\nu_0$ of the surrogate problem \linebreak[4] $\frac{a^2}{2} u'' - \rho u = \nu\, u$ on $(0,1)$ with homogeneous Neumann boundary conditions and \linebreak[4] $\rho(t) = z \mathbf{1}_{[0,p]}(t) + Z \mathbf{1}_{(p,1]}(t)$ satisfies:
	\begin{enumerate}[(a)]
		\item $\nu_0 \in (-Z,-z)$
		\item $\lim_{p\to 0} \nu_0 = -Z$
		\item $\lim_{a\to 0} \nu_0 = -z$
	\end{enumerate}
Additionally, by similar arguments to those in Theorem~\ref{thm:eigbounds}, we expect
\[
\lim_{a\to\infty} \nu_0 = -\int_0^1 \rho(t)dt = - pz - (1-p)Z.
\]
\end{proposition}

\paragraph{The case $\nu < -Z$} In this case
\[
f(\nu) = \frac{|\omega_z|}{|\omega_Z|} \mathrm{i} \tan(|\omega_z| p) + \mathrm{i} \tan(|\omega_Z| (1-p)),
\]
and its zeros are sandwiched between singularities of the two trigonometric tangent functions. The situation is depicted in Figure~\ref{fig:surrogate_example}(b) for
$
z = 2,\quad
Z = 40,\quad
p = 0.25,\quad
\frac{a^2}{2} = \frac{1}{\pi^2}.
$
The red crosses indicate the eigenvalues of \eqref{eq:surrogate2} computed by a finite difference scheme on a uniform grid of 1000 nodes.

\small
\bibliographystyle{myalpha}
\bibliography{references}

\end{document}